\documentclass[a4paper, 11pt]{article}

\usepackage{calc, amsmath, amsthm, a4, latexsym, amssymb, layout, subfigure, bm, xspace, url, bbm} 
\usepackage[normalem]{ulem} 
\usepackage[dvips]{graphicx}
\usepackage{mathtools}
\usepackage{stmaryrd}
\usepackage[usenames,dvipsnames,svgnames,table]{xcolor}

\mathtoolsset{showonlyrefs}

\newtheorem{thm}{Theorem}[section]
\newtheorem{theorem}[thm]{Theorem}
\newtheorem{lemma}[thm]{Lemma}
\newtheorem{corollary}[thm]{Corollary}
\newtheorem{prop}[thm]{Proposition}

\theoremstyle{definition}

\newtheorem{rem}[thm]{Remark}
\newtheorem{remark}[thm]{Remark}

\newcommand{\be}[1]{\begin{equation}\label{#1}}
\newcommand{\ee}{\end{equation}}
\newcommand{\ba}{\begin{array}}
\newcommand{\ea}{\end{array}}
\newcommand{\bal}{\begin{aligned}}
\newcommand{\eal}{\end{aligned}}

\newcommand{\R}{\mathbb{R}}
\newcommand{\C}{\mathbb{C}}
\newcommand{\N}{\mathbb{N}}
\newcommand{\Q}{\mathbb{Q}}
\newcommand{\Z}{\mathbb{Z}}
\newcommand{\E}{\mathbb{E}}
\newcommand{\p}{\mathbb{P}}

\newcommand{\calC}{\mathcal{C}}

\newcommand{\calM}{\mathcal{M}}

\newcommand{\calS}{\mathcal{S}}

\newcommand{\calU}{\mathcal{U}}

\newcommand{\cA}{\mathcal{A}}

\newcommand{\cC}{\mathcal{C}}

\newcommand{\cH}{\mathcal{H}}
\newcommand{\cI}{\mathcal{I}}
\newcommand{\cJ}{\mathcal{J}}

\newcommand{\cM}{\mathcal{M}}

\newcommand{\cS}{\mathcal{S}}

\newcommand{\cU}{\mathcal{U}}

\newcommand{\1}{1\hspace{-0.098cm}\mathrm{l}} 

\newcommand{\la}{\lambda}

\newcommand{\eps}{\varepsilon}
\newcommand{\om}{\omega}

\newcommand{\ra}{\rightarrow}

\renewcommand{\d}{\mathrm{d}}

\newcommand{\ssup}[1] {{\scriptscriptstyle{({#1}})}}
\newcommand{\sse}[1] {{\scriptscriptstyle{[{#1}}]}} 

 \newcommand{\supp}{{\rm supp}}

\newcommand{\tem}{\mathrm{tem}}

\renewcommand{\rho}{\varrho}

\newcommand{\bE}{\ensuremath{\mathbb{E}}}

\newcommand{\bN}{\ensuremath{\mathbb{N}}}

\newcommand{\bP}{\ensuremath{\mathbb{P}}}
\newcommand{\bQ}{\ensuremath{\mathbb{Q}}}
\newcommand{\bR}{\ensuremath{\mathbb{R}}}

\newcommand{\bZ}{\ensuremath{\mathbb{Z}}}

\newcommand{\ind}{\ensuremath{\mathbbm{1}}}

\newcommand{\abs}[1]{\left\vert \, #1 \, \right\vert}
\newcommand{\norm}[1]{\left\Vert \, #1 \, \right\Vert}

\newcommand{\ddx}[1][1]{\ifnum#1=1 \frac{d}{dx} \else \frac{d^{#1}}{dx^{#1}} \fi}
\newcommand{\ddy}[1][1]{\ifnum#1=1 \frac{d}{dy} \else \frac{d^{#1}}{dy^{#1}} \fi}
\newcommand{\ddt}[1][1]{\ifnum#1=1 \frac{d}{dt} \else \frac{d^{#1}}{dt^{#1}} \fi}

\newcommand{\bfx}{{\mathbf x}}
\newcommand{\bfy}{{\mathbf y}}
\newcommand{\bfz}{{\mathbf z}}
\newcommand{\bfu}{{\mathbf u}}
\newcommand{\bfv}{{\mathbf v}}
\newcommand{\bfX}{{X}}
\newcommand{\bfY}{{\mathbf Y}}
\newcommand{\bfI}{{\mathbf I}}

\begin{document}

\begin{center}
{\Large \bf A new look at duality for the  symbiotic branching model}
\\[5mm]
\vspace{0.7cm}
\textsc{Matthias Hammer\footnote{Institut f\"ur Mathematik, Technische Universit\"at Berlin, Stra\ss e des 17. Juni 136, 10623 Berlin, Germany.},
Marcel Ortgiese\footnote{
Department of Mathematical Sciences, University of Bath, Claverton Down, Bath, BA2 7AY,
United Kingdom.} 
and Florian V\"ollering\footnotemark[2]
} 
\\[0.8cm]
\fboxsep2mm
{\small \today} 
\end{center}

\vspace{0.3cm}

\begin{abstract}\noindent 
The symbiotic branching model is a spatial 
population model describing the dynamics of two interacting types
that can only branch if both types are present. 
A classical result for the underlying stochastic partial differential 
equation identifies moments of the solution via a duality to 
a system of Brownian motions with dynamically changing colors.
In this paper, we revisit this duality and give it a new interpretation. 
This new approach allows us to extend the duality to the limit as the branching rate 
$\gamma$
is sent to infinity. This limit is particularly interesting since it captures
the large scale behaviour of the system. 
As an application of the duality, we can explicitly identify the $\gamma = \infty$
limit  when the driving noises are perfectly negatively correlated. 
The limit is  a system of annihilating Brownian motions
with a drift that depends on the initial imbalance between the types.
  \par\medskip

  \end{abstract}


\section{Introduction}\label{intro}


In \cite{EF04} Etheridge and Fleischmann introduce a spatial population model that 
describes the evolution of two interacting types. 
The dynamics follows locally a branching process,
where each type branches with a rate proportional to the frequency
of the other type. Additionally, types are allowed to migrate to neighbouring 
colonies. In the 
continuum space and large population limit, the symbiotic branching model is the process $(\bfu_t)_{t\ge0}=(u^\ssup1_t, u^\ssup2_t)_{t\ge0}$, where the frequencies $u^\ssup{1}_t(x)$ and $u^\ssup{2}_t(x)$
of the respective types are given by the 
nonnegative solutions  of
the stochastic partial differential equations
\begin{align}\bal
\label{eqn:spde}
 	{\mathrm{cSBM}(\varrho,\gamma)}_{ \bfu_0}: \quad
 		\begin{cases}
  			\frac{\partial }{\partial t}
u_t^\ssup{1}(x) & \!\!\!\!= \frac{\Delta}{2} u^\ssup{1}_t(x) + 
                     	\sqrt{ \gamma u^\ssup{1}_t(x) u^\ssup{2}_t(x)} \, \dot{W}^\ssup{1}_t(x),\\[0.3cm]
 			\frac{\partial }{\partial t}u_t^\ssup{2}(x) 
 & \!\!\!\!= \frac{\Delta}{2} u_t^\ssup{2}(x) + 
                     	\sqrt{ \gamma u^\ssup{1}_t(x) u^\ssup{2}_t(x)} \, \dot{W}^\ssup{2}_t(x),
 		\end{cases}
	\eal\end{align}
with suitable nonnegative initial condition $\bfu_0=(u^\ssup1_0, u^\ssup2_0)$, $u^\ssup{i}_0(x)\ge 0$, $x \in \R, i=1,2$.
Here, $\gamma > 0$ is the branching rate and $ (\dot{W}^\ssup{1},
\dot{W}^\ssup{2})$ is a pair of correlated standard Gaussian white noises on $\mathbb{R}_+ \times \mathbb{R}$ with correlation 
governed by a parameter $\varrho
\in [-1,1]$. 

There is also a discrete-space version of the model on the lattice $\Z^d$, 
where $\Delta$ is the discrete Laplacian and the white noises are replaced by an independent system $(W^1(x), W^2(x))$, $x\in\Z^d$, of
$\varrho$-correlated two-dimensional Brownian motions. We will refer to the discrete space model with initial condition $\bfu_0$ as ${\mathrm{dSBM}(\varrho,\gamma)}_{\bfu_0}$.
Existence (for $\rho \in [-1,1]$) and uniqueness (for $\rho \in [-1,1)$) in both continuous and discrete space was
proved in~\cite{EF04} for a large class of initial conditions.

The symbiotic branching model generalizes several well-known examples of spatial populations dynamics.
In particular, for the case $\varrho=-1$ of perfectly negatively correlated noises, the sum $u_t+v_t$ solves the (deterministic) heat equation, so we have $u^\ssup{1}_t+u^\ssup{2}_t=S_t(u_0+v_0)$ for all $t \geq 0$, with $(S_t)_{t\ge0}$ denoting the heat semigroup. In particular, for initial conditions $u_0 = 1- v_0$ summing up to one, we have $u_t:=u^\ssup{1}_t = 1- u^\ssup{2}_t$ for all $t \geq 0$, and
the system reduces to the continuous-space {\em stepping stone model}
\begin{equation}\label{eqn:stepping stone}
  			\tfrac{\partial }{\partial t}
u_t(x)  = \tfrac{\Delta}{2} u_t(x) + 
                     	\sqrt{ \gamma\, u_t(x) \left(1-u_t(x)\right)} \, \dot{W}_t(x)
\end{equation}
analysed e.g.\ in \cite{Shiga86} and \cite{T95}.
For $\varrho=0$, 
the system is known as the {\em mutually catalytic model} due to Dawson and Perkins \cite{DP98}. 

We are particularly interested in the case that initially both types are  spatially separated. In the simplest case, this corresponds to starting the system with
`complementary Heaviside initial conditions' $\bfu_0=(u_0^\ssup{1},u_0^\ssup{2})=(\1_{\R^-},\1_{\R^+})$.
Then one would like to understand the evolution of the \emph{interface} between the two types, which is defined by	
		\begin{eqnarray*}
			{\rm Ifc}_t = \mbox{cl} \big\{x\in\R : u^\ssup{1}_t(x) u^\ssup{2}_t(x) > 0 \big\},
		\end{eqnarray*}
where $\bfu_t=(u_t^\ssup{1},u_t^\ssup{2})$ is the solution of $\mathrm{cSBM}(\varrho,\gamma)_{\bfu_0}$ at time $t>0$, and $\mbox{cl}(A)$ denotes the closure of the set $A$ in $\R$.
Results on the growth of this interface due to \cite{EF04}, \cite{BDE11} seem to suggest diffusive behaviour for the interface. 
This conjecture is supported by
the following \emph{scaling property} of the model, see  \cite[Lemma~8]{EF04}: 
Let $(\bfu_t)_{t\ge0}$ denote the solution to $\mathrm{cSBM}(\rho, \gamma)_{\bfu_0}$. 
If we rescale time and space diffusively, i.e.\
if given $K > 0$ we define 
\[\bfv^\sse{K}_t(x) := \bfu_{K^2t}(K x),\qquad x \in \R,\, t \geq 0,\] 
then $(\bfv^\sse{K}_t)_{t\ge0}$ is a solution to $\mathrm{cSBM}(\rho, K \gamma)_{\bfv^\sse{K}_0}$, i.e.\ a symbiotic branching process with branching rate $K\gamma$ and correspondingly transformed initial  states $\bfv^\sse{K}_0(x)=\bfu_{0}(Kx)$, $x\in\R$. 
Thus provided that the initial conditions are invariant under diffusive rescaling (as is the case for complementary Heaviside), 
this rescaling of the system is equivalent (in law) to increasing the branching rate.
This observation suggests to  
investigate an infinite rate limit $\gamma \ra \infty$,
which can also be considered for more general initial conditions.

This program has been carried out first for the discrete-space model $\mathrm{dSBM}(\rho,\gamma)$. 
Also inspired by the scaling property
of the continuous model, 
\cite{KM10a, KM11b, KM11c,DM11a, DM12} construct a
non-trivial limiting process for $\mathrm{dSBM}(\rho,\gamma)$ as $\gamma \to \infty$ and study its
long-term properties. They also give a very explicit description of the limit 
in terms of an infinite system of jump-type stochastic differential equations (SDEs). The limit corresponds to a system where the two types remain separated, i.e. at each lattice point only one type is present. 
In our recent preprint \cite{HO15} we also consider initial conditions where the types are not separated initially, 
and for $\rho \in (-1,0)$ we show existence of an infinite rate limit and separation of types for positive times. 
We will refer to the limit as
the \emph{discrete-space infinite rate symbiotic branching model}, abbreviated as $\mathrm{dSBM}(\rho,\infty)$. 

For the continuous space model $\mathrm{cSBM}(\rho,\gamma)$, an infinite rate limit has been shown to exist recently in~\cite{BHO15} for the case of negative correlations $\rho\in(-1,0)$ and a large class of initial conditions. More precisely, it was proved that the measure-valued processes
obtained by taking the solutions of $\mathrm{cSBM}(\rho,\gamma)_{\bfu_0}$ as densities
converge in law as $\gamma\uparrow\infty$ to
a measure-valued process $(\mu_t^\ssup{1},\mu_t^\ssup{2})_{t \geq 0}$ also satisfying a certain separation-of-types condition, see \cite[Thm.\ 1.10]{BHO15}. 
We call the limit $(\mu_t^\ssup{1},\mu_t^\ssup{2})_{t \geq 0}$ the \emph{continuous-space
infinite rate symbiotic branching model} $\mathrm{cSBM}(\rho,\infty)$.
As for the discrete case, the convergence is generally in the Meyer-Zheng `pseudo-path' topology 
on $D_{[0,\infty)}$, which is strictly weaker than the standard Skorokhod topology. 
Under the more restrictive condition that $\bfu_0 = (\1_{\R^-},\1_{\R^+})$ and $\rho \in (-1,-\frac{1}{\sqrt{2}})$, 
we showed convergence in the stronger Skorokhod topology on $\calC_{[0,\infty)}$, see ~\cite[Thms.\ 1.5, 1.12]{BHO15}. Also, we showed that 
in this case the limiting measures $\mu_t^\ssup{1}$ and $\mu_t^\ssup{2}$ are absolutely continuous with respect to Lebesgue measure
and mutually singular, i.e.\ we have 
\begin{equation}\label{singularity1}
\mu_t^\ssup{1}(\cdot) \mu_t^\ssup{2}(\cdot)\equiv0\qquad \text{almost surely}.
\end{equation}
In \cite{BHO15}, we characterized  $\mathrm{cSBM}(\rho,\infty)$ in terms of a rather abstract martingale problem. 
In contrast to the discrete-space case (see \cite{KM11b}), a fully explicit characterization of the continuous-space limit is still missing so far. 
Recently, in \cite{HO15} the gap between the two infinite rate models has been narrowed somewhat by showing that for all $\rho\in(-1,0)$, $\mathrm{dSBM}(\rho,\infty)$ converges to $\mathrm{cSBM}(\rho,\infty)$ under diffusive rescaling.
This opens up the possibility to obtain results on the continuous model from analogous ones for the discrete model. 
For example, it was shown in \cite{HO15} that for complementary Heaviside initial conditions, a single interface exists almost surely, for each fixed time $t>0$.

In \cite{BHO15} we did not consider $\rho=-1$.
In this case, for complementary Heaviside initial conditions,
a diffusive scaling limit was already proved in~\cite{T95} for the continuum stepping stone model, 
as one of the steps of understanding the diffusively rescaled interface. 
Under this assumption it was shown that the measure-valued limit $(\mu_t^\ssup{1}, \mu_t^\ssup{2})_{t \geq 0}$ 
has the law of
\begin{equation}\label{eq:tribe} (\1_{\{x\le B_t\}}\, dx,\1_{\{x\ge B_t\}}\, dx)_{t \geq 0} , \end{equation}
for $(B_t)_{t \geq 0}$ a standard Brownian motion.
More generally, instead of just complementary Heaviside \cite{T95} considered also initial conditions with `multiple interfaces', but still under the assumption that they sum up to one so that the system reduces to the stepping stone model \eqref{eqn:stepping stone}. 
It was shown that these solutions with a finite number of interfaces converge to 
a system of annihilating Brownian motions.

The original motivation for this work was the question what happens for $\rho=-1$ and \emph{general} initial conditions. 
This includes the case where $\bfu_0$ is still complementary but does not satisfy $u_0^\ssup{1}+u^\ssup{2}_0\equiv1$, but also `overlapping' initial configurations where the two populations are no longer well-separated.
These scenarios correspond to a stepping stone model where we do not look at relative frequencies of types
and where we would like to understand how an initial imbalance in absolute numbers propagates.
These general initial conditions are not covered by the results in~\cite{T95}, and the problem is also open for the discrete model, see e.g.\ \cite{DM12}, p. 43. 

In fact, all tightness results obtained in \cite{BHO15} continue to hold for $\rho=-1$ as well. 
The problem is however uniqueness, since the self-duality approach employed for the case $\rho>-1$ in \cite{EF04}, \cite{KM11b} and \cite{BHO15} breaks down. The same  
problem arises for the finite rate model as well. Instead of using self-duality,  \cite{EF04} 
establish uniqueness for  $\mathrm{SBM}(-1,\gamma)$ using a new moment duality introduced there, which we recall below and
which replaces Shiga's \cite{Shiga86} moment duality involving coalescing Brownian motions.
 For $\rho=-1$ the process is characterized by its moments precisely because the sum $u_t^\ssup1+u_t^\ssup2$ still solves the deterministic heat equation.
This suggests a strategy to prove uniqueness for $\mathrm{SBM}(-1,\infty)$ by extending the moment duality of \cite{EF04} to the infinite rate limit. Indeed, this (rather non-trivial) extension is one of the main results in the present work. 

We now briefly recall the moment duality due to \cite{EF04}.
Let $\bfu_t = (u_t^\ssup{1},u_t^\ssup{2})$ denote the solution of $\mathrm{SBM}(\rho,\gamma)$ (in discrete or continuous space) with initial condition $\bfu_0=(u_0^\ssup{1},u_0^\ssup{2})$.
Let $\calS\in\{\Z^d,\R\}$. 
Fix $n\in\bN$. For $\bfx=(x_1,\ldots,x_{n})\in\calS^n$ and $c=(c_1,\ldots,c_n)\in\{1,2\}^n$,
one is interested in the (mixed) moments  
\[\E_{u_0^\ssup{1},u_0^\ssup{2}}\left[\Pi_{i=1}^n u_t^\ssup{c_i}(x_i)\right],\qquad t>0.\]
We can interpret the two types as `colors', thus we will call each element of $\{1,2\}^n$ a \emph{coloring}. 
In discrete resp.\ continuous 
space, the dual process is defined as follows:
At time $t=0$, we start with $n$ particles located at $\bfx\in\calS^n$ and colored according to $c\in\{1,2\}^n$, i.e. $c_i$ is the color of particle $i$ located at $x_i$. 
The particles move on paths given by a family of independent simple symmetric random walks
resp.\ Brownian motions $(X_t)_{t\ge0}$ starting at $\bfx$. 
When two particles meet, they start collecting collision local time.
If both particles are of the same color, each of them changes color when their collision local time 
exceeds an (independent) exponential time with parameter $\gamma/2$, resulting in a total rate of $\gamma$ per pair of particles for a change of color. 
Denote by $(C_t)_{t\ge0}$ the resulting coloring process of $X$, that is $C_t \in \{1,2\}^n$ prescribes the color of the $n$ particles at time $t$. 
Finally, denote by $L_t^=$ the total collision local time collected by all pairs of the same color up to time $t$, and let $L_t^{\neq}$ be the  collected local time of all pairs of different color up to time $t$.
For a more rigorous definition of the dual process and more details, we refer to \cite{EF04}, Sections 3.2 and 4.1.
The mixed moment duality function is given (up to an exponential correction involving $L_t^=$ and $L_t^{\neq}$) by
\begin{align}\label{defn:H}
H(\bfu; \bfx,c) := \bfu^{(c)}(\bfx) := \prod_{i=1}^n u^{(c_i)}(x_i).
\end{align}
Now the moment duality for $\mathrm{SBM}(\rho,\gamma)$ reads as follows \cite[Prop. 9, 12]{EF04}:
\begin{align}\label{moment duality finite gamma}
\bE_{\bfu_0}\left[ H(\bfu_t;\bfx,c)\right] 
= \bE_{\bfx,c}\left[H(\bfu_0;X_t,C_t)\, e^{\gamma(L^=_t+\rho L^{\neq}_t)}\right]. 
\end{align}
Although the moment duality is particularly important
for $\rho=-1$ since there it is needed to ensure uniqueness, it holds for all values of $\rho\in[-1,1]$. In \cite{BDE11}, it was used 
(together with the self-duality) to investigate how the long-term behavior of the moments depends on the parameter $\rho$. 
More precisely, define the \emph{critical curve} $p: [-1,1) \to (1, \infty]$ of the symbiotic branching model by 
	\begin{align}\label{criticalcurve}
 		p(\varrho)=\frac{\pi}{ \arccos(-\varrho)} ,
 	\end{align}
and denote its inverse by $\varrho(p) = - \cos(\frac{\pi}{p})$ for $p\in(1,\infty]$. 
As shown in \cite[Thm.\ 2.5]{BDE11},
this critical curve determines the long-term behavior of the moments of $\mathrm{SBM}(\rho,\gamma)$ for uniform initial conditions: We have
\begin{equation}\label{moment_behavior}
\varrho<\varrho(p) \quad \Rightarrow \quad  \E_{\1,\1}\big[u^\ssup{i}_t(x)^p\big] \mbox{ is bounded uniformly in all }t\ge 0,\, x\in\cS,\, i=1,2,
\end{equation}
and the above condition is sharp in the recurrent case, i.e.\ for $\cS\in\{\R,\Z,\Z^2\}$.
This suggests that the infinite rate model $\mathrm{SBM}(\rho, \infty)$ has finite $p$-th moments for $p<p(\rho)$, and indeed this holds true (see \cite[Prop.\ 2.8]{HO15}).

In this work, we extend the moment duality from \cite{EF04} to the infinite rate limit $\mathrm{SBM}(\rho,\infty)$ for all $n\in\N$ such that $\rho + \cos(\pi/n) < 0$, corresponding to integer moments below the critical curve.
As we will explain below, this is not straightforward and first requires a suitable reinterpretation of the moment duality \eqref{moment duality finite gamma}. 
Indeed, it is intuitive that for $\gamma\uparrow\infty$, the color change mechanism in the dual system of colored particles described above will happen instantaneously upon the collision of two particles. 
However, it is not at all clear what the exponential correction term in \eqref{moment duality finite gamma} will converge to.
For $\rho=-1$, the extension to $\gamma=\infty$ establishes the uniqueness of $\mathrm{SBM}(-1,\infty)$ for general initial conditions, which was so far open in both discrete and continuous space. 
As a further application, we investigate the continuous space model in more detail. 
In fact, assuming only boundedness of the initial condition $\bfu_0$, we provide a complete description of $\mathrm{cSBM}(-1,\infty)_{\bfu_0}$ in terms of annihilating Brownian motions with drift.

\section{Main results}

In this section, we give a precise statement of our main results. 
From now on, we will write $(\bfu_t^\sse{\gamma})_{t\ge0}$ for the solution to $\mathrm{SBM}(\rho,\gamma)_{\bfu_0}$ (in discrete or continuous space)
with initial condition $\bfu_0$ and finite branching rate $\gamma\in(0,\infty)$. 
The notation $(\bfu_t)_{t\ge0}$ will usually\footnote{An exception is Cor.\ \ref{cor:explicit-moment-computations} (ii).} be reserved for the infinite rate limit $\mathrm{SBM}(\rho,\infty)_{\bfu_0}$.
We recall from \cite[Thm.\ 1.10]{BHO15} resp.\ \cite[Thm.\ 2.2]{HO15} that the limiting process $(\bfu_t)_{t\ge0}$ takes values in $D_{[0,\infty)}(\cM_\tem(\cS)^2)$, where $\cM_\tem(\cS)$ denotes the space of tempered measures on $\cS\in\{\Z^d,\R\}$, see e.g.\ Appendix A.1 in \cite{BHO15} for a precise definition.
There we also recall the definition of the Meyer-Zheng topology that we use throughout. 

In order to simplify the proofs, in this paper we will consider only bounded initial densities $\bfu_0=(u_0^\ssup{1},u_0^\ssup{2})$. 
However, at the expense of more technicalities, everything should work also with tempered initial densities as used in \cite{BHO15} and \cite{HO15}.

\subsection{Moment duality for $\mathrm{SBM}(\rho,\infty)$}

We return to the classical moment duality \eqref{moment duality finite gamma} for $\mathrm{SBM}(\rho,\gamma)$ due to \cite{EF04},
which we would like to extend to $\gamma=\infty$. 
Recall that we call the elements of $\{1,2\}^n$ \emph{colorings}. Let $\calM(\{1,2\}^n)$ denote the space of measures on colorings, which can be identified with $(\R^+)^{(2^n)}$.
Given the paths $\bfX=(X^\ssup{1}_t,\ldots,X^\ssup{n}_t)_{t\geq0}$ of either simple symmetric random walk in $(\Z^d)^n$ or Brownian motion in $\R^n$ and a measure on colorings $M_0\in\cM\left(\{1,2\}^n\right)$,
let $M_t^\sse{\gamma}\equiv M_t^\sse{\gamma}(\bfX,M_0)$ be a process taking values in $\calM(\{1,2\}^n)$ with initial state $M_0$
and evolution given by
\begin{align}\label{eq:ODE_X_new}
 dM^\sse{\gamma}_t(b) 
&= \frac{\gamma\rho}{2} \sum_{ i,j =1 }^n \ind_{b_i\neq b_j} M^\sse{\gamma}_t(b)\,dL_t^{i,j} + \frac{\gamma}{2}
 \sum_{i,j=1}^n \ind_{b_i\neq b_j} M^\sse{\gamma}_t(\widehat{b}^i)\,dL_t^{i,j},\qquad b\in\{1,2\}^n.
\end{align}
Here $(L^{i,j}_t)_{t\ge0}$ are the pair local times of the motions $\bfX$ and $\widehat{b}^i$ is the coloring $b$ flipped at $i$. The process $M_t^\sse{\gamma}(\bfX,M_0)$ is well-defined almost surely w.r.t.\ the law of the random walk or Brownian motion.
\begin{thm}\label{thm:rewrite}
Fix $\rho\in[-1,1]$, and let $\calS\in\{\Z^d,\R\}$.
Given $n\in\N$ and a coloring $c\in\{1,2\}^n$, the original moment duality \eqref{moment duality finite gamma} can be rewritten as
\begin{align}\label{eq:moment duality finite gamma 2}
\bE_{\bfu_0}\big[ H(\bfu^\sse{\gamma}_t;\bfx,c)\big] = \bE_{\bfx}\bigg[\sum_{b\in\{1,2\}^n} M_t^\sse{\gamma}(\bfX,\delta_c)(b)\, \bfu_0^\ssup{b}(X_t) \bigg].
\end{align}
\end{thm}
\begin{rem}
Let us remark that we have a lot of control over the process $M^\sse{\gamma}$, in terms of eigenvectors and eigenvalues. The details of this are postponed to Section~\ref{sec:duality}. 
\end{rem}
Our first result, which is crucial for the extension of the moment duality to $\gamma=\infty$, is that 
the process $(M_t^\sse{\gamma}(\bfX,M_0))_{t\ge0}$ converges almost surely as $\gamma\uparrow\infty$, provided $\rho + \cos(\pi/n) < 0$.

\begin{thm}\label{thm:M_t-infty}
Suppose that $\rho + \cos(\pi/n) < 0$.  In the continuous-space case $\calS=\R$, assume also that the starting point $\bfx=(x_1,\ldots,x_n)\in\R^n$ of the Brownian motion $\bfX$ is such that 
no two coordinates are the same.
Then as $\gamma \ra \infty$,
the process $(M_t^\sse\gamma)_{t \geq 0}$ defined in \eqref{eq:ODE_X_new} converges almost surely 
to a l\`adc\`ag limiting process $(M_t^\sse{\infty})_{t\ge0}=(M_t^\sse{\infty}(\bfX,M_0))_{t\ge0}$. 
\end{thm}
\begin{rem}
\begin{itemize}
 \item[a)]
The construction of the process $M^\sse{\infty}$ in Theorem \ref{thm:M_t-infty} is explicit, but also involved. 
Therefore we postpone its description to Section~\ref{sec:duality}, see in particular Propositions~\ref{prop:thm_M_infty_discrete} and~\ref{prop:thm_M_infty_continuous}.
 \item[b)]
In continuous space, the extra condition on the starting point of the Brownian motions that $x_i\ne x_j$ for $i\ne j$ is only technical and may be removed, albeit at the expense of a much more involved proof. 
But since we will need the convergence only for Lebesgue-almost all starting points $\bfx\in\R^n$, we will not prove this.
\end{itemize}
\end{rem}

Our second main result extends the moment duality \eqref{eq:moment duality finite gamma 2} to the infinite rate limit. 
Note that we exclude $\rho=-1$ since in this case the infinite rate limit has not yet been constructed for general initial conditions. 
In fact, in Section \ref{ssec:-1} we will use Theorems \ref{thm:rewrite} and \ref{thm:M_t-infty} to remedy this and obtain much more details about the case $\rho=-1$.
Since in continuous space the limiting process $\mathrm{cSBM}(\rho,\infty)$ is measure-valued, instead of the `pointwise' duality function employed in \eqref{eq:moment duality finite gamma 2} 
we state the duality in a weak formulation.

\begin{thm}\label{thm:moment_duality_infinite} 
Suppose $\rho\in(-1,0)$.  For $\calS\in\{\Z^d,\R\}$, consider nonnegative and bounded initial densities $\bfu_0 = (u_0^\ssup{1},u_0^\ssup{2})$. 
Let $(\mathbf{u}_t)_{t\ge0}\in D_{[0,\infty)}(\cM_\tem(\cS)^2)$ denote the solution to $\mathrm{SBM}(\rho,\infty)_{\bfu_0}$, 
and let $(M_t^\sse{\infty})_{t\ge0}$ be the limiting process from Theorem~\ref{thm:M_t-infty}. 
Then $(\mathbf{u}_t)_{t\ge0}$ satisfies the following moment duality: For all $n\in\N$ such that $\rho + \cos(\pi/n) < 0$, all $t>0$, all colorings $c=(c_1,\ldots,c_n)\in\{1,2\}^n$ and all nonnegative test functions 
$0\le\phi\in L^1(\cS^n)$,
we have
\begin{equation}\label{eq:moment duality infinite gamma}  
\E_{\bfu_0} \left[\int_{\cS^n} \phi(\bfx)\,\bfu_t^\ssup{c}(d\bfx)\right]  = \int_{\calS^n}\phi(\bfx)\, \E_{\bfx}\left[\sum_{b\in\{1,2\}^n}M_t^\sse{\infty}(\bfX,\delta_c)(b)\,\bfu_0^\ssup{b}(X_t)\right]\,d\bfx < \infty,
\end{equation}
where we write $\bfu_t^\ssup{c}(d\bfx):=\bigotimes_{i=1}^n u_t^\ssup{c_i}(dx_i)$ and the integral on the right hand side
is taken w.r.t.\ Lebesgue measue if $\calS=\R$ and w.r.t.\ counting measure if $\calS=\Z^d$.
\end{thm}

We will now state several  consequences of the new moment duality \eqref{eq:moment duality infinite gamma} resp.\ of the reformulation of the finite rate duality \eqref{eq:moment duality finite gamma 2}.
First of all, for the infinite rate model we can explicitly compute second mixed moments, where we recover a known identity (see \cite[Thm.\ 1.2]{DM11a} for discrete space and \cite{HO15} for the general case).
Secondly, for the finite rate model on $\bZ^d$, $d=1,2$, it is known that the critical curve $\rho+\cos(\pi/n)=0$ determines which moments remain bounded in time. In dimensions where the random walk is transient, the condition $\rho+\cos(\pi/n)<0$ is still sufficient for the uniform boundedness of $n$th moments, see \cite[Thm.\ 2.5]{BDE11}. However, as we will see, we can recover a result of~\cite{AD11} to see it is no longer sharp. 
\begin{corollary}\label{cor:explicit-moment-computations}
\begin{itemize}
\item[(i)] Let $\rho\in(-1,0)$, $\calS\in\{\bZ^d,\bR\}$ and $(\bfu_t)_{t\ge0}$ be the solution of  ${\rm SBM} (\rho, \infty)_{\bfu_0}$ for nonnegative and bounded initial condition $\bfu_0$. 
For second mixed moments, the duality \eqref{eq:moment duality infinite gamma} reads
\begin{equation}\label{eq:mixed_second_moment}
\E_{\bfu_0}\left[\langle u^\ssup{1}_t,\phi\rangle\langle u^\ssup{2}_t,\psi\rangle\right]=\int_{\cS^2}\phi(x)\psi(y)\,\E_{x,y}\left[u_0^\ssup{1}(X_t^\ssup{1})u_0^\ssup{2}(X_t^\ssup{2})\1_{t<\tau}\right]\,dx\,dy,\end{equation}
where $\phi,\psi\in L^1(\R)$ and $\tau$ denotes the first collision time of the two random walks resp.\ Brownian motions $(X^\ssup{1},X^\ssup{2})$.
\item[(ii)] Let $\rho\in[-1,1]$, $\cS = \Z^d$ for $d\ge3$ and $(\bfu_t)_{t\ge0}$ be the solution of ${\rm dSBM}(\rho,\gamma)_{\bfu_0}$ for $\gamma \in (0,\infty)$ and nonnegative bounded initial condition $\bfu_0$. 
Let $p_d$ be the return probability of a random walk to the origin.
If $\frac{\rho\gamma}{2(1-p_d)}<1$, then 
\[ \bE_{\bfu_0} \big[ u_t^\ssup{1}(x)u_t^\ssup{2}(x)  \big]\leq \norm{\bfu_0}_\infty^2\left(1-\frac{\rho\gamma}{2(1-p_d)}\right)^{-1} \quad\text{ for all $t\geq 0,\, x\in\bZ^d$}. \]
If $\frac{\rho\gamma}{2(1-p_d)}\geq1$, then $\lim_{t\to\infty}\bE_{(\1,\1)}\left[ u_t^\ssup{1}(x)u_t^\ssup{2}(x)\right]=\infty$.
\end{itemize}
\end{corollary}

\medskip
\begin{rem}
\begin{itemize}
\item[(a)] 
Note for part (ii) that $\frac{\rho\gamma}{2(1-p_d)}<1$ for all $\gamma>0$ only if $\rho\leq 0$. In particular $\rho=0$ has uniformly bounded second moments in $d\geq 3$. 
\item[(b)]Results on the second moments of ${\rm dSBM}(\rho,\gamma)$ have already been obtained in \cite[Prop.\ 2.3]{AD11}. 
We give an alternative proof using the explicit control on $M_t^\sse\gamma$, see
the representation \eqref{eq:M_t-2-simple} below. 
However, the latter is no longer true for higher moments
and instead $M_t^\sse\gamma$ is a random product of matrices (namely the matrices $K_{\tau_{k+1}-\tau_k}(\cdot,\pi(\bfX_{\tau_k}))$ in the terminology of Section \ref{sec:duality}). 
Since the leading eigenvalue and corresponding eigenvectors are known explicitly, there is some hope to obtain results. Nevertheless, since the random matrices are neither independent nor stationary, this is not trivial.
\end{itemize}
\end{rem}

\subsection{The case $\rho=-1$}\label{ssec:-1}
In Theorem \ref{thm:moment_duality_infinite}, we excluded $\rho=-1$ since in this case the infinite rate limit has not yet been constructed for general initial conditions. 
Our third main result remedies this situation and establishes the infinite rate symbiotic branching model $\mathrm{SBM}(-1,\infty)$ for both continuous and discrete space, and characterizes it via the moment duality \eqref{eq:moment duality infinite gamma} and the Markov property. 
In order to state the result properly, we introduce the following additional notation: For $\calS\in\{\Z^d,\R\}$, consider the space of all (equivalence classes of) nonnegative bounded measurable functions on $\cS$, which can be identified with the space of all Radon measures on $\cS$ having bounded densities w.r.t.\ Lebesgue measure if $\cS=\R$ resp.\ counting measure if $\cS=\Z^d$. Writing $\cM_b(\cS)$ for this space, we have $\cM_b(\cS)\subseteq\cM_\tem(\cS)$, and we topologize $\cM_b(\cS)$ by the subspace topology inherited from $\cM_\tem(\cS)$ (for which we refer the reader e.g.\ to Appendix A.1 in \cite{BHO15}).

\begin{thm}\label{thm:main1}
Let  $\rho=-1$. For $\calS\in\{\Z^d,\R\}$, consider initial conditions $\bfu_0\in \cM_b(\calS)^2$. 
For each $\gamma >0$, let $(\bfu_t^\sse{\gamma})_{t \geq 0}$ denote the solution to $\mathrm{SBM}(-1,\gamma)_{\bfu_0}$, considered as measure-valued process.
\begin{itemize}
\item[a)]
As $\gamma \ra \infty$, $(\bfu_t^\sse{\gamma})_{t\ge0}$ converges in law in $\cC_{[0,\infty)}(\calM_b(\calS)^2)$ w.r.t.\ the Skorokhod topology
to a Markov process $(\bfu_t)_{t \geq 0}\in \cC_{[0,\infty)}(\calM_{b}(\calS)^2)$ 
satisfying the moment duality \eqref{eq:moment duality infinite gamma} for all $n\in\N$, and this property uniquely determines the law of $(\bfu_t)_{t \geq 0}$. 
\item[b)]
The limiting process $(\bfu_t)_{t\ge0}$ from a) has the strong Markov property w.r.t.\  the usual augmentation of its natural filtration.
Moreover, for each fixed time $t>0$,
$\bfu_t$ satisfies the separation of types-property, i.e.\ $u_t^\ssup{1}$ and $u_t^\ssup{2}$ are mutually singular almost surely.
\end{itemize}
\end{thm}

\begin{rem}
We will refer to the limiting process as the solution of $\mathrm{SBM}(-1,\infty)$
(resp.\ $\mathrm{cSBM}(-1,\infty)$ for $\cS = \R$ and $\mathrm{dSBM}(-1,\infty)$ for $\cS = \Z^d$).
For continuous space, Theorem \ref{thm:main1} extends the results of \cite{BHO15} to the case $\rho=-1$, which was left open in that paper (however, 
$\rho=-1$ was erroneously included in the statement of \cite[Thm.\ 1.5]{BHO15}). 
The characterization of the limit in \cite[Thms.\ 1.10, 1.12]{BHO15} in terms of a martingale problem is replaced by a characterization via moments.
The latter is possible, since for $\rho = -1$ the noises are perfectly negatively correlated so that the sum of the solutions 
solves the heat equation, thus the solutions are dominated by a deterministic function. 
This fact in particular implies also the absolute continuity of the limit measures (if $\calS=\R$) and the boundedness of the densities, so that the limiting process takes indeed values in $\cM_b(\cS)^2$. 
Combining this with the identity \eqref{eq:mixed_second_moment} for second mixed moments, the separation of types-property of $\bfu_t$ can be derived as in \cite{BHO15}.
\end{rem}

We now proceed to a more explicit characterization of the limit in Theorem \ref{thm:main1} for the continuous-space case. First, we collect some additional notation:
Let $\cU$ denote the space of all pairs of absolutely continuous Radon measures $\bfu=(u^\ssup{1},u^\ssup{2})$ on $\R$ with bounded densities such that $u^{\ssup1}$ and $u^{\ssup2}$ are mutually singular and $u^\ssup{1}+u^\ssup{2}$ is equivalent to Lebesgue measure (equivalently, $u^\ssup{1}(x)u^\ssup{2}(x)=0$ and $ u^\ssup{1}(x)+u^\ssup{2}(x)>0$ for almost all $x\in\bR$).
Note that $\cU\subseteq\cM_b(\R)^2\subseteq\cM_\tem(\R)^2$, and again we topologize $\cU$ with the subspace topology inherited from $\cM_\tem(\R)^2$.
For $\bfu\in\cU$, we define
\begin{align}
\cI(\bfu) := \supp(u^\ssup{1})\cap \supp(u^\ssup{2}),
\end{align} 
where $\supp(u)$ denotes the measure-theoretic support of $u\in\cU$, i.e.
\[ \supp(u):=\{x\in\R :  u\left(B_\eps(x)\right)>0 \text{ for all }\eps>0\} . \]
We call the elements of $\cI(\bfu)$ \emph{interface points} or just \emph{interfaces}.
The configurations with exactly $n\in\bN$ interface points are denoted by $\cU_n$, i.e.
\[ \cU_n := \{ \bfu \in \cU \, : \, |\cI(\bfu)| = n \} . \]
We write $m(\bfu,x):=1$ if $x\in\supp(u^\ssup1)\setminus \cI(u)$ and $m(\bfu,x):=2$ if $x \in \supp(u^\ssup 2)\setminus \cI(u)$, while setting $m(\bfu,x) := 0$ if $x \in\cI(u)$. 

Throughout the rest of this section, given initial conditions $\bfu_0=(u_0^\ssup{1},u_0^\ssup{2})\in \cM_b(\R)^2$ 
for $\mathrm{cSBM}(-1,\infty)_{\bfu_0}$, we write 
\[w_t:=S_tw_0\]
for the solution to the deterministic heat equation with initial condition $w_0:=u^\ssup{1}_0+u^\ssup{2}_0$
and recall that since $\rho = -1$, we know $u_t^\ssup{1} + u_t^\ssup{2} = w_t$ for all $t \geq 0$.
Note that in view of Theorem \ref{thm:main1} b), for each fixed $t>0$ we have almost surely $\bfu_t\in\cU$, provided $w_0\ne0$. 

Our next result deals with initial conditions of `single interface type':

\begin{thm}\label{thm:annihilating-BM-_onepoint}
Assume $\bfu_0\in\cU_1$. Let $(\bfu_t)_{t\ge0}$ denote the solution of
$\mathrm{cSBM}(-1,\infty)_{\bfu_0}$.
Then we have, almost surely, 
\[  \bfu_t\in\cU_1\text{ for all }t \geq 0 . \]
Let $(I_t)_{t\ge0}$ denote the single interface process defined by the unique element of $\cI(\bfu_t)$, $t>0$. Then, 
  almost surely $(I_t)_{t \geq 0}$ is continuous and
  there exists a standard Brownian motion $(B_t)_{t \geq 0}$ 
 such that 
\begin{align}\label{eq:interface-movement3}
I_t = I_0 - \int_0^t \frac{w'_s(I_s)}{w_s(I_s)}\, ds +B_t, \quad t \geq 0,
\end{align}
where the integral on the right hand side exists as an improper integral, and the process $(I_t)_{t\ge0}$ is the unique (in law) weak solution of the SDE \eqref{eq:interface-movement3}.
Moreover, $\bfu_t$ can be recovered as
\begin{align}\label{representation-single-interface} u_t^\ssup{i} (dx) = \left\{ \begin{array}{ll} \1_{\{x \leq I_t\}}w_t(x)\,dx & \mbox{if } \lim_{x \ra -\infty} m(\bfu_0 ,x) = i, \\
\1_{\{x \geq I_t\}} w_t(x)\,dx & \mbox{if } \lim_{x \ra -\infty} m(\bfu_0 ,x) = 3-i, \end{array} \right. \end{align}
$i=1,2$.
\end{thm}

\begin{rem}
If $\bfu_0\in\cU_1$ is such that $w_0:=u^\ssup1_0+u^\ssup2_0$ is continuously differentiable with $\sup_{x\in\R}|w_0'(x)|<\infty$ and $\inf_{x\in\R} w_0(x)>0$, then
the drift term $(s,x)\mapsto\frac{w_s'(x)}{w_s(x)}$ in equation~\eqref{eq:interface-movement3} is continuous and globally bounded on $\R^+\times\R$,
and existence and uniqueness of a weak 
solution to the SDE follow from the standard theory (see e.g.\ \cite{SV79} or \cite[Ch.\ 5.3]{KS98}).
In particular, in this case the integrand 
in~\eqref{eq:interface-movement3} 
is bounded and the integral is a proper integral.
For general $\bfu_0\in\cU_1$ however, we know only that $w_0$ is bounded and strictly positive almost everywhere, and standard theory does not cover existence and uniqueness for equation~\eqref{eq:interface-movement3}. In particular, it can occur that $w_0(I_0)=0$, $w'_0(I_0-)=-\infty$ and $w'_0(I_0+)=+\infty$ so that the integrand need not even be locally bounded. Our result shows that nevertheless the SDE \eqref{eq:interface-movement3} has a unique weak solution.
\end{rem}

Our next result covers the case that $\cI(\bfu_0)$ consists of more than a single point -- even infinitely many points --
as long as there are no accumulation points.
Informally, in this case the interfaces follow each the dynamics of a Brownian motion 
with drift as in~\eqref{eq:interface-movement3} and upon collision both motions annihilate.

To describe the limiting system formally, we introduce the following terminology:
We call a collection $\{ (Y_t^i)_{t \geq 0} \, : \, i \in J \}$ of c\`adl\`ag stochastic processes indexed by an at most countable set $J$ and
taking values in $\R \cup \{ \dagger\}$, where $\dagger$ is interpreted as a cemetery state, a \emph{regular annihilating system} if,
almost surely, the following holds:
\begin{itemize}
	\item The initial positions are distinct and $\{ Y_0^i \}_{i \in J}$ has no accumulation points.
		\item 	$Y_t^i = \dagger$ if and only if there exist $s \leq t$ and $j \in J\setminus \{ i \}$ such that 
	$Y_{s-}^i =  Y_{s-}^j \in \R$.
	\item $\tau_i := \inf\{t \geq 0 \, : \,   Y_t^i = \dagger \} > 0$.
	\item Each process $Y^i$ is continuous at any time $t < \tau_i$.
	\item There are no triple annihilations, i.e.\  there are no times $t$ such that $Y_{t-}^i = Y_{t-}^j = Y_{t-}^k \in \R$ for distinct $i,j,k \in J$. 
\end{itemize}

Now suppose we are given $\bfu_0 \in \cU$ such that $\cI(\bfu_0)$ has no accumulation points and a regular annihilating system $\{ (Y_t^x)_{t \geq 0} \, : \, x \in \cI(\bfu_0) \}$ indexed by and starting from $\cI(\bfu_0)$. 
Consider the partition of $[0,\infty) \times \R$ induced by the graphs of the annihilating paths:
We can `color' each component of the partition in a way that is consistent with the coloring $m(\bfu_0,\cdot)$ of $\R$.
More precisely, define a mapping $\hat m:[0,\infty)\times\R\to\{0,1,2\}$ 
as follows: Let $\hat m$ be equal to $0$ on the closure of the graphs of the annihilating paths, i.e.
 \[\hat m(t,x):=0\qquad \text{for } (t,x)\in \cJ := {\rm cl} \Big(\bigcup_{j \in J} \{ (t,Y^j_t)\, : \, t \in [0,\tau_j) \}\Big) . \]
Then setting
\[\hat m(0,x) := m(\bfu_0,x)\qquad \text{for } x\in\R,\]
$\hat m(\cdot,\cdot)$ is defined by the
requirement that it is locally constant on the complement $\cJ^c$, cf.\ Figure~\ref{fig:colouring-1}. 
We call $\hat m$ the \emph{standard colouring} of $[0,\infty) \times \R$ induced by $\bfu_0$ and $\{ (Y_t^x)_{t \geq 0} \, : \, x \in \cI(\bfu_0) \}$. 
Obviously, $\hat m(t,x)$ is the unique extension of $\hat m(0,x)=m(\bfu_0,x)$ from $\{0\}\times\R$ to $[0,\infty)\times\R$ which is continuous on $\cJ^c$, jointly in the variabes $(t,x)$.

\begin{figure}
\centering
\includegraphics{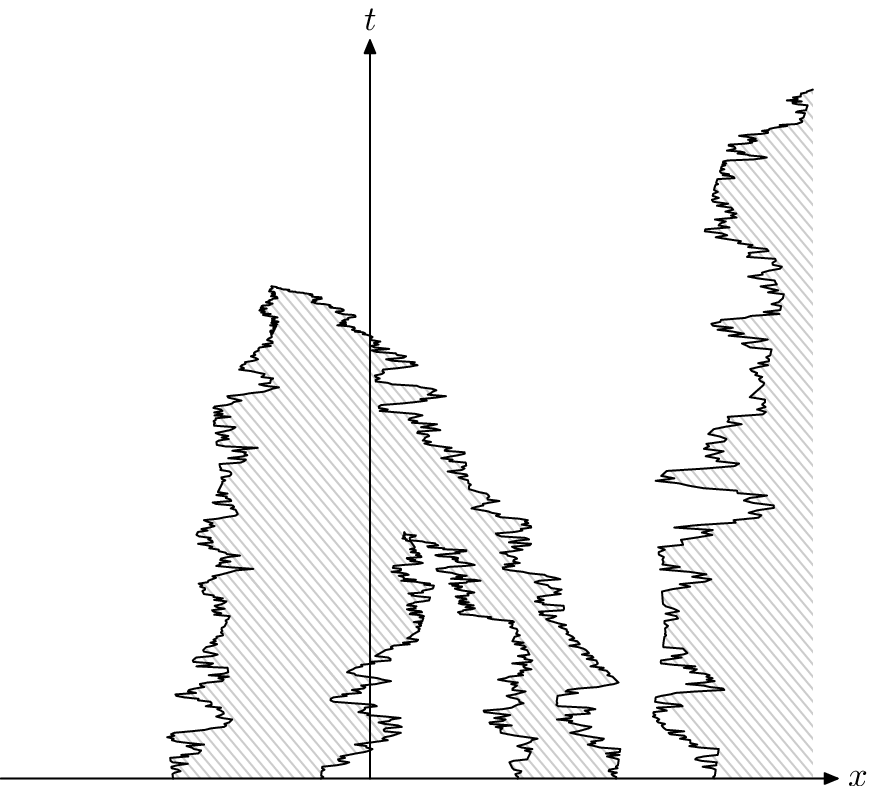}
\caption{
An illustration of the standard colouring $\hat m$ of $[0,\infty)\times\R$ induced by an initial configuration $\bfu_0$ with five interfaces and a regular annihilating system starting from $\cI(\bfu_0)$.
Type 1 is drawn in white and type 2 is shaded grey.
The corresponding standard element $(\hat\bfu_t)_{t\ge0}\in\calC_{[0,\infty)}(\cU)$ is such that $\hat\bfu^\ssup{1}_t(x)$ (resp.\ $\hat\bfu^\ssup{2}_t(x)$) agrees with $w_t(x)$ for all $(t,x)$ in the white (resp.\ shaded) area and is zero otherwise.
}
\label{fig:colouring-1}
\end{figure}

Using $\hat m$, we define 
\begin{equation}\label{eq:standard_element} 
\hat \bfu_t(x) := \left(w_t(x) \1_{\{\hat m(t,x) = 1 \}}, w_t(x) \1_{\{\hat m(t,x) = 2 \}} \right) , \quad   t \geq 0,\, x \in \R.
\end{equation}
It is easy to see that $\hat\bfu_t\in\cU$ for each $t>0$ and (by the properties of the regular annihilating system) that the process $(\hat\bfu_t)_{t\ge0}$ has continuous paths (recall that $\cU$ is topologized by the subspace topology inherited from $\cM_\tem(\R)^2$). 
We call $(\hat\bfu_t)_{t\ge0}$ the \emph{standard element} of $\calC_{[0,\infty)}(\cU)$ induced by $\bfu_0$ and $\{ (Y_t^x)_{t \geq 0} \, : \, x \in \cI(\bfu_0) \}$. 

\newpage

\begin{thm}\label{thm:annihilating-BM}
Assume that $\bfu_0 \in \cU$ and $\cI(\bfu_0)$ has no accumulation points. 
Let $(\bfu_t)_{t\ge0}$ denote the solution of $\mathrm{cSBM}(-1,\infty)_{\bfu_0}$.
Then there exists a regular annihilating system $\{(I_t^x )_{t \geq 0}\, : \, x \in \cI(\bfu_0)\}$ starting from $\cI(\bfu_0)$
such that each coordinate independently follows the law of the single-point interface process of Theorem~\ref{thm:annihilating-BM-_onepoint}
up to the first collision with another (surviving) motion, upon which both motions annihilate. 
Denote by $(\hat \bfu_t)_{t\ge0}$ the standard element of $\calC_{[0,\infty)}(\cU)$ induced by $\bfu_0$ and $\{ (I_t^x )_{t \geq 0}\, : \, x \in \cI(\bfu_0)\}$,
as defined in \eqref{eq:standard_element}.
Then we have
\[(\bfu_t)_{t \geq 0}\overset{d}=(\hat\bfu_t)_{t \geq 0} \qquad\text{on }\cC_{[0,\infty)}(\calM_b(\R)^2).\]
\end{thm}

\begin{remark}
If $\cI(\bfu_0)$ is finite, then it is obvious how to construct the regular annihilating system. 
However, due to the lack of monotonicity, it becomes trickier to define the system for infinitely many particles. 
Our construction uses a coupling which embeds the annihilating system into a system of instantaneously\emph{ coalescing} Brownian motions with drift,
as also used (in a random walks context) by \cite{A81}.
\end{remark}

We now deal with the case of completely general initial densities, which do not have to be mutually singular anymore. 
Here the situation is more involved, as the set of interface points $\cI(\bfu_0)$ can be an interval or even $\bR$. 
Although we know already by Theorem \ref{thm:main1} b) that the measures $u_t^\ssup{1}$ and $u_t^\ssup{2}$ are mutually singular at each positive time, a priori the set $\cI(\bfu_t)$ might still be very complicated. Our final result states that for all $t>0$ the set of interface points is in fact discrete, and these points move as in the case of Theorem~\ref{thm:annihilating-BM}. 

\newpage
\begin{thm}\label{thm:overlapping}
Assume initial conditions $\bfu_0\in \cM_b(\R)^2$ such that $w_0:=u_0^\ssup{1}+u_0^\ssup{2}\ne0$.
Let $(\bfu_t)_{t\ge0}$ denote the solution of $\mathrm{cSBM}(-1,\infty)_{\bfu_0}$. Then, almost surely, $\cI(\bfu_t)$ contains no accumulation point for any $t>0$.
Moreover, for any $t_0>0$, the law of $(\bfu_t)_{t \geq t_0}$ is given as in Theorem~\ref{thm:annihilating-BM} when started in $\bfu_{t_0}$.
\end{thm}

\begin{remark}
\begin{itemize}
\item[a)] 
In \cite{T95}, Tribe considered the special case of complementary initial densities $\bfu_0$ where the types are separated by finitely many interfaces and $u_0^\ssup{1}+u_0^\ssup{2}\equiv1$.
In that case, Tribe proved that the interfaces move as annihilating Brownian motions (without drift).

In contrast, we allow for essentially arbitrary initial conditions. 
If $u_0^\ssup{1}+u_0^\ssup{2}\ne1$, the dynamics of interfaces is influenced by the relative `height' difference of the two populations at either side of the interfaces, 
yielding the additional drift term in \eqref{eq:interface-movement3}. Moreover, we require neither a finite number of interfaces nor even the initial separation of types.
As Theorem \ref{thm:overlapping} shows, the process `comes down from infinity' in the sense that for any positive time, locally there are only finitely many interfaces.
\item[b)] \emph{Entrance laws for annihilating Brownian motions.} 
For $u^\ssup{1}_0 + u^\ssup{2}_0 \equiv 1$, Theorem~\ref{thm:overlapping} allows us to characterize entrance laws for annihilating Brownian motions. 
As in \cite[Sect.\ 2.3]{TZ11}, 
a particularly interesting entrance law can be obtained by approximation from a system
of annihilating Brownian motions starting at the points of a Poisson point process with intensity $\la$ and letting $\la \ra\infty$. 
Informally, this entrance law corresponds to a system of annihilating Brownian motions `starting from every point' on the real line. 
However, as already observed in \cite{TZ11}, the examples in \cite[Sec.\ 3]{BG80} suggest that other approximations may lead to different entrance laws.
By using Theorem~\ref{thm:overlapping} and the relation to $\mathrm{cSBM}(-1,\infty)$, one can clearly see why those examples 
can lead to different entrance laws, and in fact we can obtain a complete classification of entrance laws for annihilating Brownian motions.
For details of this correspondence, we refer to~\cite{HOV16}.
\end{itemize}
\end{remark}

The remaining paper is structured as follows:
We start by showing in Section~\ref{se:reinterpretation} 
how to reinterpret the original moment duality and we prove Theorem~\ref{thm:rewrite}. 
Then, in Section~\ref{sec:duality} we can send $\gamma \ra \infty$
in the reinterpretation of the duality and we prove Theorems~\ref{thm:M_t-infty},
\ref{thm:moment_duality_infinite} and Corollary~\ref{cor:explicit-moment-computations}. 
In Section~\ref{sec:rhominus1} we specialize to the case $\rho = -1$, establishing first the existence, uniqueness and basic properties of $\mathrm{SBM}(-1,\infty)$ 
by proving Theorem \ref{thm:main1}.
We then proceed to the explicit description of the limiting process and prove Theorems~\ref{thm:annihilating-BM-_onepoint}, \ref{thm:annihilating-BM} and~\ref{thm:overlapping}.

\section{The reinterpretation of the moment duality}\label{se:reinterpretation}

\newcommand{\tM}{{\widetilde{M}}}

In this section we prove Theorem~\ref{thm:rewrite}. The proof is based on the original moment duality \eqref{moment duality finite gamma}.
In this entire section we assume $\gamma>0$, $\rho\in[-1,1]$ and $n\in\bN$ to be fixed. 

For $\tM_0\in\cM(\{1,2\}^n)$, define 
\begin{align}\bal\label{defn:M_t^gamma}
\tM_t(X,\tM_0)(b) := \sum_{c\in\{1,2\}^n} \tM_0(c)\,\bE_{c} \left[e^{\gamma(L^=_t+\rho L^{\neq}_t)} \ind_{C_t=b}
 \;\middle|\; X_{[0,t]} \right],\qquad b\in\{1,2\}^n,
\eal 
\end{align}
that is $\tM_t(\cdot,c)\in\cM(\{1,2\}^n)$ is the expected exponential correction term as a function of the  paths $\bfX=(X_t^\ssup{1},\ldots,X_t^\ssup{n})_{t\geq0}$ of either random walks or Brownian motions and interpreted as a measure on the colorings. Note that $\tM_t(\bfX,\tM_0)$ is measurable w.r.t. $\sigma(\bfX_s:0\le s\le t)$.
We will show that $\tM_t$ coincindes with $M_t^\sse\gamma$ from \eqref{eq:ODE_X_new}. 
The proof of Theorem~\ref{thm:rewrite} is then simply the original 
duality~\eqref{moment duality finite gamma}.

\begin{lemma}\label{lemma:M-Markov}
Fix $s,t\geq0$.  
Let $(\theta_t \bfX)_r:=\bfX_{t+r}$ be the shift operator.
Then
\begin{align}
\tM_{t+s}(\bfX,\tM_0) =\tM_{t+s}(\bfX_{[0,t+s]},\tM_0)= \tM_s\left((\theta_t \bfX)_{[0,s]}, \tM_t(\bfX_{[0,t]},\tM_0)\right)
\end{align}
\end{lemma}
almost surely with respect to the law of $\bfX$.
\begin{proof}
By the definition of $\tM_{t+s}$,
\begin{align}
&\tM_{t+s}(\bfX,\tM_0)(b)= \sum_{c\in\{1,2\}^n}\tM_0(c)\,\bE_{c}\left[e^{\gamma(L^=_{t+s}+\rho L^{\neq}_{t+s})}\ind_{C_{t+s}=b} \;\middle|\; \bfX_{[0,t+s]} \right]	\\
&= \sum_{\mathclap{c\in\{1,2\}^n}}\tM_0(c)\,\bE_{c}\left[e^{\gamma(L^=_{t}+\rho L^{\neq}_{t})} \bE\left[ e^{\gamma(L^=_{t+s}-L^=_{t}+\rho (L^{\neq}_{t+s}-L^{\neq}_{t}))}\ind_{C_{t+s}=b} \;\middle|\; C_t, \bfX_{[0,t+s]}\right] \;\middle|\; \bfX_{[0,t+s]} \right].
\end{align}
The increments of the local times, $L^=_{t+s}-L^=_t, L^{\neq}_{t+s}-L^{\neq}_t$, are a function of $(\theta_t \bfX)_{[0,s]}$ and equal the local times $L^=_s, L^{\neq}_s$ of $\theta_t \bfX$. Furthermore, 
conditioned on $X$, the coloring process $(C_t)_{t\ge0}$ is a time-inhomogeneous Markov process, and the law of $C_{t+s}$ depends only on $C_t$ and the path $(\theta_t \bfX)_{[0,s]}$. 
Therefore
\begin{align}
\bE\left[ e^{\gamma(L^=_{t+s}-L^=_{t}+\rho (L^{\neq}_{t+s}-L^{\neq}_{t}))}\ind_{C_{t+s}=b} \;\middle|\; C_t, \bfX_{[0,t+s]}\right]
=
\bE_{C_t}\left[ e^{\gamma(L^=_s+\rho L^{\neq}_{s}}\ind_{C_{s}=b} \;\middle|\; (\theta_t\bfX)_{[0,s]}\right].
\end{align}
Summing over all possible values of $C_t$ and using again that $C_t$ and the local times $L^=_t, L^{\neq}_t$ depend only on $\bfX_{[0,t]}$, we get
\begin{align}
\tM_{t+s}(\bfX,\tM_0)(b)
&= \sum_{\mathclap{c,c'\in\{1,2\}^n}}\tM_0(c)\,\bE_{c}\left[e^{\gamma(L^=_{t}+\rho L^{\neq}_{t})}\ind_{C_{t}=c'}  \;\middle|\; \bfX_{[0,t]} \right]\bE_{c'}\left[ e^{\gamma(L^=_s+\rho L^{\neq}_{s}}\ind_{C_{s}=b} \;\middle|\; (\theta_t\bfX)_{[0,s]}\right]	\\
&=\sum_{\mathclap{c'\in\{1,2\}^n}} \tM_t(\bfX_{[0,t]},\tM_0)(c') \,\bE_{c'}\left[ e^{\gamma(L^=_{s}+\rho (L^{\neq}_{s})}\ind_{C_{s}=b} \;\middle|\; (\theta_t \bfX)_{[0,s]} \right]	\\
&=\tM_s\left((\theta_t \bfX)_{[0,s]},\tM_t(\bfX_{[0,t]},\tM_0)\right)(b).
\end{align}
\end{proof}

We now have to distinguish between the discrete and the continuous case, since collisions and local times of random walks and Brownian motions behave somewhat differently. In particular, in the discrete case collisions between more than two random walks can happen.

\begin{lemma}\label{lemma:M-ODE-collision}
Assume the discrete space case, where $\bfX=(X^\ssup{1},\ldots,X^\ssup{n})$ is the collection of $n$ independent random walks on $\bZ^d$. Let $\Gamma_t:=\{ (i,j) : X_t^\ssup{i}=X_t^\ssup{j}, 1\leq i < j \leq n \}$ be the set of collision pairs  at time $t$, and for $c\in\{1,2\}^n$ let $\Gamma_t^=(c) := \{ (i,j) \in \Gamma_t : c_i=c_j\}$ and $\Gamma_t^{\neq}:=\{ (i,j)\in \Gamma_t : c_i\neq c_j \}$ be the decomposition into same color collisions and different color collisions according to the coloring $c$. Then $\tM_t=\tM_t(\bfX, \tM_0)$ satisfies the following linear ODE:
\begin{align}\label{eq:ODE_X}
 \frac{d}{dt}\tM_t(b) &= \gamma\rho \abs{\Gamma^{\neq}_t(b)}\tM_t(b) + \frac{\gamma}{2} \sum_{\substack{c\in\{1,2\}^n: \\d(b,c)=1}} \abs{\Gamma^=_t(c)\cap \Gamma^{\neq}_t(b)}\tM_t(c),\qquad b\in\{1,2\}^n.
\end{align}
\end{lemma}
\begin{proof}
Since $\bfX$ is right continuous, choose $\epsilon$ small enough that $\bfX_{[t,t+\epsilon]}\equiv \bfX_t$. Then by Lemma \ref{lemma:M-Markov}
\begin{align}
\tM_{t+\epsilon}(b)-\tM_t(b) &= \sum_{c\in\{1,2\}^n}\tM_t(c) \,\bE_{c}\left[e^{\gamma(L^=_{\epsilon}+\rho L^{\neq}_\epsilon)}\ind_{C_\epsilon=b} \;\middle|\; \theta_t\bfX \right] -\tM_t(b).
\end{align}
Since the total rate of color change in a coloring $c$ is given by $\gamma\abs{\Gamma^=_t(c)}< \gamma n^2$, the probability of two color changes in time $\epsilon$ is small and the above is equal to
\begin{align}
& \tM_t(b) \left(\bE_{b}\left[e^{\gamma(L^=_{\epsilon}+\rho L^{\neq}_\epsilon)}\ind_{C_\epsilon=b} \;\middle|\;\theta_t\bfX \right] -1\right)	\\
&+\sum_{c:d(c,b)=1}\tM_t(c) \,\bE_{c}\left[e^{\gamma(L^=_{\epsilon}+\rho L^{\neq}_\epsilon)}\ind_{C_\epsilon=b} \;\middle|\; \theta_t\bfX \right] +O(\epsilon^2)\\
&= \tM_t(b) \left(e^{\gamma(\abs{\Gamma^=_t(b)}\epsilon+\rho |\Gamma^{\neq}_t(b)|\epsilon)}e^{-\gamma \abs{\Gamma^=_t(b)}{\epsilon}} -1 \right) 	\\
&+\sum_{c:d(c,b)=1}\tM_t(c)\, e^{O(\epsilon)}\left(1-e^{-\frac\gamma2\abs{\Gamma_t^=(c)\cap \Gamma_t^{\neq}(b)}\epsilon}\right)   +O(\epsilon^2),
\end{align}
with $d(\cdot,\cdot)$ denoting the Hamming distance on $\{1,2\}^n$. In the last line we used the fact that if $d(c,b)=1$, there is one coordinate, say $i$, which changed color. The rate that this color change happens is $\gamma/2$ times the number of other particles at $X_t^\ssup{i}$ which have color $c_i$, and this number is exactly given by pairs of particles which have the same color before the change, and a different afterwards. The exact value of $\gamma(L^=_{\epsilon}+\rho L^{\neq}_\epsilon)$ depends on the time of the color change, but is of order $\epsilon$ since its absolute value is bounded by $n^2\gamma\epsilon$. By writing $e^x=1+x+O(x^2)$, we get that the above is
\begin{align}
\tM_t(b) \rho\gamma \abs{\Gamma^{\neq}_t(b)}\epsilon + \frac\gamma2\sum_{c:d(c,b)=1}\tM_t(c) \abs{\Gamma_t^=(c)\cap \Gamma_t^{\neq}(b)}\epsilon +O(\epsilon^2).
\end{align}
Dividing by $\epsilon$ and then sending $\epsilon$ to 0 completes the proof.
\end{proof}

\begin{lemma}[The discrete case]\label{lem:tM-ODE}
In the discrete case, we have
\begin{align}\label{eq:tM-ODE}
\frac{d}{dt}\tM_t(b)&= \frac{\gamma\rho}{2} \sum_{i,j =1}^n \ind_{b_i\neq b_j} \tM_t(b)\frac{dL_t^{i,j}}{dt} + \frac{\gamma}{2}
 \sum_{i,j=1}^n \ind_{b_i\neq b_j} \tM_t(\widehat{b}^i)\frac{dL_t^{i,j}}{dt},\qquad b\in\{1,2\}^n.
\end{align}
Here $L^{i,j}_t=\int_0^t \ind_{X^\ssup{i}_s=X^\ssup{j}_s}\,ds$ are the pair local times and $\widehat{b}^i$ is the coloring $b$ flipped at $i$.
\end{lemma}
\begin{proof}
By simple counting we get that  $\abs{\Gamma_t^{\neq}(b)}=\frac{1}{2} \sum_{i,j =1}^n \ind_{b_i\neq b_j}\ind_{X_t^\ssup{i}=X_t^\ssup{j}}$.
For each $c$ with $d(c,b)=1$ there is some index $i$ so that $c=\widehat{b}^i$. Then 
\[ \abs{\Gamma^=_t(c)\cap \Gamma^{\neq}_t(b)}=\sum_{j=1}^n \ind_{b_j\neq b_i}\ind_{X_t^\ssup{i}=X_t^\ssup{j}},\] 
and summing over all possible $i$ we get
\[
\sum_{c: d(c,b)=1} \abs{\Gamma^=_t(c)\cap \Gamma^{\neq}_t(b)} =  \sum_{i,j =1}^n \ind_{b_i\neq b_j}\ind_{X_t^\ssup{i}=X_t^\ssup{j}}
.\]
Now \eqref{eq:tM-ODE} follows from Lemma \ref{lemma:M-ODE-collision}.
\end{proof}
Note that since \eqref{eq:tM-ODE} and \eqref{eq:ODE_X_new} are the same equations, $\tM_t(\bfX,\delta_c)$ and $M_t^\sse\gamma(\bfX,\delta_c)$ agree, thus Theorem \ref{thm:rewrite} is now proven for the discrete case.

The proof for the continuous case follows along the same lines. It is in some parts simpler, as simultaneous collisions between more than two Brownian motions do not happen almost surely. Theorem \ref{thm:rewrite} for the continuous case follows from the following lemma, which is a version of Lemma \ref{lem:tM-ODE}.
\begin{lemma}[The continuous case]\label{lemma:M-ODE-collision-continuous}
Assume the continuous space setting, where  $\bfX=(X^\ssup{1},\ldots,X^\ssup{n})$ is the collection of $n$ independent Brownian motions.
Then $\tM_t=\tM_t(\bfX, \tM_0)$ satisfies the following linear ODE almost surely w.r.t.\ the law of $\bfX$:
\begin{align}
 d\tM_t(b)&= \frac{\gamma\rho}{2} \sum_{i,j =1}^n \ind_{b_i\neq b_j} \tM_t(b)\, dL_t^{i,j} 
 + \frac{\gamma}{2}
 \sum_{i,j=1}^n \ind_{b_i\neq b_j} \tM_t(\widehat{b}^i)\, dL_t^{i,j},\qquad b\in\{1,2\}^n.  
\end{align}
\end{lemma}
\begin{proof}
Since $\bfX$ is right continuous and almost surely has no multiple collisions, choose $\epsilon>0$ small enough so that $L_{t+\epsilon}^{i,j}=L_{t}^{i,j}$ for all pairs $(i,j)$ except for possibly one, say $(k,\ell)$, where $\kappa:=L_{t+\epsilon}^{k,\ell}- L_{t}^{k,\ell}\geq0$. Then by Lemma \ref{lemma:M-Markov}
\begin{align}
\tM_{t+\epsilon}(b)-\tM_t(b) &= \sum_{c\in\{1,2\}^n}\tM_t(c)\, \bE_{c}\left[e^{\gamma(L^=_{\epsilon}+\rho L^{\neq}_\epsilon)}\ind_{C_\epsilon=b} \;\middle|\; \theta_t\bfX \right] -\tM_t(b).
\end{align}
Since the total rate of color change in a coloring $c$ is at most $\gamma$, the probability of two color changes in time $\epsilon$ is small and the above is equal to
\begin{align}
& \tM_t(b) \left(\bE_{b}\left[e^{\gamma(L^=_{\epsilon}+\rho L^{\neq}_\epsilon)}\ind_{C_\epsilon=b} \;\middle|\;\theta_t\bfX \right] -1\right)	\\
&+\sum_{c:d(c,b)=1}\tM_t(c)\, \bE_{c}\left[e^{\gamma(L^=_{\epsilon}+\rho L^{\neq}_\epsilon)}\ind_{C_\epsilon=b} \;\middle|\; \theta_t\bfX \right] +O(\epsilon^2)\\
&= \tM_t(b) \left(e^{\gamma(\kappa\ind_{b_k=b_\ell}+\rho \kappa \ind_{b_k\neq b_\ell})}e^{-\gamma\kappa \ind_{b_k=b_\ell}} -1 \right) 	\\
&+\ind_{b_k\neq b_\ell}\sum_{i\in\{k,\ell\}}\tM_t(\widehat{b}^i)\, e^{O(\kappa)}\left(1-e^{-\frac{\gamma}{2} \kappa }\right)   +O(\kappa^2).
\end{align}
By writing $e^x=1+x+O(x^2)$, we get
\begin{align}
&\frac{\gamma\rho}{2} \ind_{b_k\neq b_\ell} \tM_t(b) \kappa
 + \frac{\gamma}{2}
 \ind_{b_k\neq b_\ell} \sum_{i\in\{k,\ell\}}\tM_t(\widehat{b}^i) \kappa  +O( \kappa^2)\\
 &=\frac{\gamma\rho}{2} \sum_{i,j =1}^n \ind_{b_i\neq b_j} \tM_t(b)\, (L_{t+\epsilon}^{i,j} - L_t^{i,j}) \\
& \qquad + \frac{\gamma}{2}
 \sum_{i,j=1}^n \ind_{b_i\neq b_j} \tM_t(\widehat{b}^i)\, (L_{t+\epsilon}^{i,j} - L_t^{i,j}) +O((L_{t+\epsilon}^{k,\ell} - L_t^{k,\ell})^2)  
\end{align}
Sending $\epsilon$ to 0 completes the proof.
\end{proof}

\section{The moment duality for $\gamma = \infty$}\label{sec:duality}

In this section, we will derive the moment duality for the case $\gamma = \infty$
and thus prove Theorems~\ref{thm:M_t-infty} and~\ref{thm:moment_duality_infinite}. 
We start by looking at the discrete space case 
in Section~\ref{ssec:discret_M_gamma}. In particular, we introduce an auxiliary 
process $K$, whose asymptotics we analyze in Section~\ref{ssec:asymptotics_K}. 
We then complete the proof of Theorem~\ref{thm:M_t-infty} (discrete space) in 
Section~\ref{ssec:M_infinty_discrete}. The continuous space case is slightly different,
and we prove Theorem~\ref{thm:M_t-infty} in this setting in Section~\ref{ssec:continuous}. Finally, we 
combine Theorem~\ref{thm:M_t-infty} with a dominated convergence argument
to show the moment duality for $\gamma = \infty$, i.e.\ Theorem~\ref{thm:moment_duality_infinite},
in Section~\ref{ssec:proof_duality}.

Since the case $n =1$ is trivial, we will assume throughout that $n \geq 2$.
Also, we can treat the case $\rho =-1$ simultaneously with the case $\rho \in(-1,0)$. 

\subsection{Analysis of $M^\sse{\gamma}$ for the discrete space model}\label{ssec:discret_M_gamma}

In this section, we assume that $\bfX = (X^\ssup{1}, \ldots, X^\ssup{n})$ 
is a simple, symmetric random walk on $(\Z^d)^n$.
Remember that given the paths of $\bfX$, the process $M_t^\sse\gamma=M_t^\sse\gamma(X,M_0)$ is defined as the solution of
\begin{align}\label{eq:ODE1}
 \frac{d}{dt}M^\sse\gamma_t(m) &= \frac{\gamma\rho}{2} \sum_{i,j=1}^n \ind_{m_i\neq m_j} M^\sse\gamma_t(m)\frac{dL_t^{i,j}}{dt} + \frac{\gamma}{2}
 \sum_{i,j=1}^n \ind_{m_i\neq m_j} M^\sse\gamma_t(\widehat{m}^i)\frac{dL_t^{i,j}}{dt},   
\end{align}
for each coloring $m\in\{1,2\}^n$, with initial condition $M_0\in\mathcal{M}\left(\{1,2\}^n\right)$, where we recall that $L_t^{i,j}$ denote the pair local times and $\widehat{m}^i$ is 
the coloring $m$ flipped at $i$.

We notice that the right hand side of \eqref{eq:ODE1} is piecewise linear. 
To make this precise, we introduce the following notation: 
For a  partition $\pi$ of the set $[n] :=\{1, \ldots, n\}$ we
write $ \pi = \{ \pi_1, \pi_2, \ldots, \pi_k \}$ (in increasing order of the smallest element), 
where $\pi_i \subset [n]$ with size $|\pi_i|$ are the blocks of $\pi$, and we define $|\pi| := k$. 
Also, if $M$ is a measure on $\{ 1,2\}^n$ we interpret $M=(M(m))_{m\in\{1,2\}^n}$ as a $2^n$-dimensional vector by ordering the elements of $\{1,2\}^n$ in increasing
lexicographical order.

For a partition $\pi$ of $[n]$, let $A_\rho^\pi$ be defined as the $2^n$-dimensional matrix with nonzero entries given by
\begin{equation}\label{eq:matrixApi}
 A^\pi _\rho(m,m')  = \left\{ \begin{array}{ll} \frac{\rho}{2} \sum_{k=1}^{|\pi|} \sum_{i,j \in \pi_k} \ind_{m_i\neq m_j}
   & \mbox{if } m' = m,  \\
   \frac{1}{2}
\sum_{j \in \pi_k} \ind_{m_i\neq m_j}
 & \mbox{if } m' = \widehat{m}^i  \mbox{ for some } i \in \pi_k,\, k\in\{1,\ldots,|\pi|\}. \end{array} \right.
\end{equation}
Then we can rewrite \eqref{eq:ODE1} as (see also Lemma \ref{lemma:M-ODE-collision})
\begin{align}\label{eq:ODE2}
 \frac{d}{dt}M^\sse\gamma_t(m) &= \gamma A^{\pi(\bfX_t)}_\rho M^\sse\gamma_t(m),
\end{align}
where we take $\pi(\bfX_t)$ as the set of clusters of the random walks $\bfX_t$, i.e.\ $\pi(\bfX_t)$ is the partition of $[n]$ induced by the equivalence 
relation $i \sim j$ iff $X_t^\ssup{i} = X_t^\ssup{j}$. 

In the discrete space case, $\pi(\bfX_t)$ is piece-wise constant, which motivates us to first study the evolution under $A^{\pi}_\rho$ for a fixed partition $\pi$ of $[n]$. Let $K_t(K_0, \pi) =e^{tA_\rho^\pi}K_0$ be the unique solution of
\begin{align}\label{eq:ODE3}
 \frac{d}{dt}(K_t(K_0, \pi)) &= A^{\pi}_\rho K_t(K_0, \pi)
\end{align}
with initial value $K_0 \in \cM(\{1,2\}^n)$.
From~\eqref{eq:ODE2} it is clear that 
$M_t^\sse{\gamma}$ can be directly constructed from $K_t$, as follows:
\begin{lemma}\label{prop:M-gamma}
Let $\tau_0:=0$ and $\tau_k:=\inf\{t\geq \tau_{k-1} : \pi(\bfX_t)\neq \pi(\bfX_{\tau_{k-1}}) \}$, $k\in\N$, be the times when the partition induced by the random walk changes. Then $M_t^\sse{\gamma}$ is given recursively  
as $M_0^\sse{\gamma} = M_0$ and
for $k \in \N_0$,
\[ M_t^\sse{\gamma} = K_{\gamma(t-\tau_k)}\left(M^\sse{\gamma}_{\tau_k},\pi(\bfX_{\tau_k})\right) , \quad t \in [\tau_k, \tau_{k+1}] .\]
\end{lemma}

\subsection{The asymptotic analysis of $K_t$}\label{ssec:asymptotics_K}

To obtain the limit as $\gamma\to\infty$, the explicit form of $M_t^\sse\gamma$ given by Lemma~\ref{prop:M-gamma} suggests that we need to understand the limit as $t\to\infty$ of $K_t(\cdot,\pi)$ first.
Surprisingly, we see that the critical curve \eqref{criticalcurve} for the moments also appears here via an independent derivation.

We begin by introducing some additional notation: For $n\in\N$ and a coloring $m\in \{1,2\}^n$, we write 
\[
\#_a m := \sum_{j=1}^n \ind_{m_j=a},\qquad a\in\{1,2\}
.\]
If in addition $\pi=\{\pi_1,\ldots,\pi_k\}$ is a partition of $[n]$ with length $|\pi|=k$, we let $m|_{\pi_i}\in \{1,2\}^{|\pi_i|}$ denote the restriction of $m$ to the block $\pi_i$, $i=1,\ldots,k$.
Moreover, given $c\in\{1,2\}^k$ we define $m^{\pi,c} \in \{1,2\}^n$ by 
\[m^{\pi,c}_j:=c_i \text{ iff }j\in \pi_i, \qquad j=1,\ldots,n.\]
\begin{prop}\label{prop:general_partition} Suppose $\rho + \cos(\pi/n) < 0$, and let $\pi$ be an arbitrary partition of $[n]$. Then there exists $K_\infty(K_0,\pi)\in \cM(\{1,2\}^n)$ such that for the unique solution $K_t(K_0, \pi)$ of~\eqref{eq:ODE3} we have
 \[ \lim_{t \to \infty} K_t(K_0,\pi) = K_\infty(K_0, \pi). \]
Moreover, the limit $K_\infty$ is given explicitly as follows: For $\rho \in (-1,0)$, we have
\begin{align}\label{eq:K-infty-explicit2}
K_\infty(K_0, \pi)(m) = \sum_{c\in \{1,2\}^{|\pi|}} K_0(m^{\pi,c})\prod_{i=1}^{|\pi|} \frac{\sin(\lambda_\rho\, \#_{c_i} m|_{\pi_i})}{\sin(\lambda_\rho |\pi_i|)},\qquad m\in\{1,2\}^n,
\end{align}
where $\lambda_\rho:=\arccos ( |\rho|)$.
For $\rho = -1$, the limit is given by 
\begin{align}\label{eq:K-infty-explicit2_minus1}
K_\infty(K_0, \pi)(m) = \sum_{c\in \{1,2\}^{|\pi|}} K_0(m^{\pi,c})\prod_{i=1}^{|\pi|} \frac{\#_{c_i} m|_{\pi_i}}{ |\pi_i|},\qquad m\in\{1,2\}^n.
\end{align}
\end{prop}

It turns out that we can reduce the case of a general partition to the case of the partition $\pi  = \{ [n]\}$.
Intuitively, the reason is that since there is no interaction between different blocks $\pi_i$ of the partition $\pi=\{\pi_1,\ldots,\pi_k\}$, the solution $K_t$ of \eqref{eq:ODE3} evolves independently on each block.
Assume for simplicity that the partition $\pi$ is given by consecutive intervals, that is
\[ \pi= \{\{1,\ldots,|\pi_1|\},\ldots,\{n-|\pi_k|+1, \ldots, n\}\}. \]

By our convention on representing the elements of $ \cM(\{1,2\}^n)$ as $2^n$-dimensional vectors, we can write 
\begin{align}\label{eq:decomposition} 
A^{\pi}_\rho = A_\rho^\sse{|\pi_1|} \oplus A_\rho^\sse{|\pi_2|} \oplus \cdots \oplus A_\rho^\sse{|\pi_k|}  .
\end{align}
Here $\oplus$ denotes the Kronecker sum for matrices (see e.g.\ \cite[Ch.\ 4.4]{HornJohnson}), 
and  $A_\rho^\sse{\ell}$, $\ell \in \N$, is the $2^\ell \times 2^\ell$ matrix with non-zero entries defined by
\begin{equation}\label{eq:defnAell} A_\rho^\sse{\ell} (m,m') = \left\{ \begin{array}{ll} \frac12 \rho \sum_{i,j=1}^\ell \1_{m_i \neq m_j}  & \mbox{if } m' = m, \\
 \frac12 \sum_{j=1}^\ell \1_{m_i \neq m_j } & \mbox{if } m' = \widehat{m}^i  \mbox{ for some } i \in [\ell	],
 \end{array} \right. 
\end{equation}
for $m, m' \in \{1,2\}^{\ell}$.
In general $\pi$ is not ordered as assumed above, so there is an additional permutation of coordinates involved, which however is only a change of basis and has no influence on the dynamics.

We start with the analysis in the simpler case. 

\begin{prop}\label{prop:simple_partition} Suppose $\rho + \cos(\pi/n) < 0$. Then there exists $K_\infty(K_0,\{[n]\})$ such that
 \[ \lim_{t \to \infty} K_t(K_0,\{[n]\}) = K_\infty(K_0, \{[n]\}). \]
Moreover, for $\rho \in (-1,0)$ the limit $K_\infty$  is given explicitly by
\begin{align}\label{eq:K-infty-explicit}
K_\infty(K_0,\{[n]\})(m) = K_0((\underline 1)_n)\frac{\sin(\lambda_\rho \,\#_1 m)}{\sin(\lambda_\rho n)} + K_0((\underline 2)_n)\frac{\sin(\lambda_\rho \,\#_2 m)}{\sin(\lambda_\rho n)},
\end{align}
where $\lambda_\rho=\arccos ( |\rho|)$ and $(\underline i)_n:=(\underbrace{i,\ldots,i}_{n})\in\{1,2\}^n$ for $i\in\{1,2\}$.
For $\rho = -1$, the limit is 
\begin{align}\label{eq:K-infty-explicit_minus1}
K_\infty(K_0,\{[n]\})(m) = K_0((\underline 1)_n)\frac{\#_1 m}{n} + K_0((\underline 2)_n)\frac{\#_2 m}{ n}.
\end{align}
\end{prop}

In order to prove Proposition~\ref{prop:simple_partition}, we start with a lemma
that analyzes the eigenvalues of the matrix $A^\sse{n}_\rho$ defined in~\eqref{eq:defnAell}.

\begin{lemma}\label{le:simple_case} The matrix $A_\rho^\sse{n}$ has  eigenvalue $0$ of multiplicity (at least) $2$, and
the remaining $2^n-2$ eigenvalues have negative (resp.\ non-positive) real part iff  $\rho + \cos (\pi/n) < 0$ (resp.\ $\leq 0$).
\end{lemma}

\begin{proof}
For ease of notation we write $A := A_\rho^\sse{n}$, also since $n \geq2$ we note that $\rho \leq 0$.
Note that the first and last rows in $A$ are zero rows, in particular $A$ has eigenvalue $0$ of multiplicity at least two. 
Moreover, if we define $A'$ by removing
the first row and column and the last row and column, then the eigenvalues of $A$ are $0$ (twice) and 
those of $A'$.

For $k \in \{1, \ldots, n-1\}$, define $\cI^k := \{ m \in \{1,2\}^n : \, \#_1m = k\}$. 
We start by looking for eigenvectors of $A$ that are constant  for coordinates $m \in \cI^k$,
in which case will write $\tilde v_k = v(m)$ for $m \in \cI^k$.

Note that for $m \in \cI^k$
\[ A(m,m') = \left\{ \begin{array}{ll} \rho k(n-k) & \mbox{if } m = m',  \\
                      \frac 12 (n-k) & \mbox{if } m' = \widehat{m}^i \mbox{ and } m_i = 1,  \\
                      \frac 12 k & \mbox{if } m' = \widehat{m}^i \mbox{ and } m_i = 2 .
                     \end{array} \right. 
                    \]
Hence, by setting $\tilde v_0 := \tilde v_n := 0$ we get for $m \in \cI^k$,  $k \in \{1, \ldots, n-1\}$ that
\begin{equation}\label{eq:tildeA}
 \begin{aligned}
A v(m)  & =   \rho k(n-k) v(m) + \sum_{i\in [n]} A(m,\widehat{m}^i) v(\widehat{m}^i) \\
   & = \rho k(n-k) \tilde v_k + \frac 12 \sum_{i \in [n]} (n-k) \1_{m_i = 1} v(\widehat{m}^i) +\frac 12\sum_{i \in [n]} k \1_{m_i =2}  v(\widehat{m}^i) \\
   & = \rho k(n-k) \tilde v_k + \frac 12 \sum_{i \in [n]} (n-k) \tilde v_{k-1} \1_{m_i = 1}  + \frac 12 \sum_{i \in [n]} k \1_{m_i =2} \tilde v_{k+1 } \\
   & = \rho k(n-k) \tilde v_k + \frac 12 k (n-k) \tilde v_{k-1}   + \frac 12(n-k) k \tilde v_{k +1 } \\
   & = \tilde A \tilde v, 
\end{aligned} 
\end{equation}
where we define $\tilde A$ as the $(n-1)\times(n-1)$ matrix with non-zero entries
\[ \tilde A(k,\ell) := \left\{ \begin{array}{ll} \frac 12 k(n-k)  & \mbox{if } \ell = k-1, \\ \rho k(n-k) & \mbox{if } \ell = k ,\\
                               \frac 12 k(n-k) & \mbox{if } \ell = k+1 ,
                              \end{array} \right. \quad k, \ell\in [n-1].
\]
Clearly, $\tilde A$ can be written as $\tilde A = D B$, where $D$ is the diagonal matrix with entries $k(n-k)$ for $k \in [n-1]$, 
and $B$ is a tridiagonal Toeplitz matrix with diagonal entries $\rho$ and off-diagonal entries $1/2$.
It is well-known (and can be checked easily) that $B$ has eigenvalues $\rho + \cos (j \pi/n)$, $j \in [n-1]$. 

It is clear that $\tilde A = DB$ has the same eigenvalues as $D^{1/2} B D^{1/2}$ (in fact these two matrices are similar 
by a diagonal change of basis), and the latter is a symmetric matrix. 
Moreover, if $\rho + \cos( \pi/n) < 0$ (resp.\ $\leq 0$), then 
$B$ is negative (semi-)definite, therefore
for any vector $w\ne0$
\[ w^T D^{1/2} B D^{1/2} w = (D^{1/2} w)^T B (D^{1/2}w) < 0 \quad (\mbox{resp.} \leq 0) . \]
In this case, since $D^{1/2}B D^{1/2}$ is symmetric, all its eigenvalues are real and negative (resp.\ non-positive).
Hence, under this condition the same is true for $\tilde A$. 
In particular, its largest eigenvalue $\la_*$ is negative (non-positive). 

Let $\tilde v^*$ be an eigenvector corresponding to the eigenvalue $\la_*$ of $\tilde A$. We will use the Perron-Frobenius theorem to argue that $\tilde v^*$ has positive coordinates. 
Observing that only the  diagonal entries of $\tilde A$ are non-positive, we can choose some suitable constant $c>0$ such that $\tilde A^* := \tilde A + cI_{n-1}$ is a non-negative matrix that is irreducible and aperiodic. In particular, the Perron-Frobenius theorem applies to $\tilde A^*$. 
Moreover, w.l.o.g. the constant $c>0$ can be taken so large that the spectrum of $\tilde A^*$ is contained in $[0,\infty)$.
Then since $\lambda_*$ is the largest eigenvalue of $\tilde A$, we have that $c+\lambda_*$ is the largest eigenvalue and spectral radius of $\tilde A^*$.
Since $\tilde v^*$ is a corresponding eigenvector,
by Perron-Frobenius it must have positive coordinates. 

Coming back to the matrix $A'$ obtained from $A$ by removing the first row and column and the last row and column, we can define a $(2^n-2)$-dimensional vector $v'$ by setting $v'(m) := \tilde v^*_k$ if $m \in \cI^k$, $k \in [n-1]$.
Then by definition of $A'$ and by~\eqref{eq:tildeA}, we see that $v'$ is an eigenvector for $A'$ with 
eigenvalue $\la_*$.  
Thus it is also an eigenvector for $A^*:= A' + cI_{2^n-2}$ with eigenvalue $c+\la_*$, where $c>0$ is chosen as above. Since Perron-Frobenius applies to $A^*$ and the vector $v'$ is strictly positive, $c+\la_*$ must also be the spectral radius of $A^*$. Now let $\lambda\in\C$ be any eigenvalue of $A'$. Then $c+\lambda$ is an eigenvalue of $A^*$, and we get
\[\mathrm{Re}(\lambda)=\mathrm{Re}(c+\lambda)-c\le|c+\lambda|-c\le (c+\la_*) - c=\la_*.\]
Finally, since the eigenvalues of $A$ are $0$ (with multiplicity $2$) together with those of $A'$, we have
proved the statement of the lemma.
\end{proof}

We will now prove Proposition~\ref{prop:simple_partition} by identifying the eigenvectors 
of $A_\rho^\sse{n}$ corresponding to eigenvalue $0$.
For $k \in \{0,\ldots,n\}$, we consider again $\cI^k := \{ m \in \{1,2\}^n : \, \#_1m = k\}$. 
Then we define
two vectors $v^\ssup{n}_j = (v_j^\ssup{n}(m))_{m \in \{1,2\}^n}\in\bR^{2^n}$ $(j=1,2)$ as follows:
If $\rho = -1$, we set
\begin{equation}\label{eq:eigvec1} v_1^\ssup{n} (m) := \frac{k}{n} 
\quad \mbox{and}\quad v_2^\ssup{n} (m) := \frac{n-k}{n} \quad \mbox{ if } m \in \cI^k, \end{equation}
while for $\rho>-1$,  $\rho + \cos (\pi/n) < 0$ we set
\begin{equation}\label{eq:eigvec2} v_1^\ssup{n} (m) := \frac{\sin ( \la_\rho k )}{\sin (\la_\rho n)} \quad
\mbox{and}\quad v_2^\ssup{n} (m) := \frac{\sin ( \la_\rho (n-k) )}{\sin (\la_\rho n)}\quad  \mbox{ if } m \in \cI^k,  \end{equation}
where $\la_\rho := \arccos(|\rho|)$. Observe that by construction we have $v_1^\ssup{n}((\underline 1)_n) = v_2^\ssup{n}((\underline 2)_n) = 1$ and $v_2^\ssup{n}((\underline 1)_n) = v_1^\ssup{n}((\underline 2)_n) = 0$. Also note that the explicit form of the limit claimed in Proposition \ref{prop:simple_partition} can be restated as 
\begin{equation}\label{eq:limiting_matrix}
\lim_{t\to\infty } e^{tA^\sse{n}_\rho}=\left(v_1^\ssup{n},0,\ldots,0,v_2^\ssup{n}\right),
\end{equation}
where the RHS is a matrix consisting of the vectors $v_1^\ssup{n}$ resp.\ $v_2^\ssup{n}$ in the first resp.\ last column and zeros otherwise.

\begin{proof}[Proof of Proposition~\ref{prop:simple_partition}]
We will prove the proposition by showing that 
$v_1^\ssup{n}$ and $v_2^\ssup{n}$ are eigenvectors of $A_\rho^\sse{n}$ corresponding to eigenvalue $0$,
and also that $K_t := K_t(K_0, \{ [n]\})$ satisfies
\begin{equation}\label{eq:limit_simple} \lim_{t \ra \infty} K_t = K_0((\underline 1)_n) v_1^\ssup{n} + K_0((\underline 2)_n) v_2^\ssup{n} . \end{equation}

Let us assume that we have already shown that $ v_1^\ssup{n}$ and $ v_2^\ssup{n}$ are eigenvectors 
of $A := A^\sse{n}_\rho$ corresponding to eigenvalue $0$, so that in particular $e^{tA}v^\ssup{n}_{j}=v^\ssup{n}_{j}$, $j=1,2$.
Then, we first show that in this case~\eqref{eq:limit_simple} holds.

We let $A'$ be defined as in the proof of Lemma~\ref{le:simple_case}. Then, since by Lemma \ref{le:simple_case} all its eigenvalues have negative real part, $e^{tA'}\to0$ as $t\to\infty$. 
Let $w_3, \ldots, w_{2^{n}}$ be an orthogonal basis for $\R^{2^n-2}$. 
We define $v_k^\ssup{n}$ for $k\geq 3$ by extending the vector $w_k$ at either side by zero entries to obtain a $2^n$-dimensional vector
(so in particular $v^\ssup{n}_k((\underline 1)_n) = v^\ssup{n}_k((\underline 2)_n) = 0$, $k\ge3$). 
Together with the vectors $v^\ssup{n}_1,v^\ssup{n}_2$ this forms a basis of $\bR^{2^n}$. 
Hence, we write $K_0 = \sum_{i} \mu_i v_i^\ssup{n}$ for some coefficients $\mu_i$.
Furthermore, $A$ acts on $v^\ssup{n}_k$ as $A'$ acts on $w_k$, $k=3,...,2^n$, in the corresponding bases:
\[ \langle Av^\ssup{n}_k,v^\ssup{n}_j \rangle = \langle A'w_k,w_j\rangle,\quad j,k=3,...,2^n.\]
Therefore $\lim_{t\to\infty}e^{tA}v^\ssup{n}_k=0$ for $k\geq 3$ and 
\[ e^{tA } K_0 = \sum_i \mu_i e^{tA} v_i^\ssup{n}  \ra \mu_1 v_1^\ssup{n} + \mu_2 v_2^\ssup{n} \]
as $t\to\infty$. Moreover, notice that
since $v_1^\ssup{n}$ resp.\ $v_2^\ssup{n}$ is the only basis element with a non-zero entry in the first resp.\ last coordinate, we must have $\mu_1 = K_0((\underline 1)_n )$
and $\mu_2 = K_0((\underline 2)_n)$.

Finally, it remains to show that $v_1^\ssup{n}$ and $v_2^\ssup{n}$ are eigenvectors with eigenvalue $0$. 
This can either be verified by direct calculation or alternatively it can be derived as follows: 
In order to obtain $v_1^\ssup{n}$,
we look for an eigenvector $v$ that is constant for coordinates $m \in \cI^k$ 
and such that $v((\underline 1)_n) = 1$ and $v((\underline 2)_n) =0$. 
As in the proof of Lemma~\ref{le:simple_case}, we will write $\tilde v_k = v(m)$ for $m \in \cI^k$, and
$Av = 0$ reduces to 
the recurrence equation 
\[ 0  = \rho k(n-k) \tilde v_k +  \frac 12 k (n-k) \tilde v_{k-1}   + \frac 12 (n-k) k \tilde v_{k +1} , \]
see \eqref{eq:tildeA}.
Assuming that $\rho \in (-1,0)$, this equation has solution $\tilde v_k = c_1 t_1^k + c_2 t_2^k$, where $t_1$ and $t_2$ are the distinct roots of the equation
\begin{equation}\label{eq:char_eqn} t^2 + 2 \rho t + 1= 0 . \end{equation}
Thus since $\rho < 0$, $t_{1,2} = - \rho \pm i \sqrt{ 1 - \rho^2} = e^{\pm \la_\rho i}$, where $\la_\rho := \arctan ( \frac{\sqrt{1-\rho^2}}{|\rho|} ) = 
\arccos ( |\rho|)$. Simplifying further, we note that since $\tilde v_0 = 0$, we have $c_1 = - c_2$ and thus
\[ \tilde v_k = c_1 \, 2i \sin( \la_\rho k ). \]
Therefore, since $\tilde v_n = 1$, we get $c_1 = \frac{1}{2 i \sin (\la_\rho n)}$, so that
\[ \tilde v_k = \frac{\sin ( \la_\rho k )}{\sin (\la_\rho n)}. \]
Moreover, we note that this expression is positive for $k \in \{1, \ldots, n-1\}$, since
$\la_\rho n < \pi$ iff $ \rho +  \cos \pi/n < 0$.
 
In the case that $\rho = -1$, the characteristic equation~\eqref{eq:char_eqn} has the unique solution $t = 1$ and 
so the solution to the recurrence equation is given by $\tilde v_k = c_1  + c_2 k$ for suitable constants $c_1,c_2$, which by using the 
boundary conditions reduces to $\tilde v_k = k/n$. 

In either case, defining $v(m) := \tilde v_k$ for $m \in \cI^k$ gives the first eigenvector $v^\ssup{n}_1$ with eigenvalue $0$. 
The derivation of the second eigenvector $v^\ssup{n}_2$ is analogous.
\end{proof}

Using the decomposition~\eqref{eq:decomposition}, we can now use
Proposition~\ref{prop:simple_partition} to analyze the dynamics of $K_t$
in the case of a general partition.
                    
\begin{proof}[Proof of Proposition~\ref{prop:general_partition}]
Throughout we write $K_t = K_t(K_0, \pi)$. We first assume that $\pi = \left\{\{1, \ldots, n_1\},\{n_1 + 1,\ldots, n_2\}, \ldots, \{n_1+\ldots+ n_{k-1}+1, \ldots, n_k \}\right\}$
with $\sum_{i=1}^k n_i = n$.
Then, by the decomposition~\eqref{eq:decomposition} and the fact that $e^{t(A\oplus B)}=e^{tA}\otimes e^{tB}$, 
we have that 
\[ K_t = e^{t A_\rho^\pi} K_0 = \bigotimes_{i=1}^k e^{ tA_\rho^\sse{n_i} } K_0 . \]
Therefore, since $\rho + \cos (\pi/n) < 0$ implies $\rho + \cos( \pi/n_i)<0$ for all $i \in [k]$, 
we obtain from Proposition~\ref{prop:simple_partition} (see also \eqref{eq:limiting_matrix}) that
\begin{equation}\label{eq:limit} \begin{aligned} \lim_{t \ra \infty} K_t & = \lim_{t \ra \infty} e^{tA_\rho^\pi }K_0 = 
 \bigotimes_{i=1}^k \Big( v_1^\ssup{n_{i}}, 0, \ldots, 0, v_2^\ssup{n_{i}} \Big) K_0 \\
 & = 
 \sum_{c_1, \ldots , c_k \in \{1,2\}} K_0 \left( (\underline c_1)_{n_1},\ldots,  (\underline c_k)_{n_k}\right)\, v_{c_1}^\ssup{n_1} \otimes \cdots \otimes v_{c_k}^\ssup{n_k} , 
\end{aligned} \end{equation}
where we write
\[ (\underline c_1)_{n_1}\ldots  (\underline c_k)_{n_k} := ( \underbrace{c_1, \ldots, c_1}_{n_1 \text{ times}}, \ldots, \underbrace{c_k, \ldots, c_k}_{n_k \text{ times}} )\in\{1,2\}^n. \]
In view of the definition of the eigenvectors $v_1^\ssup{n_i}$ and $v_2^\ssup{n_i}$ (see~\eqref{eq:eigvec1} and~\eqref{eq:eigvec2}), 
this concludes the proof of Proposition~\ref{prop:general_partition} for the case of a partition consisting of consecutive intervals.

The case of general $\pi = \{\pi_1, \ldots, \pi_k\}$ (written in increasing order of smallest element) follows by permuting the entries in $m \in \{1,2\}^n$. 
\end{proof}

\subsection{Proof of Theorem~\ref{thm:M_t-infty} in the discrete case}\label{ssec:M_infinty_discrete}

In discrete space, Theorem~\ref{thm:M_t-infty} is a direct consequence 
of the following proposition.

\begin{prop}\label{prop:thm_M_infty_discrete}
Suppose $\rho + \cos(\pi/n) < 0$. 
Then, a.s.\ with respect to the random walk $\bfX=(X^\ssup{1},\ldots,X^\ssup{n})$, for any 
$t \geq 0$, 
$M_t^\sse\gamma$ converges to a limit $M_t^\sse\infty$ as $\gamma\to\infty$.
The limit $M^\sse{\infty}_t$ is given recursively as follows: 
Let $\tau_k$ be defined as in Lemma~\ref{prop:M-gamma},
i.e.\ $\tau_0 := 0$ and $\tau_k$, $k \geq 1$
are the times when the partion $\pi(\bfX_t)$ induced by the random walk changes.
Then $M^\sse{\infty}_t=K_\infty(M_0,\pi(\bfX_0))$ for $t\in[0,\tau_1]$ and
\[ M_t^\sse{\infty} = K_\infty\left(M_{\tau_k}^\sse{\infty}, \pi(\bfX_{\tau_k})\right)  \quad \mbox{for all } t\in (\tau_k, \tau_{k+1}],  \]
where $K_\infty(K_0, \pi)$ is given in Proposition~\ref{prop:general_partition}.
\end{prop}

\begin{proof}[Proof of Proposition~\ref{prop:thm_M_infty_discrete}]
We know that almost surely, the sequence $(\tau_k)_{k \geq 0}$ is strictly increasing. 
Also, by Lemma~\ref{prop:M-gamma} we know that the process $M_t^\sse\gamma$ is given by
\[ M_t^\sse{\gamma} = K_{\gamma(t-\tau_k)}\left(M^\sse{\gamma}_{\tau_k},\pi(\bfX_{\tau_k})\right) , \quad t \in [\tau_k, \tau_{k+1}] .\]
Then by Proposition~\ref{prop:general_partition}, we can argue inductively to see that for each $k \in \N_0$, as $\gamma \ra \infty$
\[ M_t^\sse{\gamma} \ra K_\infty(M_{\tau_k}^\sse\infty, \pi(\bfX_{\tau_k})) , \quad t\in (\tau_k, \tau_{k+1}]  . \qedhere\]
\end{proof}

\subsection{Proof of Theorem~\ref{thm:M_t-infty} in the continuous case}\label{ssec:continuous}

The continuous space case is slightly different from the discrete space case. However,
the assumption that the $n$-dimensional Brownian motion $\bfX = (X_t^\ssup{1}, \ldots, X_t^\ssup{n})_{t \geq 0}$
starts in $\mathbf{x}	= (x_1, \ldots, x_n)$ with $x_i \neq x_j$ for $i \neq j$
simplifies the situation somewhat, since we only have intersections of at most two Brownian motions
simultaneously.  
In order to formalize this, we need to consider the times of \emph{new} pair collisions
between the Brownian motions.
More precisely, as in Section \ref{ssec:discret_M_gamma} we write 
$\pi(\bfX_t)$ for the partition of $[n]$ induced by the equivalence relation $i \sim j$ iff $X_t^\ssup{i} = X_t^\ssup{j}$.
Then we define $\tau_0 := 0$ and $\tau_{k+1}$ as the first time $t\geq\tau_k$ so that the partition $\pi(\bfX_t)$ is not a refinement of the partition $\pi(\bfX_{\tau_k})$. The partition $\pi(\bfX_{\tau_k})$ contains for $k\geq 1$ a single pair, and otherwise singletons. Denote the pair by $\kappa_{k}=\{\kappa_{k,1},\kappa_{k,2}\}$. The only refinements of $\pi(X_{\tau_k})$ are the partition itself and the partition consisting only of singletons. So $\tau_{k+1}$ is the first time when there is a new collision with a collision pair $\kappa_{k+1}$ not identical to $\kappa_{k}$.
By the path properties of Brownian motion, almost surely we have $\tau_0 < \tau_1 < \cdots$.

With this notation, we can directly deduce Theorem~\ref{thm:M_t-infty} in the continuous space setting
from the following proposition.

\begin{prop}\label{prop:thm_M_infty_continuous}
Let $X = (X^\ssup{1}, \ldots, X^\ssup{n})$ be a Brownian motion started in $\mathbf{x}\in\R^n$
such that $x_i \neq x_j$ for $i \neq j$,
and consider $M^\sse\gamma := M^\sse\gamma(\bfX,M_0)$ for $M_0 \in\calM( \{1,2\}^n)$. Suppose $\rho + \cos(\pi/n) < 0$. 
Then, a.s.\ with respect to $X$, for any 
$t \geq 0$, 
$M_t^\sse\gamma$ converges to a limit $M_t^\sse\infty$.
The limit $M^\sse{\infty}_t$ is given recursively as 
$M^\sse{\infty}_t=M_0$ for $t\in[0,\tau_1]$ and
\[ M_t^\sse{\infty} = K_\infty\left(M_{\tau_k}^\sse{\infty}, \pi(\bfX_{\tau_k})\right) , \quad \mbox{for all } t\in (\tau_k, \tau_{k+1}],  \]
where $K_\infty(K_0, \pi)$ is given in Proposition~\ref{prop:general_partition}.
\end{prop}

\begin{proof}
We recall the equation~\eqref{eq:ODE1} for $M^\sse\gamma$,
\begin{align}\label{eq:ODE1_repeated}
 dM^\sse\gamma_t(m) &= \frac{\gamma\rho}{2} \sum_{i,j=1}^n \ind_{m_i\neq m_j} M^\sse\gamma_t(m)\,dL_t^{i,j} + \frac{\gamma}{2}
 \sum_{i,j=1}^n \ind_{m_i\neq m_j} M^\sse\gamma_t(\widehat{m}^i)\,dL_t^{i,j}. 
\end{align}
We notice that $M_t^\sse{\gamma}$ is constant equal to $M_0$ on $[0,\tau_1]$.
For $k \geq 1$, if we write
\[ L_t^\ssup{\kappa_k} := L_t^{\kappa_{k,1},\, \kappa_{k,2}} , \]
then on the interval $[\tau_k, \tau_{k+1}]$
equation~\eqref{eq:ODE1_repeated} simplifies to
\begin{align*}  dM^\sse\gamma_t(m) &= {\gamma\rho} \ind_{m_{\kappa_{k,1}}\neq m_{\kappa_{k,2}}} 
 M^\sse\gamma_t(m)\,dL_t^\ssup{\kappa_k}\\
 & \qquad + \frac{\gamma}{2} 
 \ind_{m_{\kappa_{k,1}}\neq m_{\kappa_{k,2}}} \left( M^\sse\gamma_t(\widehat{m}^{\kappa_{k,1}} ) + M^\sse\gamma_t(\widehat{m}^{\kappa_{k,2}} ) \right) dL_t^\ssup{\kappa_k}. 
 \end{align*}
Hence, the analogon of Lemma~\ref{prop:M-gamma}
 in the continuous space setting
is that $M^\sse{\gamma}$ solves 
 $M_t^\sse{\gamma} = M_0$ for $t \in [0,\tau_1]$ and
for $k \in \N$,
\begin{equation}\label{eq:representation_cts} M_t^\sse{\gamma} = K_{\gamma ( L_t^\ssup{\kappa_k} - L_{\tau_k}^\ssup{\kappa_k})}\left(M^\sse{\gamma}_{\tau_k},\pi(\bfX_{\tau_k})\right) , \quad t \in [\tau_k, \tau_{k+1}] ,\end{equation}
where $K_t$ is defined by~\eqref{eq:ODE3} and $\pi(X_{\tau_k})$ is the partition of $[n]$
consisting only of singletons, except for the pair $\kappa_k=\{ \kappa_{k,1}, \kappa_{k,2}\}$. 

Finally, we argue as in the proof of Proposition~\ref{prop:thm_M_infty_discrete}, noting that by the path properties of Brownian 
motion, almost surely, for any $t > \tau_k$, we have $L_t^\ssup{\kappa_k} - L_{\tau_k}^\ssup{\kappa_k} > 0$ .
\end{proof}

\subsection{Proof of Theorem~\ref{thm:moment_duality_infinite}}\label{ssec:proof_duality}

In order to show Theorem~\ref{thm:moment_duality_infinite} we use a 
dominated convergence argument, so that the main 
technical point is to show that the processes $M_t^\sse\gamma$ are uniformly bounded in $\gamma$.

\begin{prop}\label{prop:M_bounded} Assume $\rho+\cos(\pi/n)<0$. Fix $M_0\in\cM(\{1,2\}^n)$. In both the discrete and the continuous space setting, the process
$M^\sse\gamma = ( M_t^\sse\gamma(\bfX, M_0))_{t \geq 0}$ remains bounded uniformly in $\gamma$, $t$ and $\om$.
\end{prop}

We preceed the proof by three lemmas showing that the process $(K_t)_{t\ge0}$ defined in~\eqref{eq:ODE3}
remains bounded. 
For this, the crucial observation is the following (which we prove in the next two lemmas):
The vectors $v_1^\ssup{n}$ and $v_2^\ssup{n}$ from \eqref{eq:eigvec1}-\eqref{eq:eigvec2}, which we recall are eigenvectors with eigenvalue $0$ of the matrix $A_\rho^\sse{n}$ from \eqref{eq:defnAell}, corresponding to the trivial partition $\pi=\{[n]\}$, are in fact also eigenvectors with eigenvalue $0$ of $A_\rho^\pi$, for any partition $\pi$ of $[n]$.
\begin{lemma}\label{lemma:eigenvector-partition-simple}
Assume $\rho+\cos(\pi/n)<0$. Then, for any partition $\pi=\{\pi_1,\ldots,\pi_k\}$ of $[n]$ consisting purely of singletons except for one block $\pi_i$, $i\in\{1,\ldots,k\}$, we have
\begin{align}
K_\infty(v^\ssup{n}_j,\pi)=v^\ssup{n}_j,\qquad j=1,2.
\end{align}
\end{lemma}
\begin{proof}
Assume first $\rho>-1$. 
W.l.o.g.\ we can assume that $i=1$ and $\pi_1=[\ell]$, $1\le\ell<n$.
Let $m\in\{1,2\}^n$. Decomposing $m=(m',m'')$ with $m'\in\{1,2\}^\ell$ and $m''\in \{1,2\}^{n-\ell}$, we get by Proposition \ref{prop:general_partition} that for $j=1,2$
\begin{align*}
K_\infty(v_j^\ssup{n}, \pi)(m) & = \sum_{c\in \{1,2\}^{|\pi|}} v_j^\ssup{n}(m^{\pi,c})\prod_{i=1}^{|\pi|} \frac{\sin(\lambda_\rho \,\#_{c_i} m|_{\pi_i})}{\sin(\lambda_\rho |\pi_i|)}
\\
&= v_j^\ssup{n}\left((\underline 1)_\ell, m''\right)  \frac{\sin(\lambda_\rho \#_1 m')}{\sin(\lambda_\rho \ell)} + v_j^\ssup{n}\left((\underline 2)_\ell, m''\right) \frac{\sin(\lambda_\rho \#_2 m')}{\sin(\lambda_\rho \ell)},
\end{align*}
where we used that all blocks $\pi_i$, $i\ge2$, of the partition $\pi$ consist of singletons. 
For $j=1$, plugging in the definition of $v_1^\ssup{n}$ from \eqref{eq:eigvec2} we obtain that the above is equal to
\begin{align*}
 \frac{\sin(\lambda_\rho \#_1 m')}{\sin(\lambda_\rho \ell)} \frac{\sin(\lambda_\rho (\ell + \#_1 m''))}{\sin(\lambda_\rho n)} + \frac{\sin(\lambda_\rho (\ell-\#_1m'))}{\sin(\lambda_\rho \ell)} \frac{\sin(\lambda_\rho \#_1 m'')}{\sin(\lambda_\rho n)}.
\end{align*}
Now the claim is that the above equals $\frac{\sin(\lambda_\rho (\#_1m' + \#_1 m''))}{\sin(\lambda_\rho n)}$. After multiplying everything with $\sin(\lambda_\rho \ell)\sin(\lambda_\rho n)$, this follows directly from the addition and subtraction theorems for the sine function applied to the three sines with a sum or difference as argument. The proof for the second eigenvector $v_2^\ssup{n}$ is analogous.

For $\rho=-1$ the strategy is the same, but because of the different structure of the vectors $v_j^\ssup{n}$, the function $x\mapsto\sin(\lambda_\rho x)$ needs to be replaced by the identity.
\end{proof}

\begin{lemma}\label{lemma:eigenvector-partition}
Assume $\rho+\cos(\pi/n)<0$. Then for any partition $\pi=\{\pi_1,\ldots,\pi_k\}$ of $[n]$  we have
$K_\infty(v^\ssup{n}_j,\pi)=v^\ssup{n}_j$, $j=1,2$. 
\end{lemma}
\begin{proof}
The linear map $K_\infty(\cdot, \pi)$ is the Kronecker product of the $K_\infty(\cdot, \pi_i)$, $i=1,\ldots,k$, up to relabeling of the coordinates. This implies that
\begin{align*}
K_\infty(\cdot, \pi) = K_\infty(\cdot, \tilde\pi_1) \circ K_\infty(\cdot, \tilde\pi_2) \circ \cdots \circ K_\infty(\cdot, \tilde\pi_k),
\end{align*}
where $\tilde \pi_i$ is the partition of $[n]$ which consists of $\pi_i$ and otherwise singletons. The claim then follows from Lemma \ref{lemma:eigenvector-partition-simple}.
\end{proof}

\begin{lemma}\label{le:K_uniform_bound}
Assume $\rho+\cos(\pi/n)<0$. 
Given $a_1,a_2\geq 0$, define $V_n^{a_1,a_2} \in \cM(\{1,2\}^n)$ via $V_n^{a_1,a_2}:=a_1 v_1^\ssup{n} + a_2 v_2^\ssup{n}$. 
Then for any partition $\pi=\{\pi_1,\ldots,\pi_k\}$ of $[n]$ and any $K_0\in \cM(\{1,2\}^n)$ such that $0\leq K_0 \leq V_n^{a_1,a_2}$ coordinate wise, we have
\begin{align*}
K_t(K_0,\pi) \leq V_n^{a_1,a_2},\qquad t>0.
\end{align*}
\end{lemma}
\begin{proof}
Write $V_n:=V_n^{a_1,a_2}$ for short. By Lemma \ref{lemma:eigenvector-partition} we know that $A^\pi_\rho V_n=0$. 
Let $0\leq W \leq V_n$ coordinate wise, with $W(m)=V_n(m)$ for some fixed $m\in\{1,2\}^n$. Since $A_\rho^\pi=D+B$ is the sum of a diagonal matrix $D$ and a non-negative matrix $B$, we have
\begin{align*}
(A^\rho_\pi W)(m) = D(m,m) W(m) + (B W)(m) \leq D(m,m) V_n(m) +(B V_n)(m) = (A^\rho_\pi V_n)(m)=0.
\end{align*}
This implies that whenever $K_0\leq V_n$, the solution of \eqref{eq:ODE3} satisfies $K_t(K_0,\pi)(m)\leq V_n(m)$ for all $t>0$ and $m\in\{1,2\}^n$, since the derivative is always non-positive when $K_t(K_0,\pi)=V_n(m)$.
\end{proof}

\begin{proof}[Proof of Proposition~\ref{prop:M_bounded}]
Given $M_0\in\cM(\{1,2\}^n)$, we can always find $a_1,a_2\ge0$ such that $M_0\le a_1 v_1^\ssup{n} + a_2 v_2^\ssup{n}=V_n^{a_1,a_2}$ holds coordinate wise. 
Then we can combine Lemma~\ref{le:K_uniform_bound} with the representation of $M_t^\sse\gamma$ 
in terms of $K_t$, see Lemma~\ref{prop:M-gamma} (discrete space) and~\eqref{eq:representation_cts} (continuous space),
to get the bound $M_t^\sse{\gamma}\le V_n^{a_1,a_2}$ inductively on each time interval $[\tau_k,\tau_{k+1}]$, $k\in\N$.
\end{proof}

Equipped with these uniform bounds, we can now prove the moment duality for the infinite rate
models 
by combining Theorem~\ref{thm:M_t-infty} with a dominated convergence argument. 

\begin{proof}[Proof of Theorem~\ref{thm:moment_duality_infinite}]
We give the proof for the continuous-space case $\calS=\R$.
Suppose first that $\phi=\otimes_{i=1}^n\phi_i$ with $\phi_i\in\cC_c^+(\R)$, $i=1,\ldots,n$.
Rewriting the moment duality for finite $\gamma>0$ in weak form, by Theorem~\ref{thm:rewrite} and Fubini we have for all $t>0$ that
\begin{align}
\bE_{\bfu_0}\bigg[\int_{\R^n}\otimes_{i=1}^n\phi_i(\bfx)\, H(\bfu^\sse{\gamma}_t;\bfx,c)\,d\bfx\bigg] = \int_{\R^n}\otimes_{i=1}^n\phi_i(\bfx)\, \bE_{\bfx}\bigg[\sum_{b\in\{1,2\}^n} M_t^\sse{\gamma}(\bfX,\delta_c)(b)\, \bfu_0^\ssup{b}(\bfX_t) \bigg]d\bfx.
\end{align}
By Theorem~\ref{thm:M_t-infty}, we know that for Lebesgue-a.e.\ $\bfx\in\R^n$, almost surely under $\p_\bfx$, for any $t \geq 0$, $M^\sse\gamma_t(\bfX,\delta_c)$ converges 
to $M^\sse\infty_t(\bfX,\delta_c)$. Since $\bfu_0$ is bounded by assumption and $M^\sse\gamma_t$ is bounded
uniformly in $\gamma$ (and $\om$) by Proposition~\ref{prop:M_bounded}, we can deduce by dominated convergence that the right hand side converges to
\[  \int_{\R^n}\otimes_{i=1}^n\phi_i(\bfx)\, \bE_{\bfx}\bigg[\sum_{b\in\{1,2\}^n} M_t^\sse\infty(X,\delta_c)(b)\, \bfu_0^\ssup{b}(\bfX_t) \bigg] d\bfx , \qquad t>0.\]
For the left hand side, we know that $(\bfu_t^\sse\gamma)_{t\ge0}\to (\bfu_t)_{t\ge0}$ as $\gamma\uparrow\infty$ in $D_{[0,\infty)}(\calM_\tem(\R)^2)$ w.r.t.\ the Meyer-Zheng topology. 
By \cite[Thm.\ 5]{MZ84} (see also \cite[Cor.\ 1.4]{Kurtz91}), we may and do assume that there is a sequence $\gamma_k\uparrow\infty$ and a subset $I\subseteq\R^+$ of full Lebesgue measure such that for all $t\in I$ we have $\bfu_t^\sse{\gamma_k}\to\bfu_t$ as $k\to\infty$ in $\calM_\tem(\R)^2$, almost surely. 
Recalling the definition of the topology on $\cM_\tem(\R)$ (see e.g.\ Appendix A.1 in \cite{BHO15}), we then have also
\begin{equation}\label{proof:moment_duality_infinite}
\int_{\R^n}\otimes_{i=1}^n\phi_i(\bfx)\, H(\bfu^\sse{\gamma_k}_t;\bfx,c)\,d\bfx\xrightarrow{k\to\infty}\prod_{i=1}^n\langle u_t^\ssup{c_i},\phi_i\rangle
=\int_{\R^n}\otimes_{i=1}^n\phi_i(\bfx)\, \bfu_t^\ssup{c}(d\bfx),\end{equation}
$t\in I$, almost surely. Further, note that since $\rho+ \cos(\pi/n)<0$ we are strictly below 
the critical curve defined in~\eqref{criticalcurve}. In particular, we can find $\eps> 0$ such that the $(n+\eps)$-th moments
of $\mathrm{cSBM}(\rho,\gamma)_{\bfu_0}$ remain bounded uniformly in $\gamma$, cf.\ \cite[Cor.\ 3.8]{HO15}.
By uniform integrability, we conclude that the convergence \eqref{proof:moment_duality_infinite} holds also in expectation, 
proving the moment duality for all $t\in I$ and test functions of the considered form. 
In order to extend it to all $t>0$, choose a sequence $t_n\in I$ with $t_n\downarrow t$ and use the right-continuity of the paths of $(\bfu_t)_{t\ge0}$.
The extension to arbitrary $\phi\in L^1(\R^n)$ follows by a standard density argument.

The proof for the discrete-space case $\calS=\Z^d$ is analogous. In fact it is even simpler, since we can avoid using test functions and instead argue pointwise.
\end{proof}

\subsection{Proof of Corollary \ref{cor:explicit-moment-computations} (ii)}
\begin{proof}
Since $A^{\{\{1\},\{2\}\}}_\rho=0$ and hence $K_t(K_0,{\{\{1\},\{2\}\}})\equiv K_0$, we obtain from Lemma \ref{prop:M-gamma} that
\begin{align}\label{eq:M_t-2-simple}
M_t^\sse\gamma=K_{\gamma L_t^{1,2}}\left(M_0,\{\{1,2\}\}\right) , \qquad t\ge0.
\end{align}
Therefore, by the reformulation of the finite rate moment duality (Theorem \ref{thm:rewrite}) we have
\begin{align}
\bE_{\bfu_0}\left[ u_t^\ssup{1}(x)u_t^\ssup{2}(x) \right]
&= \bE_{(x,x)}\left[ \sum_{m\in \{1,2\}^2}  M_t^\sse{\gamma}\left(\bfX,\delta_{(1,2)}\right)(m)\, \bfu_0^\ssup{m}(\bfX_t) \right]\\
&= \bE_{(x,x)} \left[\sum_{m\in \{1,2\}^2} K_{\gamma L_t^{1,2}}\left(\delta_{(1,2)},\{\{1,2\}\}\right)(m) \,\bfu_0^\ssup{m}(\bfX_t)\right]
\end{align}
for all $x\in\Z^d$. By the transience of $Y_t:=X_t^\ssup{1}-X_t^\ssup{2}$, there is a geometric number of returns of $(Y_t)_{t\ge0}$ to $0$ with parameter $1-p_d$, where $p_d<1$ is the return probability of a random walk in $\bZ^d$. Whenever $Y_t$ arrives in $0$ it waits an exponentially distributed time with parameter 2 (since $X_t^\ssup{1}$ or $X_t^\ssup{2}$ could jump). Therefore $L_\infty^{1,2}:=\lim_{t\to\infty}L_t^{1,2}$ is a geometric number of independent exponentials, which is an exponential with parameter $2(1-p_d)$. 
Observing that
\begin{equation}
A^{\{\{1,2\}\}}_\rho=A^{[2]}_\rho=\left(\begin{matrix}
0&0&0&0\\
\tfrac{1}{2}&\rho&0&\tfrac{1}{2}\\
\tfrac{1}{2}&0&\rho&\tfrac{1}{2}\\
0&0&0&0
\end{matrix}\right),
\end{equation}
we obtain that $\rho$ is an eigenvalue of $A^{\{\{1,2\}\}}_\rho$ and $\delta_{(1,2)}$ a corresponding eigenvector. Hence $K_t\left(\delta_{(1,2)},\{\{1,2\}\}\right)=e^{\rho t}\,\delta_{(1,2)}$. Together with the description of the asymptotic local time as an exponential random variable $Z$ with parameter $2(1-p_d)$,
we obtain
\begin{align} 
\bE_{\bfu_0}\left[ u_t^\ssup{1}(x)u_t^\ssup{2}(x)\right] &\leq \|u_0^\ssup{1}\|_\infty\|u_0^\ssup{2}\|_\infty\,\bE\left[ e^{\rho\gamma Z}\right]\\
&=\begin{cases}
\|u_0^\ssup{1}\|_\infty\|u_0^\ssup{2}\|_\infty\left(1-\frac{\rho\gamma}{2(1-p_d)}\right)^{-1}, \quad &\frac{\rho\gamma}{2(1-p_d)}<1\\
\infty,\quad&\frac{\rho\gamma}{2(1-p_d)}\geq1 ,
\end{cases}
\end{align}
with the first inequality being asymptotically, as $t\to\infty$, an equality if $\bfu_0=(\1,\1)$.
\end{proof}

\section{The analysis of ${\rm SBM}(-1,\infty)$}\label{sec:rhominus1}

In this section, we prove all the results for the case $\rho = -1$. 
In particular, in Section~\ref{ssec:tightness} we show Theorem \ref{thm:main1}, establishing the existence of an infinite rate limit for $\rho=-1$ and general initial conditions which was so far open in both discrete and continuous space.
The limit is uniquely characterized by the moment duality \eqref{eq:moment duality infinite gamma} and the (strong) Markov property.
In the subsequent subsections, we move on to a more explicit description of the limit. 
In Section~\ref{ssec:notation}, we collect some preliminary lemmas which will be needed in the sequel.
We then show in Section~\ref{ssec:nonproliferation} that the number of interfaces is non-increasing. 
Then, we prove Theorem~\ref{thm:annihilating-BM-_onepoint}, the explicit characterization of the limit
in the single interface type,  in Section~\ref{ssec:single_interface}. The extension to at most countably many interfaces without accumulation point, 
Theorem~\ref{thm:annihilating-BM} is in Section~\ref{ssec:multiple_interfaces}, while
the proof for general initial conditions, Theorem~\ref{thm:overlapping} is in Section~\ref{ssec:overlapping}.

Recall that we write $w_t:=u_t^\ssup{1}+u_t^\ssup{2}$ and that since $\rho=-1$, the evolution of $(w_t)_{t\ge0}$ is governed by the heat equation. Thus denoting by $(S_t)_{t\ge0}$ the heat semi-group, we have that $w_t=S_tw_0$ and is deterministic, a fact we use frequently.

\subsection{Proof of Theorem~\ref{thm:main1}}\label{ssec:tightness}

By \cite[Prop.\ 3.8]{BHO15} resp.\ \cite[Prop.\ 3.1]{HO15}, we know that the family of processes $(\bfu_t^\sse{\gamma})_{t\ge0}$, $\gamma>0$, is tight w.r.t.\ the Meyer-Zheng topology on $D_{[0,\infty)}(\calM_\tem(\R)^2)$. 
(Note that these results were proved including the case $\rho=-1$.)

Let $(\bfu_t)_{t\ge0}\in D_{[0,\infty)}(\calM_\tem(\R)^2)$ be any subsequential limit point of $(\bfu_t^\sse{\gamma})_{t\ge0}$. 
By essentially the same proof as that of Theorem \ref{thm:moment_duality_infinite}, $(\bfu_t)_{t\ge0}$ satisfies the moment duality \eqref{eq:moment duality infinite gamma} for each $n\in\N$, $t>0$. Since $u_t^\ssup{1}$ and $u_t^\ssup{2}$ are dominated by $u_t^\ssup{1}+u_t^\ssup{2}=w_t=S_tw_0$ and $w_0=u_0^\ssup{1}+u_0^\ssup{2}$ is bounded, the moments uniquely determine the one-dimensional distributions of $(\bfu_t)_{t\ge0}$.
The same argument shows that almost surely under $\p_{\bfu_0}$, the limit measures $u_t^\ssup{1}$ and $u_t^\ssup{2}$ are absolutely continuous for all $t\ge0$ (in case $\calS=\R)$ with densities bounded uniformly by $\|w_0\|_\infty$. 
This shows in particular that $(\bfu_t)_{t\ge0}$ takes values in the subspace $\calM_b(\R)^2\subseteq\cM_\tem(\R)^2$.
Moreover, as in \cite[Lemma 4.4, Cor.\ 4.5]{BHO15} the mutual singularity of the densities $u_t^\ssup{1}$ and $u_t^\ssup{2}$ for fixed $t>0$ (separation of types) can be deduced from the identity \eqref{eq:mixed_second_moment} for second mixed moments which is a special case of the moment duality \eqref{eq:moment duality infinite gamma}. 

It remains to prove uniqueness of subsequential limit points, and the fact that the convergence is actually in the stronger Skorokhod topology so that the limit is also continuous. 
In fact, we will show strong convergence for the corresponding semigroups, which turn out to have the Feller property when restricted to suitable subspaces of $\cM_b(\R)^2$.
The Feller property in particular implies also the strong Markov property, thus the limit has all properties claimed in Theorem \ref{thm:main1}.

In order to carry out this program, let $(\bfu_t)_{t\ge0}\in D_{[0,\infty)}(\calM_b(\R)^2)$ be any subsequential limit point of the family $(\bfu_t^\sse{\gamma})_{t\ge0}$, $\gamma>0$. 
For $t>0$ and $\bfu_0\in \cM_b(\R)^2$, we define $P_t(\bfu_0,\cdot)$ as the law of $\bfu_t$ under $\p_{\bfu_0}$, i.e.
\[P_t(\bfu_0,f):=P_tf(\bfu_0):=\E_{\bfu_0}\left[f(\bfu_t)\right]
\]
for bounded measurable $f:\cM_b(\R)^2\to\R$.  
Since the one-dimensional distributions of $(\bfu_t)_{t\ge0}$ are uniquely determined by the moment duality, the definition of $P_t$ does not depend on the choice of limit point. 
Note that this defines a family of probability laws on $\cM_b(\R)^2$, but a priori we do not know that the family $(P_t)_{t\ge0}$ is a semigroup or even that $P_t(\cdot,\cdot)$ is a transition kernel for any fixed $t>0$.
 
For each $K>0$, we write $\calM_K$ for 
the space of all pairs of absolutely continuous Radon measures $\bfu=(u^\ssup{1},u^\ssup{2})$ on $\R$ such that the sum of the densities $u^\ssup{1}+u^\ssup{2}$ is bounded by the constant $K$.
Observe that $\bigcup_{K>0}\calM_K=\cM_b(\R)^2\subseteq\cM_\tem(\R)^2$, and again we endow each $\calM_K$ with the subspace topology inherited from $\calM_\tem(\R)^2$. 
It is easy to see that with this topology, $\cM_K$ is a compact space. 
Note that starting from $\bfu_0\in\cM_K$, almost surely $(\bfu_t)_{t\ge0}$ takes values in $\cM_K$, and the same holds for the finite rate
processes $(\bfu^\sse{\gamma}_t)_{t\ge0}$, $\gamma\in(0,\infty)$.

Let $(P_t^\sse{\gamma})_{t\ge0}$ denote the semigroup of $(\bfu_t^\sse{\gamma})_{t\ge0}$. We will now show that $(P_t^\sse{\gamma})_{t\ge0}$ is a Feller semigroup when restricted to the space $\cC(\cM_K)$ and that it converges strongly to $(P_t)_{t\ge0}$, whence the latter is also a Feller semigroup. 
The key is Lemma \ref{lemma:semigroup-continuous}, for which we need the following notation:
Given $n\in\N$, $m\in \{1,2\}^n$ and $\phi\in L^1(\R^n)$ we define a function  $f_{\phi,m}: \cM_b(\R)^2\to\R$ by
\begin{align}\label{defn:f_phi,m}
f_{\phi,m}(\bfu) := \int_{\R^n}\bfu^\ssup{m}(\bfx)\phi(\bfx)\,d\bfx, \qquad\bfu\in\cM_b(\R)^2,
\end{align}
where $\bfu^\ssup{m}(\bfx):=\prod_{i=1}^n u^\ssup{m_i}(x_i)$. Note that the function $f_{\phi,m}$ is bounded on each subset $\cM_K$, $K>0$. 
Moreover, if $\phi$ has the form $\phi(\bfx)=\prod_{i=1}^n\phi_i(x_i)$ with $\phi_i\in\calC_c(\R)$, $i=1,\ldots,n$, then by the definition of the topology in $\cM_\tem(\R)$ we see that $f_{\phi,m}$ is continuous.

\begin{lemma}\label{lemma:semigroup-continuous}
Let $n\in\bN$, $m\in\{1,2\}^n$, $\phi_i\in\calC_c(\R)$ for $i=1,\ldots,n$ and $\phi(\bfx):=\prod_{i=1}^n\phi_i(x_i)$. 
Then for all $t>0$, $\gamma\in(0,\infty)$ and $K>0$, the functions $P_tf_{\phi,m}$ and $P_t^\sse{\gamma}f_{\phi,m}$ are continuous on $ \calM_K$,
and we have
\begin{equation}\label{eq:convergence-special}
\sup_{\bfu\in\calM_K}|P_{t}f_{\phi,m}(\bfu) - P_t^\sse{\gamma}f_{\phi,m}(\bfu)|\to0\qquad\text{as } \gamma\to\infty.
\end{equation}
\end{lemma}

\begin{proof}
Fix $t>0$. By the moment duality,
\begin{align}\label{proof:semigroup-continuous_1}\begin{aligned}
&P_t f_{\phi,m}(\bfu) = \sum_{m'\in\{1,2\}^n}\int_{\R^n}d\bfx\,\phi(\bfx)\,\bE_{\bfx}\left[ \bfu^\ssup{m'}(\bfX_t)\,M_t^\sse{\infty}(\bfX,\delta_m)(m')\right]\\
&= \sum_{m'\in\{1,2\}^n}\int_{\R^n}d\bfx\,\phi(\bfx)\, \int_{\bR^n} d\bfz \,p_t(\bfx - \bfz) \bfu^\ssup{m'}(\bfz)\, \bE_{\bfx}\left[ M_t^\sse{\infty}(\bfX,\delta_m)(m')\;\middle|\;\bfX_t=\bfz \right]\,	\\
&=g_{t,\phi,m}^\ssup{\delta}(\bfu) 
+\sum_{m'\in\{1,2\}^n}\int_{\R^n}d\bfx\,\phi(\bfx) \int_{\bR^n}d\bfy \int_{\bR^n}d\bfz\, p_{t-\delta}(\bfx - \bfy)p_\delta(\bfy-\bfz) \bfu^\ssup{m'}(\bfz) \times\\
& \hspace{3cm} \times\Big(\bE_{\bfx}\left[ M^\sse{\infty}_t(\bfX,\delta_m)(m')\;\middle|\;\bfX_t=\bfz \right] - \bE_{\bfx}\left[ M_t^\sse{\infty}(\bfX,\delta_m)(m')\;\middle|\;\bfX_t=\bfy \right]\Big),
\end{aligned}
\end{align}
where we define
\begin{align*} & g_{t,\phi,m}^\ssup{\delta}(\bfu):=\sum_{m'\in\{1,2\}^n}\int_{\R^n}d\bfx\,\phi(\bfx) \int_{\bR^n}d\bfy \, p_{t-\delta}(\bfx - \bfy)\times\\
& \qquad\qquad\qquad\qquad\qquad\times\bE_{\bfx}\left[ M_t^\sse{\infty}(\bfX,\delta_m)(m')\;\middle|\;\bfX_t=\bfy \right] \int_{\bR^n}d\bfz\,p_\delta(\bfy-\bfz) \bfu^\ssup{m'}(\bfz).\end{align*}
Now let $K>0$. We will show that $ g_{t,\phi,m}^\ssup{\delta}(\bfu)$ is continuous as a function of $\bfu\in\calM_K$ for each fixed $\delta>0$, while the difference in \eqref{proof:semigroup-continuous_1} tends to $0$ as $\delta\downarrow0$ uniformly in $\bfu\in\calM_K$. 

First, it is clear by the definition of the topology of $\calM_\tem(\R)$ that $\int_{\R^n}p_\delta(\bfy-\bfz)\,\tilde\bfu^\ssup{m'}(\bfz)\,d\bfz\to\int_{\R^n}p_\delta(\bfy-\bfz)\,\bfu^\ssup{m'}(\bfz)\,d\bfz$ as $\tilde\bfu\to\bfu$ in $\calM_K$, for each fixed $\bfy\in\R^n$. Since the process $M_t^\sse{\infty}$ is uniformly bounded by Proposition~\ref{prop:M_bounded} and $\bfu^\ssup{m'}(\bfz)\le K^n$ uniformly in $\bfu\in\calM_K$ and $\bfz\in\R^n$, by dominated convergence we get that $g_{t,\phi,m}^\ssup{\delta}(\tilde\bfu)\to g_{t,\phi,m}^\ssup{\delta}(\bfu)$ as  $\tilde\bfu\to\bfu$ in $\calM_K$.

Moreover, for each $\bfu\in\calM_K$ we have by \eqref{proof:semigroup-continuous_1} that
\begin{align}\label{proof:semigroup-continuous_2}\begin{aligned}
&|P_t f_{\phi,m}(\bfu) - g_{t,\phi,m}^\ssup{\delta}(\bfu)|\\
&\le K^n\sum_{m'\in\{1,2\}^n}\int_{\R^n}d\bfx\,\phi(\bfx) \int_{\bR^n}d\bfy \, p_{t-\delta}(\bfx - \bfy)\int_{\bR^n}d\bfz\,p_\delta(\bfy-\bfz) \times\\
& \hspace{2cm} \times\Big|\bE_{\bfx}\left[ M^\sse{\infty}_t(\bfX,\delta_m)(m')\;\middle|\;\bfX_t=\bfz \right] - \bE_{\bfx}\left[ M_t^\sse{\infty}(\bfX,\delta_m)(m')\;\middle|\;\bfX_t=\bfy \right]\Big|.
\end{aligned}\end{align}

Since $M_t^\sse{\infty}$ depends only on the number and order of collisions between the Brownian motions $\bfX$, we have that $\bE_{\bfx}\left[ M_t^\sse{\infty}(\bfX,\delta_m)(m')\;\middle|\;\bfX_t= \cdot \right]$ is continuous outside of points where $y_i=y_j$ for some $i\neq j$. 
In particular, for almost every $\bfx,\bfy\in\R^n$ the function 
\[\bfz\mapsto \Big(\bE_{\bfx}\left[ M^\sse{\infty}_t(\bfX,\delta_m)(m')\;\middle|\;\bfX_t=\bfz \right] - \bE_{\bfx}\left[ M_t^\sse{\infty}(\bfX,\delta_m)(m')\;\middle|\;\bfX_t=\bfy \right]\Big)\] 
is continuous at $\bfy$, where it takes the value $0$. Thus the integral $\int_{\R^n}d\bfz\cdots$ in \eqref{proof:semigroup-continuous_2} tends to $0$ for almost every $\bfx,\bfy\in\R^n$. 
Again using that the process $M_t^\sse{\infty}$ is uniformly bounded, by dominated convergence the whole integral on the RHS of \eqref{proof:semigroup-continuous_2} vanishes as $\delta\downarrow0$. In particular, the LHS vanishes uniformly in $\bfu\in\calM_K$.

Consequently, given $\eps>0$ we first choose $\delta>0$ such that $|P_tf_{\phi,m}(\bfu)-g_{t,\phi,m}^\ssup{\delta}(\bfu)|<\eps/3$ for all $\bfu\in\calM_K$ and then $\eta=\eta(\delta)>0$ such that $|g_{t,\phi,m}^\ssup{\delta}(\bfu)-g_{t,\phi,m}^\ssup{\delta}(\tilde\bfu)|<\eps/3$ for $d_{\cM_\tem}(\bfu,\tilde\bfu)<\eta$ to obtain
\[
\begin{aligned}
&|P_tf_{\phi,m}(\bfu) - P_tf_{\phi,m}(\tilde\bfu)|\\
&\le |P_tf_{\phi,m}(\bfu)-g_{t,\phi,m}^\ssup{\delta}(\bfu)| + |g_{t,\phi,m}^\ssup{\delta}(\bfu)-g_{t,\phi,m}^\ssup{\delta}(\tilde\bfu)| + |P_tf_{\phi,m}(\tilde\bfu)-g_{t,\phi,m}^\ssup{\delta}(\tilde\bfu)|<\eps
\end{aligned}\]
if $d_{\cM_\tem}(\bfu,\tilde\bfu)<\eta$.
Thus $P_tf_{\phi,m}$ is continuous on $\calM_K$. 
The proof for the continuity of $P_t^\sse{\gamma}f_{\phi,m}$ is analogous, replacing $M_t^\sse{\infty}$ by $M_t^\sse{\gamma}$. 

Moreover, again using the moment duality for finite resp.\ infinite branching rate we obtain
\begin{align}
&|P_{t}^\sse{\gamma}f_{\phi,m}(\bfu) - P_tf_{\phi,m}(\bfu)|\\
&=\left|\sum_{m'\in\{1,2\}^n}\int_{\R^n}d\bfx\,\phi(\bfx)\,\bE_{\bfx}\left[ \bfu^\ssup{m'}(\bfX_t)\left(M_t^\sse{\gamma}(\bfX,\delta_m)(m') - M_t^\sse{\infty}(\bfX,\delta_m)(m')\right)\right]\right|\\
&\le K^n\sum_{m'\in\{1,2\}^n}\int_{\R^n}d\bfx\,|\phi(\bfx)|\,\bE_{\bfx}\left[ \left|M_t^\sse{\gamma}(\bfX,\delta_m)(m') - M_t^\sse{\infty}(\bfX,\delta_m)(m')\right|\right]\\
\end{align}
uniformly in $\bfu\in\calM_K$, and by dominated convergence the RHS tends to $0$ as $\gamma\uparrow\infty$. 
\end{proof}

\begin{corollary}\label{cor:Feller}
Let $K>0$. 
For all $t>0$ and $\gamma\in(0,\infty)$, we have
\begin{equation}\label{eq:Feller-property}
P_t(\cC(\cM_K))\subseteq\cC(\cM_K) \qquad\text{and}\qquad P_t^\sse{\gamma}(\cC(\cM_K))\subseteq\cC(\cM_K). 
\end{equation}
Moreover, for each $f\in\cC(\cM_K)$ we have
\begin{equation}\label{eq:strong-convergence}
\sup_{\bfu\in\calM_K}|P_{t}f(\bfu) - P_t^\sse{\gamma}f(\bfu)|\to0\qquad\text{as } \gamma\to\infty.
\end{equation}
In particular, $(P_t)_{t\ge0}$ and $(P_t^\sse{\gamma})_{t\ge0}$ are Feller semigroups on $\cC(\cM_K)$.
\end{corollary}

\begin{proof}
Fix $t>0$. 
Let $\cH\subseteq\cC(\cM_K)$ denote the system of all functions $f_{\phi,m}$ of the form \eqref{defn:f_phi,m} with $\phi=\otimes_{i=1}^n\phi_i$, where $\phi_i\in\cC_c(\R)$, $i=1,\ldots,n$. 
Note that $\cH$ is closed under multiplication and separates the points of $\cM_K$.
Writing $\cA\subseteq\cC(\cM_K)$ for the algebra generated by $\cH$ and the constant functions, we conclude 
using Lemma \ref{lemma:semigroup-continuous} that for all $f\in\cA$ we have $P_t^\sse{\gamma}f\in\cC(\cM_K)$, $P_tf\in \cC(\cM_K)$ and \eqref{eq:strong-convergence} holds. 
Since by the Stone-Weierstrass theorem $\cA$ is dense in $\cC(\cM_K)$ w.r.t.\ the uniform norm (recall that $\cM_K$ is compact), this extends easily to all $f\in\cC(\cM_K)$.

The fact that $P_tf\in\cC(\cM_K)$ for all $f\in\cC(\cM_K)$ implies in particular that $P_t(\cdot,\cdot)$ is a transition kernel for each $t>0$. Moreover, the semigroup property of $(P_t)_{t\ge0}$ follows from the convergence \eqref{eq:strong-convergence} and the semigroup property of $(P_t^\sse{\gamma})_{t\ge0}$.

Since $(\bfu_t^\sse{\gamma})_{t\ge0}$ and $(\bfu_t)_{t\ge0}$ have c\`adl\`ag paths, it is clear that
\begin{equation}\label{eq:continuity-semigroup}
P_t^\sse{\gamma}f(\bfu)\to f(\bfu)\quad\text{ and }\quad P_tf(\bfu)\to f(\bfu)\qquad \text{as }t\downarrow0,
\end{equation}
for each $f\in\cC(\cM_K)$ and $\bfu\in\cM_K$. Together with \eqref{eq:Feller-property}, this gives the Feller property of the semigroups $(P_t^\sse{\gamma})_{t\ge0}$ and $(P_t)_{t\ge0}$; in particular, these are strongly continuous on $\cC(\cM_K)$. 
\end{proof}

Now it is straightforward to finish the proof of Theorem \ref{thm:main1}: Let $\bfu_0\in\cM_K$, $K>0$. 
By standard theory (see e.g. \cite[Thm.\ 4.2.5]{EK86}), the strong convergence \eqref{eq:strong-convergence} of the Feller semigroups on $\cC(\cM_K)$ implies convergence of the family of processes $(\bfu^\sse{\gamma}_t)_{t\ge0}$ as $\gamma\uparrow\infty$ in $D_{[0,\infty)}(\cM_K)$ w.r.t.\ the Skorokhod topology, and the unique limit $(\bfu_t)_{t\ge0}$ is a Markov process with semigroup $(P_t)_{t\ge0}$. Since the approximating processes $(\bfu_t^\sse{\gamma})_{t\ge0}$ are continuous and the convergence is in the Skorokhod topology, the limit is in fact in $\cC_{[0,\infty)}(\cM_K)$. Also by standard theory, the strong Markov property of $(\bfu_t)_{t\ge0}$ w.r.t.\ the usual augmentation of its canonical filtration follow from the Feller property of $(P_t)_{t\ge0}$. 
This concludes the proof of Theorem \ref{thm:main1}.

\subsection{Some notation and preliminaries }\label{ssec:notation}
We continue to use all the notation introduced in Section \ref{ssec:-1}. Moreover,
we write $\bR^{n,\uparrow}:=\{\bfx\in\bR^n: x_1<x_2<\cdots<x_n\}$ for the space of increasing sequences of length $n$ in $\bR$, and analogously $\bQ^{n,\uparrow}$ or $[a,b]^{n,\uparrow}$ for sequences in $\bQ$ or $[a,b]$.

\begin{lemma}\label{lemma:U}
Let $\bfu=(u^\ssup{1},u^\ssup{2})\in\cU$. 
\begin{itemize}
\item[a)] Suppose that $a<b$ with $a\in\supp(u^\ssup{i})$ and $b\in\supp(u^\ssup{3-i})$, $i\in\{1,2\}$.
Then there must be an interface point between $a$ and $b$, i.e.\ $\cI(\bfu)\cap [a,b]\ne\emptyset$. 
\item[b)] Suppose that $a,b\in\cI(\bfu)$ with $a<b$ and that there are no interface points (strictly) between $a$ and $b$, i.e. $\cI(\bfu)\cap(a,b)=\emptyset$. 
Then there exists $i\in\{1,2\}$ with $(a,b)\subseteq\supp( u^\ssup{i})\setminus\supp( u^\ssup{3-i})$.
\end{itemize}
\end{lemma}

\begin{proof}
a) Suppose w.l.o.g. that $a\in\supp(u^\ssup{1})$ and $b\in\supp(u^\ssup{2})$. Define
\[r:=\max\left\{x\in[a,b]:x\in\supp(u^\ssup{1})\right\},\qquad \ell:=\min\left\{x\in[a,b]:x\in\supp(u^\ssup{2})\right\}.\]
Note that the maximum resp.\ minimum is attained since the support of a measure is necessarily closed, and we have $r\in\supp(u^\ssup{1})$, $\ell\in\supp(u^\ssup{2})$. If $r=b$ resp. $\ell=a$, then $b\in\cI(\bfu)$ resp. $a\in\cI(\bfu)$ and the proof is finished. Suppose now that $r<b$ and $\ell>a$. Then we must have $r\ge\ell$: Assume by way of contradiction that $r<\ell$. Then for all $x\in (r,\ell)$ we have by definition that $x\notin\supp(u^\ssup{1})\cup\supp(u^\ssup{2})=\supp(u^\ssup{1}+u^\ssup{2})$, which is a contradiction since $u^\ssup{1}+u^\ssup{2}$ is equivalent to the Lebesgue measure and thus $\supp(u^\ssup{1}+u^\ssup{2})=\R$. Consequently $r\ge\ell$, and moreover $r$ and $\ell$ are interface points: Let $r_n\in[a,b]$, $r_n\downarrow r$. Then $r_n\notin\supp(u^\ssup{1})$ for all $n\in\N$ by definition of $r$. Since $\supp(u^\ssup{1}+u^\ssup{2})=\R$, we must have $r_n\in\supp(u^\ssup{2})$ for all $n\in\N$ and thus $r=\lim_{n\to\infty}r_n\in\supp(u^\ssup{2})$. Since also $r\in\supp(u^\ssup{1})$ by 
definition of $r$, we have $r\in\cI(\bfu)$. An analogous argument shows that also $\ell\in\cI(\bfu)$. In any case, it follows that $\cI(\bfu)\cap[a,b]\ne\emptyset$. 

b) Since $\supp(u^\ssup{1})\cup\supp(u^\ssup{2})=\supp(u^\ssup{1}+u^\ssup{2})=\R$ and $(a,b)$ contains no interface points by assumption, each point $x\in(a,b)$ belongs to either $\supp(u^\ssup{1})$ or to $\supp(u^\ssup{2})$. However, if there were $x,y\in(a,b)$ with $x<y$, $x\in\supp(u^\ssup{1})$ and $y\in\supp(u^\ssup{2})$ (or vice versa), then by a) there would have to be an interface point in the interval $[x,y]\subseteq(a,b)$, contrary to assumption.
\end{proof}

\newpage
\begin{lemma}\label{lemma:interface-characterisation}
Let $\bfu\in\cU$ and $n\in\N$. The following are equivalent:
\begin{itemize}
\item[a)] $\bfu \in \bigcup_{k=0}^n \cU_k $.
\item[b)] For each $a\in \bQ^{2(n+2),\uparrow}$ and alternating coloring $m\in\{1,2\}^{n+2}$, we have 
\begin{equation}\label{eq:interface-characterisation}
\prod_{i=1}^{n+2} u^\ssup{m_i}([a_{2i-1},a_{2i}])=0 .
\end{equation}
\item[c)] For each $m\in\{1,2\}^{n+2}$ alternating, we have $\bfu^\ssup{m}(\bfx)=0$ for Lebesgue-a.e. $\bfx\in\R^{n+2,\uparrow}$.
\end{itemize}
\end{lemma}

\begin{proof}
a) $\Rightarrow$ b):  
Fix $n\in\N$ and assume that $\bfu \in \bigcup_{k=0}^n \cU_k $, i.e.\ $\bfu$ has at most $n$ interface points. Suppose by way of contradiction that for some $a\in \bQ^{2(n+2),\uparrow}$ and alternating configuration $m$, we have $\prod_{i=1}^{n+2} u^\ssup{m_i}([a_{2i-1},a_{2i}])>0$. 
Since $u^\ssup{1},u^\ssup{2}$ are absolutely continuous, for each $i=1,\ldots,n+2$ we have $(a_{2i-1},a_{2i})\cap\supp(u^\ssup{m_i})\ne\emptyset$ and can pick two distinct points in $(a_{2i-1},a_{2i})\cap\supp(u^\ssup{m_i})$. Labeling these points in increasing order, we get $\bfx\in\R^{2(n+2),\uparrow}$ such that $x_{2i-1},x_{2i}\in(a_{2i-1},a_{2i})\cap\supp(u^\ssup{m_i}) $.
Then we have $x_{2i}<x_{2i+1}$, $x_{2i}\in\supp(u^\ssup{m_i})$ and $x_{2i+1}\in\supp(u^\ssup{m_{i+1}})$ for all $i=1,\ldots,(n+2)$. Since $m_i\ne m_{i+1}$, by Lemma \ref{lemma:U} there must be an interface point in the interval $[x_{2i},x_{2i+1}]$, $i=1,\ldots,n+1$. Since these intervals are all disjoint and there are $n+1$ of them, there must be at least $n+1$ interfaces, contradicting the hypothesis that $\bfu\in\cU_n$.

b) $\Rightarrow$ c): Assume that b) holds. Let $m\in\{1,2\}$ be alternating. Using continuity of the measures $u^\ssup{i}$ and the fact that $\Q$ is dense in $\R$, we see that \eqref{eq:interface-characterisation} holds in fact for all $a\in\R^{2(n+2),\uparrow}$. 
In particular, given $\bfx\in\R^{n+2,\uparrow}$ we put $a_{2i-1}:=x_i-\eps$, $a_{2i}:=x_i+\eps$ for $i=1,\ldots,n+2$. Then we have $a\in\R^{2(n+2),\uparrow}$ for all $\eps>0$ small enough, and thus by \eqref{eq:interface-characterisation}
\begin{equation}\label{proof:interface-characterisation_1}
\prod_{i=1}^{n+2} u^\ssup{m_i}([x_i-\eps,x_i+\eps])=0
\end{equation}
for all $\bfx=(x_1,\ldots,x_{n+2})\in\R^{2(n+2),\uparrow}$ and $\eps=\eps(\bfx)>0$ small enough. 
Moreover, by general differentiation theory for measures we have
\[u^\ssup{i}(x)=\lim_{\eps\downarrow0}\frac{1}{2\eps}u^\ssup{i}([x-\eps,x+\eps])\qquad\text{for Lebesgue-a.e. }x\in\R,\; i=1,2,\]
and consequently
\begin{equation}\label{proof:interface-characterisation_2}
\bfu^\ssup{m}(\bfx)=\prod_{i=1}^{n+2}u^\ssup{m_i}(x_i)=\lim_{\eps\downarrow0}\frac{1}{(2\eps)^{n+2}}\prod_{i=1}^{n+2}u^\ssup{m_i}([x_i-\eps,x_i+\eps])
\end{equation}
for Lebesgue-a.e. $\bfx=(x_1,\ldots,x_{n+2})\in\R^{n+2}$. Combining \eqref{proof:interface-characterisation_1} and \eqref{proof:interface-characterisation_2} shows that $\bfu^\ssup{m}(\bfx)=0$ for Lebesgue-a.e. $\bfx\in\R^{n+2,\uparrow}$.

c) $\Rightarrow$ b): Assume that c) holds. Suppose that $\prod_{i=1}^{n+2} u^\ssup{m_i}([a_{2i-1},a_{2i}])>0$ for some $a\in\Q^{2(n+2),\uparrow}$ and one of the alternating colorings. Then since the measures are absolutely continuous, there must be subsets $A_i\subseteq[a_{2i-1},a_{2i}]$ with positive Lebesgue measure $\mathrm{Leb}(A_i)>0$ such that $u^\ssup{m_i}(x)>0$ for all $x\in A_i$, $i=1,\ldots,n+2$. Then $A:=A_1\times A_2\times\cdots\times A_{n+2}\subseteq\R^{n+2,\uparrow}$ has positive Lebesgue measure and $\bfu^\ssup{m}(\bfx)>0$ for all $\bfx\in A$, contrary to assumption.

b) $\Rightarrow$ a): Assume that b) holds. Suppose by way of contradiction that $\bfu\notin \bigcup_{k=0}^n \cU_k$, i.e.\ there are at least $n+1$ interfaces $x_1<x_2<\cdots<x_{n+1}$. We have to show that there exists $a\in \bQ^{2(n+2),\uparrow}$ and an alternating configuration $m\in\{1,2\}^{n+2}$ such that $u^\ssup{m_i}([a_{2i-1},a_{2i}])>0$ for all $i=1,\ldots,n+2$. 

Let $\delta:=\frac12\min\{x_{i+1}-x_i: 1\leq i\leq n\}$. 
Since $x_1$ is an interface point and $u^\ssup{1}+u^\ssup{2}$ is equivalent to Lebesgue measure, there must be a pair of indices $(k,\ell)\in\{(1,2),(2,1)\}$ such that $u^\ssup{k}([x_1-\delta,x_1])>0$ and $u^\ssup{\ell}([x_1,x_1+\delta])>0$.
Let $m\in\{1,2\}^{n+2}$ be the alternating coloring with $m_1=k$. Using the fact that $u^\ssup{j}([x_i-\delta,x_i+\delta])>0$ for all $i=2,...,n+1$ and $j=1,2$, we can set $a'_1:=x_1-\delta$, $a'_2:=a'_3:=x_1$, $a'_4:=x_1+\delta$ and $a'_{2i+1}:=x_i-\delta$, $a'_{2i+2}:=x_i+\delta$ for $i=2,...,n+1$ in order to obtain
\begin{align}\label{eq:5.5b->a}
\prod_{i=1}^{n+2}u^\ssup{m_i}([a'_{2i-1},a'_{2i}])>0. 
\end{align}
Now we can choose $a\in\bQ^{2(n+2),\uparrow}$ so that $[a_{2i-1},a_{2i}]\subset(a'_{2i-1},a'_{2i})$ and $u^\ssup{m_i}([a_{2i-1},a_{2i}])>0$, $i=1,\ldots,n+2$. 
Then \eqref{eq:5.5b->a} holds with $a$ in place of $a'$, contradicting b).
\end{proof}

\subsection{Non-proliferation of interfaces}\label{ssec:nonproliferation}
In this subsection, we prove that the number of interfaces is non-increasing. 

\begin{theorem}\label{thm:interfacecount-decreasing}
Suppose that $\bfu_0\in\cU_n$. Then, 
\[ {\bP}_{\bfu_0}\left(\bfu_t \in \bigcup_{k=0}^n \cU_k \text{ for all }t\geq0\right)=1. \]
\end{theorem}

\begin{remark}
The above result holds also for $\rho>-1$ as long as $\rho$ is sufficiently close to $-1$ w.r.t. $n$.
More precisely, the proof below uses the moment duality for $(n+2)$-th moments and thus needs $\rho<-\cos\left(\frac{\pi}{n+2}\right)$. For example, if $\rho<-\frac{1}{2}$, then starting from an initial configuration with a single interface ($n=1$), we obtain that almost surely there is at most one interface point at any time $t>0$.\footnote{Actually, there exists an alternative proof of this fact (using the approximation by the discrete-space model) which works for \emph{all} $\rho<0$, see \cite[Thm.\ 2.6]{HO15}.}
Moreover, the analogue of Theorem \ref{thm:interfacecount-decreasing} holds also for the discrete model.
\end{remark}

The key step to prove the theorem is the following lemma, for which we recall the notation $f_{\phi,m}$ from \eqref{defn:f_phi,m}.
\begin{lemma}\label{lemma:alternating}
Let $n\in\N_0$. 
Then for all $\bfu_0\in \calM_b(\R)^2$, 
each alternating coloring $m\in\{1,2\}^{n+2}$ and $\phi\in L^1(\R^{n+2})$ with $\supp(\phi)\subseteq\R^{n+2,\uparrow}$, we have
\begin{equation}\label{lemma:alternating_1}
P_tf_{\phi,m}(\bfu_0) = \bE_{\bfu_0}\left[f_{\phi,m}(\bfu_t)\right] = \int_{\R^{n+2,\uparrow}}\phi(\bfx)\,\E_{\bfx}\left[\bfu_0^\ssup{m}(\bfX_t)\ind_{t<\tau}\right]d\bfx,
\end{equation}
where 
\[\tau:=\inf\left\{t\geq 0: \exists 1\leq i<j\leq n \text{ with } X^\ssup{i}_t=X^\ssup{j}_t \right\}\]
is the first collision time of the Brownian motions $(\bfX_t)_{t\ge0}$. 
If in addition $\bfu_0\in\cU_n$, we have
\begin{align}\label{lemma:alternating_2}
P_tf_{\phi,m}(\bfu_0)=0.
\end{align}
\end{lemma}
\begin{proof}
Let $n\in\N_0$ and $m\in\{1,2\}^{n+2}$ alternating. For $\bfx\in\R^{n+2,\uparrow}$, denote by $(i,i+1)$ the index pair of the two Brownian motions involved in the collision at time $\tau$, starting from $\bfx$. 
Then $M^\sse{\infty}_{t}(\bfX,\delta_m)=M^\sse{\infty}_0(\bfX,\delta_m)=\delta_m$ for $ t\le\tau$, and $\pi(\bfX_{\tau})=\{\{1\},\{2\},...,\{i-1\},\{i,i+1\},\{i+2\},...,\{n+2\}\}$. 
Since $m$ is alternating, we have $m_i\neq m_{i+1}$, hence by the explicit form of $K_\infty$ (see Prop. \ref{prop:general_partition}) 
\[K_\infty(M^\sse{\infty}_\tau,\pi(\bfX_\tau))\equiv0.\]
In view of the recursive definition of $M_t^\sse{\infty}$ (see Prop. \ref{prop:thm_M_infty_continuous}), this implies $M_t^\sse{\infty}(\bfX,\delta_m)\equiv0$ for all $t>\tau$, 
thus we have
\[M_t^\sse{\infty}(\bfX,\delta_m)=\1_{t\le\tau}\,\delta_m.\]

By the moment duality, see Theorem~\ref{thm:moment_duality_infinite}, since $\supp(\phi)\subseteq\R^{n+2,\uparrow}$ this implies 
\begin{align}\label{eq:alternating}
\bE_{\bfu_0}\left[f_{\phi,m}(\bfu_t)\right] &= \int_{\R^{n+2}}\phi(\bfx)\, \E_{\bfx}\left[\sum_{b\in\{1,2\}^{n+2}}M_t^\sse{\infty}(X,\delta_m)(b)\,\bfu_0^\ssup{b}(X_t)\right]\,d\bfx\\
&=\int_{\R^{n+2,\uparrow}}\phi(\bfx)\, \E_{\bfx}\left[\bfu_0^\ssup{m}(\bfX_t)\ind_{t<\tau}\right]\,d\bfx,
\end{align}
where we also used that $\p_\bfx(\tau=t)=0$ for each fixed $t>0$. Thus \eqref{lemma:alternating_1} is proved.

Moreover, if in addition $\bfu_0\in\calU_n$, then by Lemma \ref{lemma:interface-characterisation} we have $\bfu_0^\ssup{m}(\bfx)=0$ for almost every $\bfx\in\R^{n+2,\uparrow}$, 
and under $\bP_\bfx$ we have $\bfX_t\in\R^{n+2,\uparrow}$ a.s. on the event $\{t<\tau\}$. 
Thus in this case the integrand on the RHS of the previous display is zero, and \eqref{lemma:alternating_2} is established.
\end{proof}

\begin{proof}[Proof of Theorem \ref{thm:interfacecount-decreasing}]
Fix $t\ge0$, an increasing vector $a\in\bR^{2(n+2),\uparrow}$ and an alternating configuration $m\in\{1,2\}^{n+2}$.
Since $\bfu_0\in\calU_n$, by Lemma \ref{lemma:alternating} applied to the function $\phi(\bfx):=\prod_{i=1}^{n+2}\1_{[a_{2i-1},a_{2i}]}(x_i)$ with $\supp(\phi)\subseteq\R^{n+2,\uparrow}$ we obtain
\begin{align*}
\bE_{\bfu_0}\left[\prod_{i=1}^{n+2}u_t^\ssup{m_i}([a_{2i-1},a_{2i}])\right]&=\bE_{\bfu_0}\left[\langle \bfu_t^\ssup{m},\phi\rangle\right]=0 
\end{align*}
and consequently 
\[\prod_{i=1}^{n+2}u_t^\ssup{m_i}([a_{2i-1},a_{2i}])=0\qquad \bP_{\bfu_0}\text{-a.s.}\]
Now denote by $m^\ssup j\in\{1,2\}^{n+2}$ the alternating coloring starting with 
$m_1^\ssup j =j$ for $j=1,2$.
Again using Lemma \ref{lemma:interface-characterisation}, for each fixed $t\ge0$ we have
\[
\left\{\bfu_t \in \bigcup_{k=0}^n \cU_k \right\} = \bigcap_{j\in\{1,2\}}\bigcap_{a\in \bQ^{2(n+2),\uparrow}} \left\{ \prod_{i=1}^{n+2} u_{t}^{(m^\ssup{j}_i)}([a_{2i-1},a_{2i}])=0 \right\},
\]
which is a countable intersection of events of probability $1$. As a consequence, we get $\p_{\bfu_0}\left(\bfu_t \in  \bigcup_{k=0}^n \cU_k\right)=1$ for each $t\ge0$, 
and hence also 
\[\p_{\bfu_0}\left( \bfu_t \in\bigcup_{k=0}^n \cU_k \text{ for all } t \in [0,\infty)\cap\bQ \right)=1.\] To extend this result to all $t\geq 0$, we use the fact that $(\bfu_t)_{t\ge0}$ has right-continuous paths and hence $\lim_{s\downarrow t} \prod_{i=1}^{n+2} u_{s}^\ssup{m_i^\ssup{j}}([a_{2i-1},a_{2i}])= \prod_{i=1}^n u_{t}^{(m_i^\ssup{j})}([a_{2i-1},a_{2i}])$, which shows that
\begin{align}
&\left\{\bfu_t \in \bigcup_{k=0}^n \cU_k  
\text{ for all }t \in[0,\infty)\cap\bQ \right\}\\
&=\bigcap_{t\in[0,\infty) \cap \bQ}\bigcap_{j\in\{1,2\}} \bigcap_{a\in \bQ^{2(n+2),\uparrow}}  \left\{ \prod_{i=1}^{n+2} u_{t}^{(m_i^\ssup{j})}([a_{2i-1},a_{2i}])=0 \right\} \\
&= \bigcap_{t\in[0,\infty) }\bigcap_{j\in\{1,2\}}\bigcap_{a\in \bQ^{2(n+2),\uparrow}} \left\{ \prod_{i=1}^{n+2} u_{t}^{(m_i^\ssup{j})}([a_{2i-1},a_{2i}])=0 \right\} 
=\left\{\bfu_t \in \bigcup_{k=0}^n \cU_k \text{ for all }t\geq0 \right\}.
\end{align}
Hence we have shown that almost surely, there are at most $n$ interfaces at any one time.
\end{proof}

\subsection{Movement of a single interface}\label{ssec:single_interface}
In this subsection we will prove Theorem \ref{thm:annihilating-BM-_onepoint}, showing that a single interface moves according to \eqref{eq:interface-movement3}. 
Again the key point is that the sum $w_t:=u_t^\ssup{1}+u_t^\ssup{2}$ solves the deterministic heat equation $w_t=S_tw_0$, where $(S_t)_{t\ge0}$ denotes the heat semigroup. 
In particular, recalling the definition of $m(\bfu_t,x)$ from Section \ref{ssec:-1}, we have 
\begin{equation}\label{eq:re}
u^\ssup{i}_t(dx)=\1_{\{m(\bfu_t,x)=i\}}\,w_t(x)\,dx,\qquad i=1,2,
\end{equation}
a fact we will use repeatedly.
\begin{lemma}\label{lemma:asymptotic-type}
Assume $\bfu_0\in\cU_n$ for some $n\in\N$. Then for each fixed $t>0$ we have $\lim_{x\to\pm\infty} m(\bfu_t,x)=\lim_{x\to\pm\infty} m(\bfu_0,x)$ almost surely.
\end{lemma}
\begin{proof} 
Fix $t>0$. First observe that since by Theorem \ref{thm:interfacecount-decreasing} there is a finite number of interfaces at time $t$, the limits $\lim_{x\to\pm\infty} m(\bfu_t,x)$ exist in $\{1,2\}$.  

Let $\phi\in\cC_c(\R)$. Since $w_t$ is deterministic, we can use \eqref{eq:re} and the moment duality to obtain
\begin{align}\label{proof:asymptotic-type1}
\E_{\bfu_0}\left[\int_\R\phi(x)\, \1_{m(\bfu_t,x)=i}\,dx\right]&=\E_{\bfu_0}\left[\left\langle\frac{\phi}{{w_t}}, u_t^\ssup{i}\right\rangle\right]=\left\langle \frac{\phi}{{w_t}}, S_tu_0^\ssup{i}\right\rangle\\
&=\left\langle\phi, \left(1 + \frac{S_t u_0^\ssup{3-i}}{S_t u_0^\ssup{i}}\right)^{-1}\right\rangle,\qquad i=1,2.
\end{align}
Now assume w.l.o.g.\ that $\lim_{x\to-\infty} m(\bfu_0,x)=1$ and $\min \cI(\bfu_0)=0$. Then we have $S_t u_0^\ssup{2}(x) \leq \norm{u^\ssup{2}_0}_\infty \int_0^\infty (2\pi t)^{-\frac12}e^{-(y-x)^2/2t} dy$, and since $u_0^\ssup{1}([-2\delta,-\delta])>0$ for $\delta>0$, we have for $x<-2\delta$
\[ S_t u_0^\ssup{1}(x) \geq u_0^\ssup{1}([-2\delta,-\delta]) (2\pi t)^{-\frac12}e^{-(-\delta-x)^2/2t}. \]
Therefore, for some $C=C(\bfu,\delta)>0$
\begin{align}
0\le\frac{S_t u_0^\ssup{2}(x)}{S_t u_0^\ssup{1}(x)} &\leq C \int_0^\infty e^{(-y^2 + 2yx +2\delta x)/2t} dy 
\leq C e^{\delta x/t} \int_0^\infty e^{-y^2/2t} dy,
\end{align}
which converges to 0 as $x\to-\infty$. Thus
\begin{equation}\label{proof:asymptotic-type2}
\left(1 + \frac{S_t u_0^\ssup{2}(x)}{S_t u_0^\ssup{1}(x)}\right)^{-1}\xrightarrow{x\to-\infty} 1.
\end{equation}
Now choose $\phi\in\cC_c(\R)$ nonnegative with $\int_\R\phi(x)\,dx=1$ and define $\phi_n(\cdot):=\phi(\cdot+n)$, $n\in\N$. Then \eqref{proof:asymptotic-type1} and \eqref{proof:asymptotic-type2} show that
\[\E_{\bfu_0}\left[\int_\R\phi_n(x)\, \1_{m(\bfu_t,x)=1}\,dx\right]\to1,\qquad\text{as } n\to\infty.\]
Since the limit $\lim_{x\to-\infty} m(\bfu_t,x)$ exists, this implies that it must equal $1$ a.s.
The result for $x\to+\infty$ is analogous.
\end{proof}

\begin{prop}\label{prop:single-interface}
Assume $\bfu_0\in\cU_1$. 
Then we have
\[\p_{\bfu_0}\left(\bfu_t\in\cU_1\text{ for all }t\ge0\right)=1.\]
Let $(I_t)_{t\ge0}$ denote the single interface process defined by the unique element of $\cI(\bfu_t)$, $t\ge0$. 
Then if $\lim_{x\to-\infty} m(\bfu_0,x)=1$ (resp.\ $=2$), we have
\[(u_t^\ssup{1}(dx),u_t^\ssup{2}(dx))=(\ind_{x< I_t} w_t(x)\,dx,\ind_{x> I_t} w_t(x)\,dx) \]
resp.
\[(u_t^\ssup{1}(dx), u_t^\ssup{2}(dx))=(\ind_{x> I_t}w_t(x)\,dx,\ind_{x< I_t} w_t(x)\,dx),\]
and the interface $(I_t)_{t\geq0}$ is a continuous Markov process.
\end{prop}
\begin{proof}
By Theorem \ref{thm:interfacecount-decreasing}, almost surely there is at most one interface for all $t\geq0$.

To show that there is at least one interface point for all $t$, let $\tau:=\inf\{t\geq 0: \bfu_t\in\cU_0 \}$. 
Let $\phi\in L^1(\R^2)$ with $\supp(\phi)\subseteq\R^{2,\uparrow}$ and $m=(1,2)$ (or $m=(2,1)$). 
Applying Lemma \ref{lemma:alternating} (with $n=0$), we have on $\{\tau<t\}$ that $P_{t-\tau} f_{\phi,m} (\bfu_\tau)=0$. By the strong Markov property in Theorem \ref{thm:main1} b), we get
\begin{align}
0 < \bE_{\bfu_0}\left[ f_{\phi,m} (\bfu_t) \right] 
&=\bE_{\bfu_0}\left[   \ind_{\tau<t}\, P_{t-\tau} f_{\phi,m} (\bfu_\tau)\right] + \bE_{\bfu_0}\left[ \ind_{ \tau\geq t}\,f_{\phi,m} (\bfu_t)  \right]
= \bE_{\bfu_0}\left[ \ind_{\tau\geq t}\, f_{\phi,m} (\bfu_t)\right],
\end{align}
which implies $\{\tau\ge t\}$ a.s.

Now that we have $\p_{\bfu_0}\left(\bfu_t \in \cU_1 \text{ for all }t\geq0 \right)=1$, we can identify $\bfu_t=(u_t^\ssup{1},u^\ssup{2}_t)$ with $(I_t,w_t)$ as in the statement of this proposition.
More precisely, 
by Lemma \ref{lemma:U} the function $x\mapsto m(\bfu_t,x)$ must be constant on $(-\infty,I_t)$ and on $(I_t,\infty)$. Using Lemma \ref{lemma:asymptotic-type}, we see that $\bfu_t$ is of the claimed form, depending on the form of $\bfu_0$.

The continuity of the paths of $(I_t)_{t\ge0}$ follows from the corresponding property of $(\bfu_t)_{t\ge0}$.
For the Markov property, observe that since $w_s$ is deterministic for each $s$ and the pair $(I_s,w_s)$ 
uniquely determines $\bfu_s=(u_s^\ssup{1},u_s^\ssup{2})$ and vice versa, we have
\begin{align}\label{proof:single-interface_1}\begin{aligned}
&\sigma(I_s)=\sigma(I_s,w_s)=\sigma(\bfu_s),\\
&\sigma(I_r:0\le r\le s)=\sigma(I_r,w_r:0\le r\le s)=\sigma(\bfu_r:0\le r\le s).
\end{aligned}\end{align}
Therefore the Markov property of $(I_t)_{t\ge0}$ follows from the Markov property of $(\bfu_t)_{t\ge0}$.
\end{proof}

Having established the existence of a single interface process, we proceed to identifying its law. In a first step, we compute the distribution of $I_t$ for fixed $t>0$.

\begin{lemma}
Assume $\bfu_0\in\cU_1$, so that there is a single interface process $(I_t)_{t\ge0}$ by Proposition \ref{prop:single-interface}.
Then for each fixed $t>0$ we have
\begin{equation}\label{distribution_interface_fixed_time}
{\bP}_{\bfu_0}(I_t\le x)=\frac{S_t(w_0\ind_{(I_0,\infty)})(x)}{S_tw_0(x)},\qquad x\in\R.
\end{equation}
In particular, the distribution of $I_t$ under $\p_{\bfu_0}$ is absolutely continuous with a smooth density.
\end{lemma}
\begin{proof}
Fix $t>0$ and assume w.l.o.g.\ that $\lim_{x\to-\infty}m(\bfu_0,x)=1$, i.e.\ $u_0^\ssup{1}=\1_{(-\infty,I_0)}w_0$. Then by Proposition \ref{prop:single-interface}, we have $\1_{\{I_t<x\}}\,dx=\frac{1}{w_t(x)}\,u^\ssup{2}_t(dx)$. By Fubini and the moment duality, we obtain since  $w_t$ is deterministic that for all test functions $\phi$
\begin{align}
\int_\R\phi(y)\, {\bP}_{\bfu_0}(I_t<y)\,dy&=\bE_{\bfu_0}\left[\int_\R\phi(y)\,\1_{\{I_t<y\}}\,dy\right] 
= \E_{\bfu_0}\left[\left\langle u_t^\ssup{2}, \frac{\phi}{w_t}\right\rangle\right]
\\&=\left\langle S_tu_0^\ssup{2}, \frac{\phi}{w_t}\right\rangle=\int_\R\phi(y)\frac{S_t(\ind_{(I_0,\infty)}w_0)(y)}{S_tw_0(y)}\,dy.
\end{align}
Now choose $\phi(y):=\phi^x_\delta(y):=p_\delta(x-y)$ for $x\in\R$ and let $\delta\downarrow0$: Then 
the RHS of the previous display tends to $\frac{S_t(\ind_{(I_0,\infty)}w_0)(x)}{S_tw_0(x)}$ 
for all $x\in\R$ since this expression is continuous in $x$, while the LHS tends to ${\bP}_{\bfu_0}(I_t<x)$ for almost all $x\in\R$. So we have
\[{\bP}_{\bfu_0}(I_t<x)=\frac{S_t(\ind_{(I_0,\infty)}w_0)(x)}{S_tw_0(x)}\qquad\text{for Lebesgue-a.e. }x\in\R,\]
where the RHS is continuous in $x$. 
But since $x\mapsto {\bP}_{\bfu_0}(I_t<x)$ is left-continuous, we must have equality everywhere, and \eqref{distribution_interface_fixed_time} follows.
Thus the distribution of $I_t$ under $\p_{\bfu_0}$ is absolutely continuous with a smooth density. 
\end{proof}

In order to finish the proof of Theorem \ref{thm:annihilating-BM-_onepoint}, we now show that the single interface process solves the SDE \eqref{eq:interface-movement3}. We restate this as a proposition:

\begin{prop}\label{prop:sde-interface}
Assume $\bfu_0\in\cU_1$, so that there is a single interface process $(I_t)_{t\ge0}$ by Proposition \ref{prop:single-interface}. 
Then $(I_t)_{t\ge0}$ is the unique (in law) weak solution of the SDE
\begin{align}\label{eq:interface-movement2}
I_t = I_0 - \int_0^t\frac{w'_s(I_s)}{w_s(I_s)} \,ds + B_t ,\qquad t\ge0,
\end{align}
where $(B_t)_{t\ge0}$ is a standard Brownian motion, $w_t=S_tw_0$ and the integral in \eqref{eq:interface-movement2} exists as an improper integral.
\end{prop}

\begin{proof}
We already know by Proposition \ref{prop:single-interface} that $(I_t)_{t\ge0}$ is a continuous Markov process.
Moreover, for all test functions $f\in\calC_c^\infty(\R)$ we have by \eqref{distribution_interface_fixed_time} that
\begin{align}\label{proof:sde-interface_1}
\bE_{\bfu_0}\left[f(I_t)\right] = \int_\R f(x)\,d{\bP}_{\bfu_0}(I_t\le x)
=-\int_\R f'(x)\frac{S_t(\ind_{(I_0,\infty)}w_0)(x)}{S_tw_0(x)} \,dx,\qquad t>0.
\end{align}

1) In the first step, we compute the transition function of the (time-inhomogeneous) Markov process $(I_t)_{t\ge0}$. Define $g:\calU_1\to\R$ such that $g(\bfu)$ is the unique interface point of $\bfu\in\calU_1$. Let $0\le s<t$. For $f\in\calC_c^\infty(\R)$, we have by \eqref{proof:single-interface_1} and the Markov property of the (time-homogeneous!) process $(\bfu_t)_{t\ge0}$ that
\begin{align}
\bE_{\bfu_0}[f(I_t)\,|\,\sigma(I_r:0\le r\le s)]&=\bE_{\bfu_0}[ f\circ g(\bfu_t)\,|\,\sigma(\bfu_r):0\le r\le s]\\
&=\bE_{\bfu_s}[ f\circ g(\bfu_{t-s})]=\bE_{\bfu_s}[f(I_{t-s})].
\end{align}
Applying \eqref{proof:sde-interface_1} with $\bfu_s$, $w_s$ and $t-s$ in place of $\bfu_0$, $w_0$ and $t$ respectively, we obtain
\[\bE_{\bfu_0}[f(I_t)\,|\,\sigma(I_r:0\le r\le s)]=-\int_\R f'(x)\frac{S_{t-s}(\ind_{(I_s,\infty)}w_s)(x)}{S_{t-s}w_s(x)}\,dx.\]
Writing $P_{s,t}(a;dy):=\p_{\bfu_0}(I_t\in dy\,|\,I_s=a)$ for the transition function of the 
Markov process $(I_t)_{t\ge0}$, we have thus shown that it acts on test functions $f\in\calC_c^\infty(\R)$ as
\[P_{s,t}f(a)=\bE_{\bfu_0}[f(I_t)]\,|\,I_s=a)
=-\int_\R f'(x)\frac{S_{t-s}(\1_{(a,\infty)}S_sw_0)(x)}{S_tw_{0}(x)}\,dx,\qquad 0\le s<t.\] 

2) Now we can compute the generator of $(I_t)_{t\ge0}$: For $f\in\calC_c^\infty(\R)$ and $0< s<t$, we have
\begin{align}\label{proof:sde-interface_2}
&\partial_tP_{s,t}f(a) = -\partial_t\int_\R f'(x)\frac{S_{t-s}(\1_{(a,\infty)}w_s)(x)}{w_{t}(x)}\,dx\\
&=-\int_\R \frac{f'(x)}{w_t(x)}\,\partial_tS_{t-s}(\1_{(a,\infty)}w_s)(x)\,dx + \int_\R \frac{f'(x)S_{t-s}(\1_{(a,\infty)}w_s)(x)}{w_t(x)^2}\,\partial_tw_t(x).
\end{align}
We observe that $\partial_t w_t(x)=\frac{1}{2}w_t''(x)$ and
\begin{align}
&\partial_tS_{t-s}(\1_{(a,\infty)}w_s)(x)
=\int_a^\infty\partial_t p_{t-s}(x-y)w_s(y)\,dy=\frac{1}{2}\int_a^\infty p_{t-s}''(x-y)w_s(y)\,dy\\
&=\frac{1}{2}\left(p'_{t-s}(x-a)w_s(a) + p_{t-s}(x-a)w_s'(a) + \int_a^\infty p_{t-s}(x-y)w_s''(y)\,dy\right),
\end{align}
where we used integration by parts for the last equality. (Note that all appearing derivatives exist because $s>0$.) 
Plugging this into \eqref{proof:sde-interface_2} and again integrating by parts, we obtain
\begin{align}
&\partial_tP_{s,t}f(a)\\
&=-\frac{1}{2}w_s(a)\int_\R \frac{f'(x)}{w_t(x)}\,p'_{t-s}(x-a)\,dx - \frac{1}{2}w_s'(a)\int_\R \frac{f'(x)}{w_t(x)} p_{t-s}(x-a)\,dx\\
&\qquad - \frac{1}{2} \int_\R  \frac{f'(x)}{w_t(x)}\int_a^\infty p_{t-s}(x-y)w_s''(y)\,dy\,dx + \frac{1}{2}\int_\R \frac{f'(x)S_{t-s}(\1_{(a,\infty)}w_s)(x)}{w_t(x)^2}\,w''_t(x)\\
&=\frac{1}{2}w_s(a)\int_\R \frac{f''(x)w_t(x)-f'(x)w_t'(x)}{w_t(x)^2}\,p_{t-s}(x-a)\,dx - \frac{1}{2}w_s'(a)\int_\R \frac{f'(x)}{w_t(x)} p_{t-s}(x-a)\,dx\\
&\qquad - \frac{1}{2} \int_\R  \frac{f'(x)}{w_t(x)}\int_a^\infty p_{t-s}(x-y)w_s''(y)\,dy\,dx + \frac{1}{2}\int_\R \frac{f'(x)S_{t-s}(\1_{(a,\infty)}w_s)(x)}{w_t(x)^2}\,w''_t(x).
\end{align}

Letting $t\downarrow s>0$, we obtain
\begin{align}\label{proof:sde-interface_3}\begin{aligned}
\partial_t P_{s,t}f(a)|_{t=s}&=\frac{1}{2}w_s(a)\frac{f''(a)w_s(a)-f'(a)w_s'(a)}{w_s(a)^2}-\frac{1}{2}w_s'(a)\frac{f'(a)}{w_s(a)}\\
&\qquad-\frac{1}{2}\int_a^\infty \frac{f'(y)}{w_s(y)}w_s''(y)\,dy+\frac{1}{2}\int_a^\infty \frac{f'(y)w_s(y)}{w_s(y)^2}w_s''(y)\,dy\\
&=\frac{1}{2}f''(a)-\frac{w_s'(a)}{w_s(a)}f'(a)\\
&=L_sf(a),
\end{aligned}\end{align}
where the time-dependent generator $L_s$ is defined as
\[L_sf(a):=\frac{1}{2}f''(a) - \frac{w_s'(a)}{w_s(a)}\,f'(a),\qquad a\in\R,\,s>0.\]
(Note that we cannot let $s\downarrow0$ here without imposing stronger conditions on $w_0$, see step 3) below.)
By Chapman-Kolmogorov, we then have also the forward equation
\[\partial_t P_{s,t}f(a)=P_{s,t}(L_tf)(a),\qquad 0\le s<t.\]
Consequently,
\[P_{s,t}f(a)-f(a)=\int_{s}^tP_{s,r}L_rf(a)\,dr,
\qquad 0< s<t.\]
This implies that 
\begin{equation}\label{proof:sde-interface_4}
M_t^\ssup{s}(f):=f(I_t)-f(I_s)-\int_s^t L_rf(I_r)\,dr,\qquad t\ge s
\end{equation}
is a martingale under $\p_{\bfu_0}$, for each $s>0$. 
That is, $(I_t)_{t\ge s}$ satisfies the martingale problem for the generator $L_t$ on $\cC_{[s,\infty)}(\R)$.

3) In order to finish the proof, we now assume first that $w_0$ is smooth and strictly positive.  
Note that under this assumption the drift $(t,x)\mapsto \frac{w_t'(x)}{w_t(x)}$ is continuous (in particular, locally bounded) on $\R^+\times\R$, and \eqref{proof:sde-interface_3} 
and \eqref{proof:sde-interface_4} can be extended to $s=0$. 
Then by standard theory (see e.g.\ \cite[Ch.\ 5.3]{EK86}) or \cite[Ch. 5.4]{KS98} 
\[B_t := I_t - I_0 + \int_0^t\frac{w'_r(I_r)}{w_r(I_r)} \,dr ,\qquad t\ge0\] 
is a Brownian motion, and
$(I_t)_{t\ge 0}$ is the unique weak solution to equation \eqref{eq:interface-movement2}.
Note that in this case, the integral in the previous display exists as a proper integral since the integrand is locally bounded. 
Weak uniqueness for equation \eqref{eq:interface-movement2} follows from local boundedness of the drift coefficient, see e.g. \cite[Cor.\ 10.1.2]{SV79}.

4) Now we remove the additional smoothness and strict positivity assumptions on $w_0$ and require only that $\bfu_0\in\cU_1$. 
Then for each fixed $\delta>0$, the drift term $(t,x)\mapsto \frac{w_t'(x)}{w_t(x)}$ is continuous (thus locally bounded) on $[\delta,\infty)\times\R$.
Applying the arguments of the previous step on the time interval $[\delta,\infty)$ instead of $\R^+$ shows that 
the process $\left(I_t - I_\delta + \int_\delta^t\frac{w'_r(I_r)}{w_r(I_r)} \,dr\right)_{t\ge \delta}$
is a Brownian motion starting from $0$ at time $\delta$. Thus for each $\delta>0$ we have 
\[\left(I_t - I_\delta + \int_\delta^t\frac{w'_r(I_r)}{w_r(I_r)} \,dr\right)_{t\ge \delta}\overset{d}=\left(B_t-B_\delta\right)_{t\ge\delta},\]
where $(B_t)_{t\ge0}$ denotes a standard Brownian motion. Now we can use the right-continuity of $(I_t)_{t\ge0}$ at $t=0$ to conclude that the integral $\int_\delta^t\frac{w'_r(I_r)}{w_r(I_r)} \,dr$ converges as $\delta\downarrow0$ 
and that \eqref{eq:interface-movement2} holds, with the integral interpreted in the improper sense. Also, uniqueness in law on $\cC_{[\delta,\infty)}(\R)$ again follows from \cite[Cor.\ 10.1.2]{SV79},
for each $\delta>0$, which by right-continuity can be extended to uniqueness in law on $\cC_{[0,\infty)}(\R)$.
\end{proof}

\subsection{Multiple interfaces}\label{ssec:multiple_interfaces}

In this section, we prove Theorem \ref{thm:annihilating-BM}. The proof consists of a series of lemmas and propositions. 
\begin{lemma}\label{lemma:anniBM-generator-eq}
Assume $\bfu_0\in\cU$.
\begin{enumerate}
\item Fix $\epsilon>0$, a closed interval $[a,b]\subset \bR$ and assume that $\cI(\bfu_0)\cap [a,b] = \cI(\bfu_0)\cap [a-\epsilon,b+\epsilon]$. Write $\tilde\bfu_0$ for the version of $\bfu_0$ where interfaces outside of $[a,b]$ have been removed, i.e.\ $\tilde u_0^\ssup{1}+\tilde u_0^\ssup{2}=u_0^\ssup{1}+u_0^\ssup{2}$ and $m(\tilde\bfu_0,x)=m(\bfu_0,x)$ for $x\in [a,b]$, $m(\tilde\bfu_0,x)=m(\bfu_0,a-\epsilon)$ for $x<a$ and $m(\tilde\bfu_0,x)=m(\bfu_0,b+\epsilon)$ for $x>b$.
Then, for any $k \in \bN$, any test function $\phi$ with $\supp(\phi)\subset [a,b]^k$ and $m\in\{1,2\}^k$, we have
\begin{align}
\frac{d}{dt}\left.\left(\bE_{\bfu_0}f_{\phi,m}(\bfu_t)\right)\right|_{t=0}= \frac{d}{dt}\left.\left(\bE_{\tilde\bfu_0}f_{\phi,m}(\bfu_t)\right)\right|_{t=0}.
\end{align}
\item Let $A,B$ be two disjoint closed intervals in $\bR$, $k,k'\in\bN$ and $\phi,\phi'$ with $\supp(\phi)\subset A^k, \supp(\phi')\subset B^{k'}$ and $m\in\{1,2\}^k,m'\in\{1,2\}^{k'}$. Then 
\begin{align}
&\frac{d}{dt}\left.\left(\bE_{\bfu_0}f_{\phi,m}(\bfu_t)f_{\phi',m'}(\bfu_t)\right)\right|_{t=0}	\\
&= \frac{d}{dt}\left.\left(\bE_{\bfu_0}f_{\phi,m}(\bfu_t)\right)\right|_{t=0}f_{\phi',m'}(\bfu_0) + f_{\phi,m}(\bfu_0)\frac{d}{dt}\left.\left(\bE_{\bfu_0}f_{\phi',m'}(\bfu_t)\right)\right|_{t=0}.
\end{align}
\end{enumerate}
\end{lemma}
\begin{proof}
a) The proof uses the moment duality \eqref{eq:moment duality infinite gamma}. 
For each $t>0$, define the event $D_t:=\bigcap_{i=1}^{k}\left\{ \abs{X^\ssup{i}_t -X_0^\ssup{i}}<\epsilon \right\} ,
$
where no Brownian motion ends up too far from its initial position.
Then, since $\bfu_0$ and $\tilde\bfu_0$ locally agree,
\begin{align}
&\bE_{\bfu_0}\big[ f_{\phi,m}(\bfu_t)\big] = \int_{[a,b]^k} \phi(\bfx)\ \bE_{\bfx}\bigg[\sum_{{m'\in\{1,2\}^n}}\bfu_0^\ssup{m'}(\bfX_t)  M^\sse{\infty}_t(\bfX,\delta_m)(m')\bigg] d\bfx\\
&=  \bE_{\tilde\bfu_0}\left[ f_{\phi,m}(\bfu_t)\right] - \int_{[a,b]^k}\!\!\phi(\bfx)\ \bE_{\bfx}\bigg[\sum_{{m'\in\{1,2\}^n}}(\tilde\bfu_0^\ssup{m'}(\bfX_t) -\bfu_0^{(m')}(\bfX_t) ) M^\sse{\infty}_t(\bfX,\delta_m)(m')\ind_{D_t^c}\bigg] d\bfx.
\end{align}
Since the probability of $D_t^c$ converges to 0 faster than $t$ as $t\to0$, we have proven a).

Similar to a) we obtain b) by using the fact that Brownian motions started in $A$ are sufficiently unlikely to meet Brownian motions started in $B$.
\end{proof}

The next lemma shows that starting from an initial condition with 
finitely many interfaces, up to their first `collision time' these interfaces move as independent Brownian motions with drift \eqref{eq:interface-movement2}. 
\begin{lemma}\label{lemma:anniBM-single}
Assume $\bfu_0\in\cU_n$, $n\ge2$. 
Let $\{ (I_t^x )_{t\geq0}\, : \, x \in \cI(\bfu_0)\}$ denote a system of $n$ independent Brownian motions with drift starting in $\cI(\bfu_0)$ and each moving according to \eqref{eq:interface-movement2},
and let $\sigma$ denote their first collision time. Moreover, define
\[d_t:=\inf\{|y-z|:y,z\in\cI(\bfu_s),y\ne z,0\le s\leq t\}\]
and
\begin{equation}\label{eq:tau}
\tau:=\inf\{t>0:d_t=0\}.
\end{equation}
(Observe that $\sigma$ is defined in terms of the Brownian motions with drift, while $\tau$ is defined in terms of the solution $(\bfu_t)_{t\ge0}$ to $\mathrm{cSBM}(-1,\infty)_{\bfu_0}$.) 
Denote by $\hat m$ the 
standard colouring on $[0,\sigma)\times \bR$ induced by $\bfu_0$ and $\{ (I_t^x )_{0\leq t\leq \sigma}\, : \, x \in \cI(\bfu_0)\}$ as defined in the paragraph before Theorem \ref{thm:annihilating-BM}, and
\[ \hat \bfu_t(x) := \left(w_t(x) \1_{\{\hat m(t,x) = 1 \}}, w_t(x) \1_{\{\hat m(t,x) = 2 \}} \right) , \quad   0 \leq t< \sigma,\ x \in \R.\]
Then 
$\big(\tau, (\bfu_t)_{0\leq t< \tau}\big)$ has the same law as $\big(\sigma,(\hat\bfu_t)_{0\leq t< \sigma}\big)$.
\end{lemma}
\begin{proof}
Let $\delta_1:=\frac13d_0$ be one third of the minimal distance between two interface points in $\cI(\bfu_0)$, let $B_i^\ssup1$ be open balls around the interface points $x_i\in\cI(\bfu_0)$ of radius $\delta_1$, $i=1,\ldots,n$, and let 
\[\tau_1:=\inf\Big\{t>0:\cI(\bfu_t)\not\subset\bigcup_{i=1}^n B_i^\ssup1\Big\}.\]
By Lemma \ref{lemma:anniBM-generator-eq} part a), the evolution of $\bfu_t|_{B_i^\ssup1}$ agrees with the evolution of the corresponding single interface process, which by Prop. \ref{prop:sde-interface} is given by a Brownian motion with drift \eqref{eq:interface-movement2}. 
Hence it is possible to couple the interface position in $B_i^\ssup1$ with $I_t^{x_i}$. We can do this for all interface points simultaneously, and by Lemma \ref{lemma:anniBM-generator-eq} part b) the different motions are independent. Hence we can couple the solution $\bfu_t$ and the system of independent Brownian motions with drift so that 
\begin{equation}\label{eq:coupling_interface}
\cI(\bfu_t)=\{ I_t^x \, : \, x \in \cI(\bfu_0)\}
\end{equation}
and $\bfu_t = \hat\bfu_t$ for all $0\leq t<\tau_1$ a.s. 

We can repeat the argument starting from $\bfu_{\tau_1}$ up to the time 
\[\tau_2:=\inf\Big\{t>\tau_1:\cI(\bfu_t)\not\subset\bigcup_{i=1}^n B_i^\ssup{2}\Big\}\]
by looking at new balls $B_i^\ssup{2}$ with radius $\delta_2:=\frac 13 d_{\tau_1}$. This allows us to extend the coupling up to time $\tau_2$.

Iterating, we obtain a sequence $0< \tau_1 < \tau_2 < \cdots$ as well as a random sequence $\delta_k>0$. Let $\tau_\infty:=\lim_{k\to\infty}\tau_k\in(0,\infty]$. Up to time $\tau_\infty$, the coupling  of $\bfu_t$ with $\hat\bfu_t$ is valid. Furthermore, the first collision time $\sigma$ of the system $\{ (I_t^x )_{t\geq0}\, : \, x \in \cI(\bfu_0)\}$ is clearly bigger than $\tau_k$ for any $k$ and hence $\sigma\geq \tau_\infty$. On the event $\{\tau_\infty=\infty\}$, we thus have $\sigma=\infty$ and $\bfu_t=\hat\bfu_t$ for all $t\ge0$, and the assertion is proved. On $\{\tau_\infty<\infty\}$, $\tau_{k+1}-\tau_{k}$ converges to 0. But this difference is the time it takes one of the Brownian motions with drift to leave the ball of radius $\delta_k$, hence $\delta_k$ must converge to 0 as well. 
Note that here we use that \eqref{eq:interface-movement2} has a continuous global solution.
Therefore 
\[d_{\tau_\infty}= \inf\left\{|y-z| : y,z\in\cI(\bfu_t),y\neq z, 0\leq t\le\tau_\infty\right\}=0  \]
and $\sigma=\tau_\infty=\tau$. 
\end{proof}

Now we prove a version of Theorem \ref{thm:annihilating-BM} for initial conditions with finitely many interfaces. We show that in this case the solution $(\bfu_t)_{t\ge0}$ to $\mathrm{cSBM}(-1,\infty)$ is described in law by a finite regular annihilating system of Brownian motions with drift \eqref{eq:interface-movement2}. In the finite case, the existence of such a system is straightforward.

\begin{prop}\label{thm:annihilating-BM-finite}
Assume that $\bfu_0\in \cU_n$ for some $n\geq 2$.  
Let $\{ (I_t^x )_{t\geq0}\, : \, x \in \cI(\bfu_0)\}$ denote a regular annihilating system
starting from $\cI(\bfu_0)$ such that each coordinate independently follows the dynamics \eqref{eq:interface-movement2}
up to the first collision with another motion, upon which both motions annihilate. 
Denote by $(\hat \bfu_t)_{t\ge0}$ the standard element of $\calC_{[0,\infty)}(\cU)$ induced by $\bfu_0$ and $\{ (I_t^x )_{t \geq 0}\, : \, x \in \cI(\bfu_0)\}$, as defined in \eqref{eq:standard_element}. 
Then we have
\[(\bfu_t)_{t \geq 0}\overset{d}=(\hat\bfu_t)_{t \geq 0} \qquad\text{on }\cC_{[0,\infty)}(\calM_b(\R)^2).\]
\end{prop}

\begin{proof}
Define $\tau$ as in \eqref{eq:tau}, and let 
\[\tau'_0:=0,\qquad \tau'_{k+1}:=\inf\{t>\tau_k':|\cI(\bfu_t)|<|\cI(\bfu_{\tau_k'})|\},\quad k\ge1\]
denote successive jump times of the `interface counting process' $\left(|\cI(\bfu_t)|\right)_{t\ge0}$. 
Since the number of interfaces is non-increasing by Theorem \ref{thm:interfacecount-decreasing}, we have $|\cI(\bfu_t)|=|\cI(\bfu_{\tau_k'})|$ for $t\in[\tau'_k,\tau'_{k+1})$.
Moreover, since by Lemma \ref{lemma:anniBM-single} the evolution of the set $\cI(\bfu_t)$ strictly before time $\tau$ is described by a system of $n$ independent Brownian motions with drift,
we clearly have $\tau'_1\ge\tau$.
We now show that in fact $\tau'_1=\tau$ a.s.\ under $\p_{\bfu_0}$
and that the transition from $\bfu_{\tau-}$ to $\bfu_\tau$ is described in law by a regular annihilating system as in the statement of this proposition. 

Let $\sigma$ be the first annihilation time in the regular annihilating system $\{ (I_t^x )_{t\geq0}\, : \, x \in \cI(\bfu_0)\}$, which coincides in law with the first collision time of $n$ independent Brownian motions with drift \eqref{eq:interface-movement2}.
By (the proof of) Lemma \ref{lemma:anniBM-single}, we know that we can couple the processes $(\bfu_t)_t$ and $(\hat\bfu_t)_t$ on a common probability space such that
\begin{equation}\label{proof:multiple_1}
\tau=\sigma\qquad\text{and}\quad(\bfu_t)_{0\leq t<\tau} = (\hat\bfu_t)_{0\leq t<\sigma} \qquad \text{a.s.} 
\end{equation}
On the event $\{\sigma=\infty\}$, we have nothing more to show, thus suppose now that $\sigma<\infty$ with positive probability.
But on $\{\sigma<\infty\}$ we can use the continuity of the paths of both processes $(\bfu_t)_t$ and $(\hat\bfu_t)_t$ in order to see that $\bfu_\tau=\hat\bfu_\sigma$ a.s.
Since $\hat\bfu_\sigma\in\cU_{n-2}$ by the properties of the regular annihilating system, we conclude in particular that $\tau'_1=\tau$, and we have shown that
\begin{equation}\label{proof:multiple_2}
\big(\tau_1', (\bfu_t)_{0\le t\le\tau_1'}\big)\overset{d}=\big(\sigma, (\hat\bfu_t)_{0\le t\le\sigma}\big).
\end{equation}

By the strong Markov property (see Theorem\ \ref{thm:main1}),
it follows immediately that \eqref{proof:multiple_2} determines the evolution of the process $(\bfu_t)_{t\ge0}$ on any random time interval $[\tau_k',\tau_{k+1}']$, $k\ge0$.
Thus our assertion is proved.
\end{proof}

Before turning to the proof of Thm.\ \ref{thm:annihilating-BM}, we need to establish the existence of a countable regular annihilating system of Brownian motions with drift. Due to the non-monotonicity, this is not trivial.

\begin{lemma}\label{lem:existence-aBM-infinite}
Let $\bfx\subset\bR$ be without accumulation points, let $w_0$ be bounded 
and strictly positive almost everywhere, and let $w_t=S_tw_0$.
Then there exists a regular annihilating system $(\bfI_t)_{t\ge0}$ with $\bfI_0=\bfx$ such that each motion in the system independently follows the law of the SDE \eqref{eq:interface-movement2} up to its annihilation time.
\end{lemma}
\begin{proof}
The existence for $|\bfx|$ finite is straightforward. However, adding more motions is a 
non-monotone process due to the annihilation mechanism, which makes the 
construction of an infinite system delicate. From now on we assume that 
$\bfx$ is neither bounded from above nor below. It will be clear from the 
argument how to modify the proof when $\bfx$ is unbounded in one 
direction only.

Let $\bfx_n:=\bfx\cap[-n,n]$, and let $\bfY_t$ and $\bfY^\sse{n}_t$ be 
the corresponding systems of \emph{coalescing} Brownian motions with drift \eqref{eq:interface-movement2}, starting from
$\bfY_0=\bfx$ and $\bfY_0^\sse{n}=\bfx_n$ respectively, and constructed on the same 
probability space so that $\bfY^\sse{1}\subset 
\bfY^\sse{2}\subset\dots\subset\bfY$. This monotonicity also directly 
implies the existence of the infinite system. 
Denote by $Y^x$ the motion in $\bfY$ started in $x\in\bfx$.
For $y\in\bfY^\sse{n}_t$, 
let $C_n(t,y)$ be the total number of motions which have coalesced in 
the path arriving in $y$ at time $t$, i.e. $C_n(t,y)=\abs{\{x\in\bfx_n: Y^x_t=y\}}$. Consider now the annihilating version $\bfI^\sse{n}$ of $\bfY^\sse{n}$, with $I^{\sse{n},x}_t=Y^x_t$ until the first collision time with another motion $I^{\sse{n},x'}$, upon which both are moved to the cemetery state $\dagger$. By simple counting, one sees that the killing time of $I^{\sse{n},x}$ is given by 
$\tau^\sse{n}_x=\inf\{t\geq 0:C_n(t,Y_t^x)\text{ is even}\}$.
Therefore, $\bfI^\sse{n}$ is given by 
$\bfY^\sse{n}$ restricted to the points where $C_n$ is odd, that is
\begin{align}
I^{\sse{n},x}_t=\begin{cases}
Y_t^x,\quad&C_n(t,Y_t^x)\text{ is odd};\\
\dagger,&C_n(t,Y_t^x)\text{ is even}.
\end{cases}
\end{align}
Note that for fixed $t$ adding more motions can change the parity of the counting, and hence $C_n(t,Y_t^x)\neq C_{n'}(t,Y_t^x)$ for many $x\in\bfx_n$, $n'>n$. However, the map $n\mapsto C_n(t,y)$ is increasing, and hence converges if and only if it is bounded.

We conclude that to obtain the infinite system of annihilating Brownian motions it suffices 
to show that the increasing map $n\mapsto C_n(t,y)$ remains bounded for 
any $t$ and $y$.

Fix $a\in\bR$, and for $x\in\bfx$ let 
$\bfu_0^\sse{x}:=(w_0\ind_{\cdot<x},w_0\ind_{\cdot>x})$. Then by \eqref{distribution_interface_fixed_time}, we have
\begin{align}\label{eq:move-far}
\bP(Y^x_t>a) = \frac{\bE_a[ w_0(X_t)\ind_{X_t<x}]}{w_t(a)} 
\leq \frac{\norm{w_0}_\infty}{w_t(a)}\,\bP_a(X_t<x),
\end{align}
with $(X_t)_{t\ge0}$ denoting a standard Brownian motion. Let $z_n \subset \bfx$ be a decreasing sequence with $\abs{z_n}\geq n$. 
By \eqref{eq:move-far},
\[ \sum_{n\in\bN} \bP(Y^{z_n}_t>a) \leq \frac{\norm{w_0}_\infty}{w_t(a)} 
\sum_{n\in\bN}\bP_a(X_t\leq -n)<\infty. \]
Hence the Borel-Cantelli-Lemma implies that almost surely, only a finite number of the 
$Y^{z_n}_t$ are to the right of $a$. In particular there is some $k$ so 
that
$Y_t^{z_k}<a$, which in turn implies that all motions started to the 
left of $z_k$ also end up to the left of $a$ by the coalescence property.

Repeating the same argument from the right for an arbitrary $b>a$ shows 
that there is also some $\tilde{z}_{\tilde{k}}$ so that motions started 
to the right of this point do not end up to the left of $b$. Together 
this implies that $C_n(y,t)$ remains bounded for any $y\in(a,b)$. Since 
$a$ and $b$ were arbitrary this then holds for any $y$, and hence $\bfI$ is 
obtained from $\bfY$ via reduction to the points with odd counting number.

The fact that the annihilating system $\bfI$ is regular follows directly from this construction and the properties of the coalescing system $\bfY$ of Brownian motions with drift.
\end{proof}

Now we can finally prove the full version of Theorem \ref{thm:annihilating-BM}, admitting initial conditions $\bfu_0\in\cU$ with infinitely many interfaces as long as they do not accumulate.
\begin{proof}[Proof of Theorem \ref{thm:annihilating-BM}]
Choose a sequence $a_n\uparrow \infty$ so that no points in $\cI(\bfu_0)$ lie on the boundary of the intervals $[-a_n,a_n]$. Let $\tilde\bfu_0^\sse{n}$ be the version of $\bfu_0$ where all interface points outside $[-a_n,a_n]$ have been removed, as in Lemma \ref{lemma:anniBM-generator-eq}. 
Note that $\tilde\bfu_0^\sse{n}\in\bigcup_{m\in\N_0}\cU_m$ for each $n\in\N$ and that $\tilde\bfu_0^\sse{n}\to \bfu_0$ as $n\to\infty$ in the topology of $\calM_\tem(\R)^2$. 
Also, choosing $K>0$ such that $\bfu_0\in\calM_K$, we have $\tilde\bfu_0^\sse{n}\in\calM_K$ for all $n\in\N$, thus $\tilde\bfu_0^\sse{n}$ converges to $\bfu_0$ in $\calM_K$. 
Fix $t>0$ and consider a function $f_{\phi,m}$ as in \eqref{defn:f_phi,m}, with $\phi(\bfx)=\prod_{i=1}^n\phi_i(x_i)$ and $\phi_i\in\calC_c(\R)$, $i=1,\ldots,n$. 
By Lemma \ref{lemma:semigroup-continuous}, the function $P_tf_{\phi,m}$ is continuous on $\calM_K$, where $(P_t)_{t\ge0}$ denotes the transition semigroup of $(\bfu_t)_{t\ge0}$. Thus we have 
\[\bE_{\tilde\bfu_0^\sse{n}}\left[f_{\phi,m}(\bfu_t)\right]=P_tf_{\phi,m}(\tilde\bfu_0^\sse{n})\to P_tf_{\phi,m}(\bfu_0)=\bE_{\bfu_0}\left[f_{\phi,m}(\bfu_t)\right],\qquad n\to\infty.\]

On the other hand,
by Proposition~\ref{thm:annihilating-BM-finite} the evolution of $(\bfu_t)_{t\ge0}$ under $\p_{\tilde\bfu_0^\sse{n}}$ is described by the standard element $(\hat\bfu_t^\sse{n})_{t\ge0}\in\cC_{[0,\infty)}(\cU)$ which is induced by $\tilde\bfu_0^\sse{n}$ and a finite regular annihilating system of Brownian motions with drift \eqref{eq:interface-movement2} starting from $\cI(\tilde\bfu_0^\sse{n})$. 
In view of (the proof of) Lemma \ref{lem:existence-aBM-infinite}, this system converges as $n\to\infty$ to a corresponding infinite system of such motions starting from $\cI(\bfu_0)$,
and $\hat\bfu_t^\sse{n}\to\hat\bfu_t$ in $\calM_K$.
Thus
\[\bE\left[f_{\phi,m}(\hat\bfu_t)\right]=\lim_{n\to\infty}\bE\left[f_{\phi,m}(\hat\bfu_t^\sse{n})\right]=\lim_{n\to\infty}\bE_{\tilde\bfu_0^\sse{n}}\left[f_{\phi,m}(\bfu_t)\right]=\bE_{\bfu_0}\left[f_{\phi,m}(\bfu_t)\right]=P_tf_{\phi,m}(\bfu_0).
\]
Since the above class of functions $f_{\phi,m}$ is measure-determining, the one-dimensional distributions of $(\bfu_t)_{t\ge0}$ and $(\hat\bfu_t)_{t\ge0}$ coincide and are given by the semigroup $(P_t)_{t\ge0}$. 
Since both are Markov processes, we conclude that $(\bfu_t)_{t\ge0}\overset{d}=(\hat\bfu_t)_{t\ge0}$,
finishing the proof.
\end{proof}

\subsection{General initial configurations}\label{ssec:overlapping}
In this section, we finally deal with general initial conditions and prove Theorem \ref{thm:overlapping}. The main point is to show that the set $\cI(\bfu_t)$ of interface points is discrete for each $t>0$.
We first need some preliminary results.
\begin{lemma}\label{lemma:interval-existence}
Assume $\bfu\in\cU$ with $u^\ssup1+u^\ssup2=1$ and let $x\in\cI(\bfu)$. In any open neighborhood $O$ around $x$ there exists an interval $A$ with 
\begin{align}\label{eq:interval}
u^\ssup1(A)=u^\ssup2(A)=\frac{\abs{A}}{2}>0.
\end{align}
\end{lemma}
\begin{proof}
First we show that there exist non-degenerate intervals $A_1,A_2\subset O$ with $u^\ssup{i}(A_i)\geq\frac12\abs{A_i}$, $i=1,2$. Assume this is false for $u^\ssup1$. Then for all intervals $B$ in $O$ we have $u^\ssup1(B)<u^\ssup2(B)$, which implies that the restriction of $u^\ssup1$ is absolutely continuous with respect to $u^\ssup2$. But by the definition of $\cU$, we have that $u^\ssup1$ and $u^\ssup2$ are mutually singular (``separation of types"), which implies that $u^\ssup1(O)=0$. This however is a contradiction to the assumption that $x\in O$ is an interface point.

Now that we have $A_1$ and $A_2$, let $g$ be the map defined on $[0,1]$ which interpolates linearly between the intervals $g(0):=A_1$ and $g(1):=A_2$. Then the map $g_0(x)=\frac{u^\ssup1(g(x))}{\abs{g(x)}}$ is continuous with $g_0(0)\geq \frac12$ and $g_0(1)\leq \frac12$, hence there is an $x_0$ with $g_0(x_0)=\frac12$. Then $g(x_0)$ satisfies \eqref{eq:interval}.
\end{proof}

\begin{lemma}\label{lemma:half-interval}
Assume $\bfu\in\cU$ with $u^\ssup1+u^\ssup2=1$ and let $A$ be an interval satisfying \eqref{eq:interval}. Then there exists an interval $A'\subset A$ of size $\abs{A'}=\frac12 \abs{A}$ satisfying \eqref{eq:interval}.
\end{lemma}
\begin{proof}
W.l.o.g. assume $A=(a,b)$. Consider the map $[0,1]\ni x\mapsto g(x):=(a+x\frac{b-a}{2},a+(x+1)\frac{b-a}{2})$ as well as $g_0(x)=\frac{u^\ssup1(g(x))}{\abs{g(x)}}$. Since by assumption $u^\ssup1(g(0))+u^\ssup2(g(0))=\frac{b-a}{2}$ assume w.l.o.g. $g_0(0)\geq \frac12$. Since $A$ satisfies \eqref{eq:interval} this implies $g_0(1)\leq \frac12$ and by continuity there is some $x_0$ with $g_0(x_0)=\frac12$ and $g(x_0)$ is the desired interval.
\end{proof}
Denote the overlapping dyadic intervals of length $2^{-n+1}$ by
\begin{align}
A_{j,n}:=\left(j2^{-n},(j+2)2^{-n}\right),\quad j\in\bZ,n\in\bN.
\end{align}
We will be particularly interested in intervals $A_{j,n}$ which satisfy \eqref{eq:interval} in some approximate sense, defined as follows:
\begin{align}\label{eq:good}
\exists (x,y)\subset A_{j,n}: \text{$(x,y)$ satisfies \eqref{eq:interval} and }y-x\geq \frac14 \abs{A_{j,n}}=2^{-(n+1)}.
\end{align}
Note that \eqref{eq:good} implies that $A_{j,n}\cap\cI(\bfu)\neq\emptyset$. However, counting the $A_{j,n}$ which satisfy \eqref{eq:good} does not give a lower bound on the number of interface points, since we may be overcounting due to the overlap of the intervals. For technical reasons, we fix this in a slightly complicated way: We say $A_{j,n}$ is \emph{good} in $[a,b]$ if $A_{j,n}\subset[a,b]$, satisfies \eqref{eq:good} and $A_{j-1,n}$ is not good. Note that the definition is recursive, but since there is a minimal $j_0$ so that $A_{j_0,n}\subset[a,b]$ it is well defined, as $A_{j_0-1,n}$ is always not good. 
Denote by $J_n([a,b])$ the subset of $\bZ$ where $A_{j,n}$ is good in $[a,b]$.
By definition, the intervals $A_{j,n}$, $j\in J_n([a,b])$, are disjoint, and hence $\abs{J_n([a,b])}\leq \abs{\cI(\bfu)\cap[a,b]}$.
But in fact we have much more:
\begin{lemma}\label{lemma:approximate-interface}
Assume $\bfu\in\cU$ with $u^\ssup1+u^\ssup2=1$ and $a<b \in \bR$.
\begin{enumerate}
\item If $x\in \cI(\bfu)\cap(a,b)$, then for any open neighborhood $O$ around $x$ there is some $n\in\bN$ and $j\in J_{n}([a,b])$ so that $A_{j,n}\subset O$.
\item Assume $j \in J_{n}([a,b])$, and let $j'$ be the smallest integer so that $A_{j',n+1}\subset A_{j,n}$. Then $\{j'-1,j',j'+1,j'+2\}\cap J_{n+1}([a,b]) \neq \emptyset$.
\item We have $\abs{J_1([a,b])}\leq \abs{J_2([a,b])} \leq \cdots \leq \abs{\cI(\bfu)\cap(a,b)}\leq\infty$.
\item If $\abs{\cI(\bfu)\cap(a,b)}=\infty$, then $\lim_{n\to\infty}\abs{J_{n}([a,b])}=\infty$.
\end{enumerate}
\end{lemma}
\begin{proof}
a) Let $x\in\cI(\bfu)\cap(a,b)$ and $x\in O\subset[a,b]$ an open neighborhood. Choose $n_0,j_0$ so that $x\in A_{i,n_0}\subset O$ for $i=j_0-1,j_0$.
By Lemma \ref{lemma:interval-existence} there is an interval $A$ contained in $A_{j_0,n_0}$ satisfying \eqref{eq:interval}. 
By repeatedly applying Lemma \ref{lemma:half-interval}, we may w.l.o.g.\ assume that there exists $n\geq n_0$ so that $\abs{A}\in[2^{-n-1},2^{-n})$. Then there is some $j\in\bZ$ with $A\subset A_{j,n}\subset A_{j_0,n_0}$, hence $A_{j,n}$ satisfies \eqref{eq:good}. Thus $A_{j,n}$ or $A_{j-1,n}$ is good and either one is a subset of $A_{j_0-1,n_0}\cup A_{j_0,n_0} \subset O$. 

b) Let $j\in J_n([a,b])$. Then by \eqref{eq:good} there is an interval $A\subset A_{j,n}$ satisfying \eqref{eq:interval} and $\abs{A}\geq \frac14 \abs{A_{j,n}}$. By Lemma \ref{lemma:half-interval} we can find a subinterval $A'\subset A$ with $\abs{A'}=\frac{1}{2}\abs{A}\geq\frac{1}{8}\abs{A_{j,n}}$ and satisfying \eqref{eq:interval}. We distinguish two cases: If $\abs{A'}\le\frac{1}{4}\abs{A_{j,n}}$, then $A'\subset A_{j'+k,n+1}\subset A_{j,n}$ for some $k \in \{0,1,2\}$, hence $A_{j'+k,n+1}$ satisfies \eqref{eq:good}. Hence either $j'+k\in J_{n+1}([a,b])$ or $j'+k-1 \in J_{n+1}([a,b])$. 

In the case where $\abs{A'}>\frac{1}{4}\abs{A_{j,n}}$ we can apply Lemma \ref{lemma:half-interval} again to obtain another interval $A''\subset A'$ satisfying \eqref{eq:interval} with $\frac{1}{4}\abs{A_{j,n}}\ge\abs{A''}=\frac{1}{2}\abs{A'}\geq\frac{1}{8}\abs{A_{j,n}}$, for which the first case holds.

c) The monotonicity is an immediate consequence of b): If $j_1,j_2\in J_n([a,b])$, $j_1\neq j_2$, then there are corresponding $j_1', j_2'$ and furthermore, since $\abs{j_1-j_2}\geq 2$ we have $\abs{j_1'-j_2'}\geq 4$. This implies that the two good $A_{j_1,n}$ and $A_{j_2,n}$ induce two different good intervals in $J_{n+1}([a,b])$. The upper bound is a simple consequence of the fact that each interval satisfying \eqref{eq:interval} contains at least one interface point.

d) Suppose $\abs{\cI(\bfu)\cap(a,b)}\geq K$. Then we find $K$ interface points $x_1,\ldots,x_K \in (a,b)$. Choose $n_0$ large enough and indices $j_1,\ldots,j_K$ so that the minimal distance between two points is at least $2^{-n_0+3}$ and so that $x_i\in A_{j_i,n_0}\subset[a,b]$, $i=1,\ldots,K$. For any $i\in\{1,\ldots,K\}$, by part a) we find $(j'_i,n'_i)$ so that $j'_i\in J_{n'_i}([a,b])$ and $A_{j'_i,n'_i}\subset A_{j_i,n_0}$. Set $n'':=\max(n'_1,\ldots,n'_K)$. By repeatedly applying b) we find $j''_i \in J_{n''}([a,b])$ with $A_{j''_i,n''}\subset A_{j_i-1,n_0}\cup A_{j_i,n_0}$. Note how we cannot guarantee that $A_{j''_i,n''}\subset A_{j_i,n_0}$, since in each application of b) the interval could start slightly to the left of the previous one. But the maximal shift is bounded by $\sum_{\ell=2}^\infty 2^{-\ell} \abs{A_{j_i,n_0}}$, which proves $A_{j''_i,n''}\subset A_{j_i-1,n_0}\cup A_{j_i,n_0}$.

To complete the proof, we use that by our choice of $n_0$ the intervals $A_{j_i,n_0}$ are sufficiently far apart to guarantee that the $j''_i$ are different. Hence $\abs{J_{n''}([a,b])}\geq K$. Together with monotonicity from c), this proves the assertion.
\end{proof}

\begin{prop}\label{prop:from-infty}
Let  $\bfu_0\in\calM_b(\R)^2$, and assume that $w_0:=u_0^\ssup{1}+u_0^\ssup{2}\ne0$.
For any fixed $t>0$ and interval $[a,b]$, we have $|\cI(\bfu_t)\cap [a,b]|<\infty$ almost surely under $\p_{\bfu_0}$. In particular $\cI(\bfu_t)$ has no accumulation points.
\end{prop}

\begin{proof}
Fix $t>0$. Under our assumptions, $w_t=S_tw_0$ is strictly positive and $\bfu_t\in\cU$ a.s. (recall that the separation of types holds by Thm. \ref{thm:main1} b)).
Write $\bfu_t/w_t = ( u^\ssup1_t/w_t, u^\ssup2_t/w_t )$ for the normalized measure induced by the quotient of densities. 
This new measure satisfies the assumptions of Lemma \ref{lemma:approximate-interface} and $\cI(\bfu_t) = \cI(\bfu_t/w_t)$. 
Furthermore, by the duality for second mixed moments (see \eqref{eq:mixed_second_moment}) we obtain
\begin{align}
\bE_{\bfu_0}\big[ \abs{J_n([a,b])}\big] &\leq \sum_{j:A_{j,n}\subset[a,b]} \bE_{\bfu_0}\big[\ind_{\{A_{j,n}\text{ satisfies \eqref{eq:good} w.r.t. }\bfu_t/w_t\} }\big] \\
&\leq \sum_{j:A_{j,n}\subset[a,b]} \bE_{\bfu_0}\Bigg[ \frac{ \int_{A_{j,n}^2} \frac{u_t^\ssup1}{w_t}(x_1)\frac{u_t^\ssup2}{w_t}(x_2)\;d\bfx }{\frac1{64}\abs{A_{j,n}}^2} \Bigg]\\
&\leq \frac{64}{\inf_{x\in[a,b]}w_t(x)^2} \sum_{j:A_{j,n}\subset[a,b]} \sup_{\bfx\in A_{j,n}^2}\bE_{\bfx} \left[ u_0^\ssup1(X_t^\ssup{1})u_0^\ssup2(X_t^\ssup{2}) \ind_{\tau>t} \right]	\\
&\leq 64\frac{\sup_{x\in\bR}w_0(x)^2}{\inf_{x\in[a,b]}w_t(x)^2} \sum_{j:A_{j,n}\subset[a,b]} \bP_{0,2^{-n+1}}(\tau>t ).
\end{align}
The probability that two Brownian motions do not meet up to time $t$ when started $2^{-n+1}$ apart is of order $2^{-n}$, which follows from the reflection principle. 
Hence we have proven that $\sup_{n\in\bN}\bE_{\bfu_0}\left[\abs{J_n([a,b])}\right] <\infty$. By Lemma \ref{lemma:approximate-interface} part c) and monotone convergence,
it follows that $\bE_{\bfu_0}\left[\lim_{n\to\infty}  \abs{J_n([a,b])} \right] < \infty$, whence together with part d) we conclude that $\bP_{\bfu_0}(\abs{\cI(\bfu_t)\cap(a,b)}=\infty)=0$.
\end{proof}

\begin{proof}[Proof of Theorem \ref{thm:overlapping}]
Let $\bfu_0\in\calM_b(\R)^2$ such that $w_0:=u_0^\ssup{1}+u_0^\ssup{2}\ne0$.
Fix $t_0>0$. Then by Thm.\ \ref{thm:main1} we have $\bfu_{t_0}\in\cU$, and by Prop.\ \ref{prop:from-infty} we know that $\cI(\bfu_{t_0})$ has no accumulation point. Thus the Markov property implies that the evolution of $(\bfu_t)_{t\ge t_0}$ is given by Theorem \ref{thm:annihilating-BM} when started from $\bfu_{t_0}$. In particular, we conclude that almost surely, $\cI(\bfu_t)$ has no accumulation point for any $t \ge t_0$. Since $t_0>0$ is arbitrary, it follows that almost surely, $\cI(\bfu_t)$ has no accumulation point for any $t>0$. 
\end{proof}

{\bf Acknowledgments.} 
This project received financial support
by the German Research Foundation (DFG) within
the DFG Priority Programme 1590 `Probabilistic Structures in
Evolution', grants no. BL 1105/4-1 and OR 310/1-1.

 \bibliographystyle{alpha}
 \bibliography{duality}

\end{document}